\documentclass[12pt,reqno]{article}
\usepackage{amssymb,amsfonts,amsmath,amsthm,color,url}
\usepackage{a4,bm}

\newcommand{\Rp}{\R^+}

\newcommand{\sgn}{\mathrm{sgn}}

\newcommand{\vv}[1]{\mathbf{#1}}
\numberwithin{equation}{section}
\begin{document}

 \bibliographystyle{plain}

 \newtheorem{theorem}{Theorem}[section]
 \newtheorem{proposition}{Proposition}[section]
 \newtheorem{theoremnoname}{Theorem \hspace*{-2ex}}
 \renewcommand{\thetheoremnoname}{}
 \newtheorem{lemma}{Lemma}[section]
 \newtheorem{corollary}{Corollary}[section]
 \newtheorem{problem}{Problem}[section]
 \newtheorem{claim}{Claim}[section]
 \newtheorem{conjecture}{Conjecture}[section]
\theoremstyle{definition}
 \newtheorem{definition}{Definition}[section]
\theoremstyle{remark}
 \newtheorem{remark}{Remark}[section]
 \newtheorem{example}{Example}[section]

 \newcommand{\mc}{\mathcal}
 \newcommand{\rar}{\rightarrow}
 \newcommand{\Rar}{\Rightarrow}
 \newcommand{\lar}{\leftarrow}
 \newcommand{\lrar}{\leftrightarrow}
 \newcommand{\Lrar}{\Leftrightarrow}
 \newcommand{\zpz}{\mathbb{Z}/p\mathbb{Z}}
 \newcommand{\mbb}{\mathbb}
 \newcommand{\A}{\mc{A}}
 \newcommand{\B}{\mc{B}}
 \newcommand{\cc}{\mc{C}}
 \newcommand{\D}{\mc{D}}
 \newcommand{\E}{\mc{E}}
 \newcommand{\F}{\mc{F}}
 \newcommand{\G}{\mc{G}}
 \newcommand{\cS}{\mc{S}}
  \newcommand{\ZG}{\Z (G)}
 \newcommand{\FN}{\F_n}
 \newcommand{\I}{\mc{I}}
 \newcommand{\J}{\mc{J}}
 \newcommand{\M}{\mc{M}}
 \newcommand{\nn}{\mc{N}}
 \newcommand{\qq}{\mc{Q}}
 \newcommand{\U}{\mc{U}}
 \newcommand{\X}{\mc{X}}
 \newcommand{\Y}{\mc{Y}}
 \newcommand{\itQ}{\mc{Q}}
 \newcommand{\C}{\mathbb{C}}
 \newcommand{\R}{\mathbb{R}}
 \newcommand{\N}{\mathbb{N}}
 \newcommand{\Q}{\mathbb{Q}}
 \newcommand{\Z}{\mathbb{Z}}
 \newcommand{\ff}{\mathfrak F}
 \newcommand{\fb}{f_{\beta}}
 \newcommand{\fg}{f_{\gamma}}
 \newcommand{\gb}{g_{\beta}}
 \newcommand{\vphi}{\varphi}
 \newcommand{\whXq}{\widehat{X}_q(0)}
 \newcommand{\Xnn}{g_{n,N}}
 \newcommand{\lf}{\left\lfloor}
 \newcommand{\rf}{\right\rfloor}
 \newcommand{\lQx}{L_Q(x)}
 \newcommand{\lQQ}{\frac{\lQx}{Q}}
 \newcommand{\rQx}{R_Q(x)}
 \newcommand{\rQQ}{\frac{\rQx}{Q}}
 \newcommand{\elQ}{\ell_Q(\alpha )}
 \newcommand{\oa}{\overline{a}}
 \newcommand{\oI}{\overline{I}}
 \newcommand{\dx}{\text{\rm d}x}
 \newcommand{\dy}{\text{\rm d}y}
\newcommand{\cH}{{\cal H}}
\newcommand{\diam}{\operatorname{diam}}
\newcommand{\ve}{\varepsilon}
\newcommand{\cT}{\mathcal{T}}
\newcommand{\cA}{\mathcal{A}}

\parskip=2ex

%\title{\Large\bf Sums of reciprocals of $(n\alpha-\gamma)$ modulo~1 and \\ multiplicative Diophantine approximation}
%\title{The distribution of $n\alpha$ modulo one and \\ multiplicative Diophantine approximation}

\title{\Large\bf Sums of reciprocals of fractional parts and \\ multiplicative Diophantine approximation}

\author{
 V.~Beresnevich\thanks{VB:~Research supported by EPSRC grant  EP/J018260/1.}
 \and
 A.~Haynes\thanks{AH:~Research supported by EPSRC grants EP/F027028/1,  EP/J00149X/1 and EP/M023540/1.}
 \and
 S.~Velani\thanks{SV:~Research supported by EPSRC grants EP/E061613/1, EP/F027028/1 and EP/J018260/1.}   ~
}

\date{}

\maketitle

\centerline{\small\it Dedicated  to Wolfgang M. Schmidt on $N^{o} 80$}

\begin{abstract}
\footnotesize
There are two main interrelated goals of this paper. Firstly we investigate the sums
$$
S_N(\alpha,\gamma):=\sum_{n=1}^N\frac{1}{n\|n\alpha-\gamma\|}
$$
and
$$
R_N(\alpha,\gamma):=\sum_{n=1}^N\frac{1}{\|n\alpha-\gamma\|}\,,
$$
where $\alpha$ and $\gamma$ are real parameters and $\|\cdot\|$ is the distance to the nearest integer. Our theorems improve upon previous results of W.\,M.~Schmidt and others, and are (up to constants) best possible.
Related to the above sums, we also obtain upper and lower bounds for the cardinality of
$$
\{1\le n\le N:\|n\alpha-\gamma\|<\ve\} \, ,
$$
valid for all sufficiently large $N$ and all sufficiently small $\ve$.
This first strand of the work is motivated by applications to multiplicative Diophantine approximation, which are also considered. In particular, we obtain complete Khintchine type results for multiplicative simultaneous Diophantine approximation on fibers in $\R^2$. The divergence result is the first of its kind and represents an attempt of developing the concept of ubiquity to the multiplicative setting.
\end{abstract}

\allowdisplaybreaks

\maketitle


\newpage
\tableofcontents
\newpage

\part{\Large Problems and main results}

\bigskip
\bigskip

\section*{Notation}

To simplify notation the Vinogradov symbols $\ll$
and $\gg$ will be used to indicate an inequality with an
unspecified positive multiplicative constant. If $a \ll b$ and $a
\gg b$ we write $ a \asymp b $, and  say that the quantities $a$
and $b$ are \emph{comparable}. For a real number $x$, the quantity $\{x\}$ will denote the fractional part of $x$.  Also, $\lfloor x\rfloor:=x-\{x\}$  denotes the  largest integer not greater than $x$   and  $\lceil x\rceil:=-\lfloor -x\rfloor:=\min\{m\in\Z:m\ge x\}$  denotes  the smallest integer not less than $x$. Given a real number $x$, $\sgn(x)$ will stand for $1$ if $x>0$, $-1$ if $x<0$ and $0$ if $x=0$.
Finally, $\|x\|=\min\{|x-m|:m\in\Z\}$ will denote the distance from $x \in \R $ to the nearest integer, and $\Rp=(0,+\infty)$.

\section{Sums of reciprocals}\label{sec1}

\subsection{Background}\label{intro}

Let $\alpha,\gamma\in\R$ and suppose that
\begin{equation}\label{contfrac.eqn.|na-g|}
\|n\alpha-\gamma\|>0 \quad \forall   \ n\in\N.
\end{equation}
For $N\in\N$ consider the sums
\[
S_N(\alpha,\gamma):=\sum_{n=1}^N\frac{1}{n\|n\alpha-\gamma\|}
\]
and
\[
R_N(\alpha,\gamma):=\sum_{n=1}^N\frac{1}{\|n\alpha-\gamma\|}\,.
\]

\noindent Note that for any real number $x$ we have that $\|x\|$ is the minimum of the fractional parts $\{x\}$ and $\{-x\}$. Hence,  it is legitimate that the sums $S_N(\alpha,\gamma)$ and $R_N(\alpha,\gamma)$ are referred to as the `sums of reciprocals of fractional parts'.

The above sums can be related by the following well known partial summation formula: given sequences    $(a_n) $  and $ (b_n) $ with $ n \in \N $
\begin{equation}\label{psum1}
\sum_{n=1}^Na_nb_n=\sum_{n=1}^N(a_n-a_{n+1})(b_1+\dots+b_n)+a_{N+1}(b_1+\dots+b_N)\,.
\end{equation}
In particular, it is readily seen that
\begin{equation}\label{pg}
S_N(\alpha,\gamma)=\sum_{n=1}^N\frac{R_n(\alpha,\gamma)}{n(n+1)}
+\frac{R_N(\alpha,\gamma)}{N+1}\,.
\end{equation}

Motivated by a wide range of applications, bounds for the above sums have been extensively studied over a long period of time, in particular, in connection with counting lattice points in polygons, problems in the theory of uniform distribution, problems in the metric theory of Diophantine approximation, problems in dynamical systems and problems in electronics engineering (see, for example, \cite{Beck1994}, \cite{Beck2014}, \cite{Behnke}, \cite{Chowla}, \cite{HardyLittlewood1922}, \cite{HardyLittlewood1922b}, \cite{HardyLittlewood1930}, \cite[pp.\,108-110]{Koksma}, \cite{KuipersNiederreiter1974}, \cite{Kruse}, \cite{Saligrama}, \cite{Schmidt1964}, \cite{SinaiUlcigrai2009}, \cite{Walfisz1},\cite{Walfisz2}).  Schmidt has shown in \cite{Schmidt1964} that for any $\gamma\in\R$ and for any $\varepsilon >0$
\begin{equation}\label{schmidt}
(\log N)^2\ll S_N(\alpha,\gamma)\ll (\log N)^{2+\varepsilon},
\end{equation}
for almost all $\alpha\in\R$. In other words, the set of $\alpha$ for which  \eqref{schmidt} fails is a set of Lebesgue measure zero. It should be emphasised that \eqref{schmidt} is actually a simple  consequence of a much more general result, namely Theorem~2 in \cite{Schmidt1964}, established in higher dimensions for sums involving  linear forms of integral polynomials with real coefficients. To a large extent Schmidt's interest in understanding the behavior of sums such as $ S_N(\alpha,\gamma)$ lies in applications to metric Diophantine approximation, more specifically to obtaining Khintchine type theorems.  Our motivation is somewhat similar.

In the homogeneous case, that is to say when  $\gamma=0$,  the inequalities \eqref{schmidt} are known to be true with $\ve=0$  for any badly approximable $\alpha$.  This result was originally proved by Hardy and Littlewood \cite[Lemma~3]{HardyLittlewood1922b}. Today this classical statement can be found as a set exercise in the monograph \cite[Exercise 3.12]{KuipersNiederreiter1974} of Kuipers and Niederreiter. However, there is a downside  in that the set of badly approximable numbers is of Lebesgue measure zero and therefore we are only guaranteed comparability  in \eqref{schmidt} for $\alpha$ in a thin set.

Walfisz \cite[(48\raisebox{-0.6ex}{\tiny II}) on p.\,571]{Walfisz1} proved that for any $\ve>0$ we have that $S_N(\alpha,0)\ll (\log N)^{3+\ve}$ for almost every $\alpha\in\R$. In another paper \cite[p.\,787]{Walfisz2}, he showed that for any $\ve>0$ we have that $R_N(\alpha,0)\ll N(\log N)^{1+\ve}$ for almost every $\alpha \in \R$. This together with  \eqref{pg} implies the right hand side of \eqref{schmidt} was essentially known to  Walfisz.   Indeed, such upper bounds can  be deduced from the even earlier work of Behnke \cite{Behnke} and metric properties of continued fractions appearing in \cite{Khintchine}.

More generally (see  \cite[Exercise 3.12]{KuipersNiederreiter1974}) if $\alpha\in\R$ and $f:\Rp\rar\Rp$ is a nondecreasing function such that
\begin{equation}\label{intro.eqn.approxtype}
\inf_{n\in\N} nf(n)\|n\alpha\|>0,
\end{equation}
%{\color{red} (why $f$ and not $\psi$)}
then
\begin{equation}\label{intro.eqn.S_Nbound1}
S_N(\alpha, 0)\ll (\log N)^2+f(N)+\sum_{1\le n\le N}\frac{f(n)}{n}.
\end{equation}
A simple consequence of the Borel-Cantelli Lemma in probability theory (or equivalently the trivial convergence part of Khintchine's Theorem \cite[Theorem 2.2]{Harman1998} -- see also \S\ref{MDA}),   implies that if the sum
\begin{equation}\label{intro.eqn.Khinsum}
\sum_{n=1}^\infty\frac{1}{nf(n)}
\end{equation}
converges then \eqref{intro.eqn.approxtype}, and so \eqref{intro.eqn.S_Nbound1}, holds for almost all $\alpha$. However, when \eqref{intro.eqn.Khinsum} converges the last term of \eqref{intro.eqn.S_Nbound1} grows at least as fast as $(\log N)^2\log\log N$. This means that even when $\gamma=0$, for almost all $\alpha$ the lower bound of $(\log N)^2$  in  \eqref{schmidt} does not coincide with the upper bound  given by  \eqref{intro.eqn.S_Nbound1}.

The true magnitude of $S_N(\alpha,0)$ for almost every number  was eventually discovered by Kruse \cite[Theorem~6(b)]{Kruse}. He  proved that for almost every $\alpha\in\R$
\begin{equation}\label{Kruse1}
S_N(\alpha,0)\asymp (\log N)^2\,.
\end{equation}
Beyond this `almost sure' result, Kruse \cite[Theorem~1(g)]{Kruse} showed that for any irrational $\alpha$
\begin{align}
&S_N(\alpha,0)\gg\log q_K\cdot \log (N/q_K)+\sum_{n=1}^{K-1}\log q_n\cdot\log a_{n+1}+\sum_{n=1}^{K+1}a_n\,, \label{Kruse3a}\\[2ex]
&S_N(\alpha,0)\ll
  \log q_K\cdot \log (N/q_K)+\sum_{n=1}^{K-1}\log q_n\cdot(1+\log a_{n+1})+\sum_{n=1}^{K+1}a_n\, \label{Kruse3b}
\end{align}
where $[a_0;a_1,a_2,\dots]$ is the continued fraction expansion of $\alpha$ and $K$ is the largest integer such that the denominator $q_K$ of $[a_0;a_1,a_2,\dots,a_K]$ satisfies $q_K\le N$ (see below for further details of continued fractions).    Kruse also provides the following  simplification  of  \eqref{Kruse3b}
\begin{equation}\label{Kruse2}
  S_N(\alpha,0)\ll (\log N)(\log q_K)+\sum_{n=1}^{K+1}a_n\, .
\end{equation}
Observe that the estimate \eqref{Kruse3a} is not always sharp. For example, if $\alpha=[1;1,1,\dots]$, the golden ratio, and $N=q_K$, the lower bound  \eqref{Kruse3a} becomes $S_N(\alpha,0)\gg \log N$ and this  is significantly smaller than the truth  -- see Theorem~\ref{T1} below. In fact, based on this example it is simple to construct  many other real numbers for which \eqref{Kruse3a}  is not optimal.  In short, they correspond to real numbers  that contain sufficiently large blocks of $1$'s in their continued fraction expansion. In this paper we will  rectify this issue. Moreover, we shall  prove that the two sides of \eqref{Kruse2} are  comparable and provide explicit constants.

Turning our attention to $R_N(\alpha,0)$, we have already mentioned Walfisz's result
that $R_N(\alpha,0)\ll N(\log N)^{1+\ve}$ for almost every $\alpha \in \R$.
Beyond this `almost sure' result, Behnke \cite[pp.\,289-290]{Behnke} showed that for any irrational $\alpha$
\begin{equation}\label{Behnke1}
R_N(\alpha,0)\ll N\log N+\sum_{\substack{n=1\\[0.5ex] n\equiv0\bmod{q_k}}}^N\frac{1}{\|n\alpha\|}\,,
\end{equation}
where $K$ is the same as in \eqref{Kruse3a}. The sum appearing on the right hand side of \eqref{Behnke1} may result in substantial spikes and needs to be analysed separately.
Lang \cite[Theorem~2, p.\,37]{LangDAbook} has shown that if, for some increasing function $g:\Rp\to\Rp$, the inequality  $q_{K+1}\ll q_Kg(q_K)$ holds for all sufficiently large $K\in\N$, where $q_K$ are the denominators of the principal convergents of the continued fraction expansion of $\alpha$, then
\begin{equation}\label{lang}
R_N(\alpha,0)\ \ll \ N\log N+Ng(N)\,.
\end{equation}
As is mentioned in \cite[Remark~1, p.\,40]{LangDAbook} both terms on the right of \eqref{lang} are necessary. This should be interpreted in the following sense: there are  functions $g$, irrationals $\alpha$  and arbitrarily large $N$ such that the inequality in \eqref{lang} can be reversed.

Although \eqref{lang} is not optimal for all choices of $\alpha$ and $N$, for badly approximable $\alpha$ the result is precise. Indeed, when $\alpha$ is badly approximable, \eqref{lang} becomes
\begin{equation}\label{lang+}
N\log N \ \ll \ R_N(\alpha,0) \ \ll \ N\log N\,.
\end{equation}
for all $N>1$. The latter was originally obtained by Hardy and Litlewood \cite{HardyLittlewood1922, HardyLittlewood1922b} in connection with counting integer points in certain polygons in $\R^2$. The upper bound of \eqref{lang+} also appears in the work \cite[p.\,546]{Chowla} of Chowla. Recently, L\^e and Vaaler \cite{Vaaler} have investigated a much more general problem that involves the sums $R_N(\alpha,0)$  and their higher dimensional generalisations. In particular, they prove in \cite[Theorem~1.1]{Vaaler} that
\begin{equation}\label{vaaler}
R_N(\alpha,0)\gg N\log N
\end{equation}
for all $N$ irrespective of the properties of the irrational number $\alpha$. Indeed, the implicit constant within $\gg$ can be taken to be $1$ provided $N$ is sufficiently large. In the same paper, L\^e and Vaaler  show that inequality \eqref{vaaler} is best possible for a large class of  $\alpha$. More precisely, they prove in \cite[Theorem~1.3]{Vaaler} that
for every sufficiently large $N$ and every $0<\ve<1$ there exists a subset $ X_{\ve;N}$ of $\alpha\in[0,1]$ of Lebesgue measure $\ge1-\ve$ such that
$$
R_N(\alpha,0)\ll_\ve N\log N    \qquad   \forall    \   \alpha \in X_{\ve;N}  \,.
$$

The above results  for $R_N(\alpha,0)$ can in fact be  derived from  the work of Kruse \cite{Kruse}. More precisely,  inequalities (75) and (76) in \cite{Kruse} state that for any irrational $\alpha$ and $N$ sufficiently large:
\begin{align}
&\sum_{\substack{n=1\\[0.5ex] n\not\equiv0\bmod{q_K}}}^N\frac{1}{\|n\alpha\|}\asymp N\log q_K+ q_Ka_{K+1}\log\frac{a_{K+1}}{\max\{1,a_{K+1}-N/q_K\}}\,,\label{Kruse3c}\\[2ex]
&\sum_{\substack{n=1\\[0.5ex] n\equiv0\bmod{q_K}}}^N\frac{1}{\|n\alpha\|}\asymp q_Ka_{K+1}(1+\log (N/q_K))\, ,\label{Kruse3d}
\end{align}
where $K$ is the same as in \eqref{Kruse3a}.  These can then be combined to give (see formula  (77) in  \cite{Kruse}) the  estimate
\begin{equation}\label{Kruse3e}
  R_N(\alpha,0) \, \asymp  \,  N\log N+q_Ka_{K+1}(1+\log (N/q_K))\, .
\end{equation}
The techniques developed in this paper allow us to re-establish the above estimates of Kruse with fully explicit constants.

\medskip

One of our  principle goals is
to undertake an in-depth investigation of the sums $S_N(\alpha,\gamma)$ and $R_N(\alpha,\gamma)$, especially in the currently fragmented inhomogeneous case  ($\gamma\neq0$).  The intention is to establish (up to constants) best possible upper and lower bounds for the sums in question.  We begin with the homogeneous case.

\subsection{Homogeneous results and corollaries}

In the homogeneous case ($\gamma =0$)  we are  pretty much able to give exact bounds for
$$
S_N(\alpha,0):=\sum\limits_{1\le n\le N}\,\dfrac{1}{n\|n\alpha\|}\qquad\text{and}\qquad
R_N(\alpha,0):=\sum_{n=1}^N\frac{1}{\|n\alpha\|}
$$
that are valid for all irrational $\alpha$. In order to state our results  we require notions from  the theory of continued fractions, which will be recalled in \S\ref{sec.contfrac}. For now let $q_k=q_k(\alpha)$ denote the denominators of the  principal convergents and let $a_k=a_k(\alpha)$ denote the partial quotients of (the continued fraction expansion of) $\alpha$. Also, given $k\in\N$, let
$$
A_k=A_k(\alpha):=\sum_{i=1}^ka_i
$$
denote the sum of the first $k$ partial quotients of $\alpha$. Our first result concerns the sum $S_N(\alpha,0)$.

\begin{theorem}\label{T1}
Let $\alpha\in\R\setminus\Q$, $N\in \N$
and let $K=K(N,\alpha)$ denote the largest integer satisfying $q_K\le N$. Then, for all sufficiently large $N$
\begin{equation}\label{eq12}
\max\Big\{\,\tfrac12(\log N)^2,~A_{K+1}\,\Big\}~\le~ S_N(\alpha,0)~\le~
33(\log N)^2+10A_{K+1}\, .
\end{equation}
\end{theorem}

\medskip

\begin{remark}
The term $A_{K+1}$ appearing in \eqref{eq12} is natural since
$$
\frac{1}{2q_i\|q_i\alpha\|}-1<a_{i+1}<\frac{1}{q_i\|q_i\alpha\|}
$$
(see \eqref{vb7} below) and thus we have that
$$
A_{K+1}\ \asymp \ \sum_{i=1}^{K}\frac{1}{q_i\|q_i\alpha\|}\,.
$$
Note that it is possible to quantify explicitly the meaning of `sufficiently large $N$' in Theorem~\ref{T1}. The upper bound within \eqref{eq12} is a consequence of the following statement for the `wilder' behaving sum $R_N(\alpha,0)$. Note that the sum $R_N(\alpha,0)$ is split into two subsets which are treated separately.
\end{remark}

\begin{theorem}\label{T2}
Let $\alpha\in\R\setminus\Q$, $N\ge q_3$ and let $K=K(N,\alpha)$ be the largest integer satisfying $q_K\le N$. Then
\begin{equation}\label{eq1}
\tfrac{1}{24} N\log q_K-(\tfrac13\log q_2+\tfrac12)N\ \le\hspace*{-4ex}\displaystyle \sum_{\substack{1\le n\le N \\[0.2ex] n\not\equiv0,\,q_{K-1}\, ({\operatorname{mod}{q_K})}}}\hspace*{-4ex}\frac{1}{\|n\alpha\|}\ \le \ 64N\log q_K +2q_3N,
\end{equation}
and
\begin{equation}\label{eq4}
q_{K+1}\log(1+N/q_K)\ \le \hspace*{-3ex}\sum_{\substack{1\le n\le N\\[0.2ex] n\equiv0,\,q_{K-1}\, ({\operatorname{mod}{q_K})}}}\hspace*{-4ex}\frac{1}{\|n\alpha\|}\ \le \ 4q_{K+1}(1+\log(1+N/q_K))\,.
\end{equation}
Furthermore, let  $c>0$.  Then
\begin{equation}\label{eq2}
\sum_{\substack{1\le n\le N\\[0.2ex] n\equiv0,\,q_{K-1}\, ({\operatorname{mod}{q_K})}}}\hspace*{-4ex}\min\left\{cN,\frac{1}{\|n\alpha\|}\right\}\le 12\,N\,(ca_{K+1})^{\frac12}\,.
\end{equation}
\end{theorem}

\medskip

\begin{remark}
Note that the upper and lower bounds in \eqref{eq1} and \eqref{eq4} are comparable and are thus (up to constants) best possible. The absolute constants appearing in the above results (and indeed elsewhere) can be improved.  For the sake of clarity,  during the course of  proving our results,  we do not attempt to obtain  the sharpest possible constants, let alone asymptotic formulae. Instead, we aim to minimize  technical details.
\end{remark}

\medskip

\subsubsection{Corollaries to Theorem~\ref{T1} regarding $S_N(\alpha,0)$}

First of all, we have the following straightforward consequence  of Theorem~\ref{T1}.

\begin{corollary}\label{coro4}
Suppose  $\alpha\in\R\setminus\Q$ satisfies the condition
\begin{equation}\label{cond2}
A_{k+1}= o(k^2)\,.
\end{equation}
Then, for all sufficiently large $N$
\begin{equation}\label{ineq1}
\tfrac1{2}\,(\log N)^2\ \le \  S_N(\alpha,0)  \ \le \ 34\,(\log N)^2  \, .
\end{equation}
\end{corollary}

\begin{proof}
This result follows from \eqref{eq12}, \eqref{cond2}, the fact that $q_K\le N$, and the estimate $K\ll \log q_K$ (see \eqref{zvb} below).
\end{proof}

\bigskip

The set of $\alpha \in \R$ satisfying condition \eqref{cond2} is of full Lebesgue measure. In fact, for any $\ve>0$ the set of $\alpha\in\R$ such that $A_k\le k^{1+\ve}$ for all sufficiently large $k$ is of full Lebesgue measure (see \cite{DiamondVaaler1986}). Thus we have the following result.

\begin{corollary}\label{coro4+}
The upper bound in \eqref{ineq1} holds for almost every real number $\alpha$, and the lower bound holds for all $\alpha$.
\end{corollary}

\medskip

\subsubsection{Corollaries to Theorem~\ref{T2} regarding $R_N(\alpha,0)$  \label{Ty} }

We now discuss various consequences of Theorem~\ref{T2}. By definition, $\log q_K\le \log N$. Therefore, the sum in \eqref{eq1} is always $\ll N\log N$. In fact, we shall see that the sum is actually comparable to $N\log N$ unless $\alpha$ is a Liouville number. Throughout the paper  $\mathfrak{L}$ will denote the set of  Liouville numbers; i.e. $\alpha\in\R\setminus\Q$ such that
\[
\liminf_{n\rar\infty}n^w\|n\alpha\|=0   \qquad \forall   \ w >0  \, .
\]
It follows from the Jarn\'ik-Besicovitch Theorem \cite{Bes34,Ja29} (alternatively see \cite{BDV-06}) that
\begin{equation}\label{Liou}
  \dim \mathfrak{L}=0\, ;
\end{equation}
that is,  $\mathfrak{L}$ has zero Hausdorff dimension.

\begin{corollary}\label{coro1}
Let $\alpha\in\R\setminus\Q$. Then $\alpha \notin \mathfrak{L}$  if and only if for all sufficiently large $N$
\begin{equation}\label{L1}
\sum_{\substack{1\le n\le N \\[0.2ex] n\not\equiv0,\,q_{K-1}\, ({\operatorname{mod}{q_K})}}}\frac{1}{\|n\alpha\|}\ \asymp N\log N \,.
\end{equation}
\end{corollary}

To establish Corollary~\ref{coro1}, we will make use of well known facts about the growth of the denominators $q_k$ associated with Liouville numbers. Recall that the exponent of approximation of $\alpha \in \R$ is defined as
\begin{equation}\label{cors}
w(\alpha):=\sup\{w>0:\|q\alpha\|<q^{-w}\text{ for i.m. }q\in\N\}\,.
\end{equation}
Note that, for an irrational $\alpha$, by definition, $w(\alpha)<\infty$ if and only if $\alpha$ is not Liouville. Also, by Dirichlet's Theorem, $w(\alpha)\ge1$ for all $\alpha$.

\begin{lemma}\label{Liouvillelem1}
Let $\alpha\in\R\setminus\Q$ and let $q_k$ denote the denominators of the principal convergents of $\alpha$. Then
\begin{equation}\label{liou}
w(\alpha)=\limsup_{k\to\infty}\frac{\log q_{k+1}}{\log q_k}\,.
\end{equation}
In particular, $\alpha\not\in\mathfrak{L}$ if and only if $\log q_{k+1}\ll \log q_k$ for all $k$.
\end{lemma}

\begin{proof}
This result is well known, however the proof is short and we give it for completeness. It easily follows from the definition of $w(\alpha)$ that
$$
w(\alpha)=\limsup_{q\to\infty}\frac{\log\|q\alpha\|^{-1}}{\log q}\,.
$$
Recall that the principal convergents of the continued fraction expansion of $\alpha$ are best approximations, that is $\|q_k\alpha\|\le \|q\alpha\|$ whenever $q_k \le q <  q_{k+1}$. Hence
$$
w(\alpha)=\limsup_{k\to\infty}\frac{\log\|q_k\alpha\|^{-1}}{\log q_k}\,.
$$
Since $(2q_{k+1})^{-1}<\|q_k\alpha\|<q_{k+1}^{-1}$, this expression for $w(\alpha)$ immediately implies the required result.
\end{proof}

\bigskip

\noindent{\em Proof of Corollary~\ref{coro1}.}
Assume that $\alpha\not\in\mathfrak{L}$. Then, by Lemma~\ref{Liouvillelem1}, $\log q_K\asymp\log q_{K+1}$. Since, by definition, $q_K\le N<q_{K+1}$, we also have that $\log N\asymp\log q_K$. By Theorem~\ref{T2}, the sum in (\ref{eq1}) is comparable to $N\log q_K\asymp N\log N$. This proves one direction of Corollary~\ref{coro1}.
Now assume that $\alpha\in\mathfrak{L}$. Then, by Lemma~\ref{Liouvillelem1}, there is a sequence of $K_i$ such that $\log q_{K_i+1}/\log q_{K_i}\to \infty$ as $i\to\infty$. Taking $N=N_i:=q_{K_i+1}-1$ implies that the sum in (\ref{eq1}) is $\asymp N_i\log q_{K_i}=o(N_i\log N_i)$ as $i\to\infty$. This shows that (\ref{L1}) does not hold for Liouville numbers and thereby completes the proof.\qed

\bigskip

 Another consequence of (\ref{eq1}) is that the sum in \eqref{L1} can become  $o(N\log N)$, along a subsequence of $N$,  when  $\alpha$ is a Liouville number. Nevertheless, on combining the estimates \eqref{eq1} and \eqref{eq4} we are able to recover  the lower bound \eqref{vaaler} obtained by L\^e and Vaaler in \cite{Vaaler}.  At this point it is worth formally restating  \eqref{vaaler} as a result in its own right.

\begin{corollary}\label{coro2}
Let $\alpha\in\R\setminus\Q$. Then, for all sufficiently large $N$
\begin{equation}\label{zvb2+2}
R_N(\alpha,0) \ \gg \  N\log N\,.
\end{equation}
\end{corollary}

\begin{proof}
If $N\le q_K^2$ then $\log q_K\ge\tfrac12\log N$ and Corollary~\ref{coro2} is a consequence of \eqref{eq1} appearing within Theorem~\ref{T2}. Otherwise $N/q_K>N^{1/2}$,  and it follows  that $\log (1+N/q_K)\ge \tfrac12\log N$. In addition, by the definition of $q_K$, we have that $q_{K+1}>N$ and so in this case Corollary~\ref{coro2} is a consequence of (\ref{eq4}) appearing within Theorem~\ref{T2}.
\end{proof}

\bigskip

The proof of \eqref{zvb2+2} given in \cite{Vaaler} has its basis in harmonic analysis.  However, we will demonstrate in the  next section  that such homogeneous lower bound estimates do not require either the full power of Theorem~\ref{T2} or indeed the harmonic analysis tools utilized in \cite{Vaaler}.

\medskip

\subsubsection{Homogeneous lower bounds via Minkowski's Theorem}\label{linforms}

It turns out that  appropriately exploiting  Minkowski's Convex Body Theorem  \cite[p.71]{CasselsGofN} from the geometry of numbers,  leads to sharper lower bounds for  $ S_N(\alpha,0)$ and $ R_N(\alpha,0)$ than those described above in Corollary \ref{coro4+} and Corollary  \ref{coro2}.  Furthermore, the approach provides lower bounds within  the more general linear forms setting.  To start with,  we prove the following key statement, which only uses Minkowski's theorem  for  convex bodies and partial summation.

\begin{theorem}\label{T3}
Let $A=(\alpha_1,\dots,\alpha_n)$ be any $n$-tuple of real numbers such that $1,\alpha_1,\dots,\alpha_n$ are linearly independent over $\Q$. Let $T_1,\dots,T_n$ be any positive integers such that $T:=T_1\cdots T_n\ge2$ and let $L\ge2$ be a real number. Then
\begin{align}
\nonumber \sum_{\substack{\vv q\in\Z^n\setminus\{\vv0\}\\[0.5ex] |q_j|\le T_j\ (1\le j\le n)}}\! &\min\left\{L,\big\|q_1\alpha_{1}+\dots+q_n\alpha_n\big\|^{-1}\right\} \\[0ex]
&\ge 2T\min\{\log L,\log T\}+(2^{n+1}-2-\log4)T+4\,.\label{vb+102}
\end{align}
\end{theorem}

\begin{proof}
Since both sides of \eqref{vb+102} depend on $L$ continuously, we can assume without loss of generality that $L>2$. Also in the case $T=2$ the left hand side is at least
$2\times\prod_{i=1}^n(2T_i)=2^{n+2}$, while the right hand side is $4\log2+(2^{n+1}-2-\log4)2+4=2^{n+2}$. Thus, \eqref{vb+102} holds for $T=2$ and we can assume without loss of generality that
\begin{equation}\label{vb+123}
\min\{T,L\}>2\,.
\end{equation}
Fix any $b\in\R$ with $1<b<\sqrt2$ such that for some $m\in\Z$,
\begin{equation}\label{m}
b^m=\min\{T,L\}\,.
\end{equation}
Let $m_0\in\N$ satisfy
\begin{equation}\label{m_0}
b^{m_0}\le 2<b^{m_0+1}.
\end{equation}
Since $b<\sqrt2$, we have that $m_0\ge2$. By \eqref{vb+123}, $b^m>2$ and so $2\le m_0<m$.
Next, given an integer $k\ge 1$, let
  $$
  N_{k}:=\Big\{\vv q\in\Z^n\setminus\{\vv 0\}:|q_j|\le T_j\ (1\le j\le n),\ b^{-k-1}< \big\|\sum_{j=1}^n\alpha_{j}q_j\big\|\le b^{-k}\Big\}
  $$
and let $\Phi_k:=\bigcup_{\ell=k}^\infty N_{\ell}$. Obviously the sets $N_\ell$ are disjoint and so
$\#\Phi_k=\sum_{\ell=k}^\infty \#N_{\ell}$, where $\#X$ stands for the cardinality of $X$. Note that, by definition, $\Phi_k$ consists precisely of integer vectors
$\vv q\in\Z^n\setminus\{\vv 0\}$ such that
$|q_j|\le T_j$  \ $(1\le j\le n)$  and  $\big\|\sum_{j=1}^n\alpha_{j}q_j\big\|\le b^{-k}$.
Hence, by Minkowski's  Convex Body Theorem we conclude that
\begin{equation}\label{w1}
\#\Phi_k\ge 2\lfloor b^{-k}T\rfloor\ge 2b^{-k}T-2\quad\text{when }~b^k\le T.
\end{equation}
Furthermore, note that we trivially have that
\begin{equation}\label{w2}
\#\Phi_k= (2T_1+1)\cdots(2T_n+1)-1 \ge 2^{n}T\qquad\text{when}\qquad b^k\le 2\,;
\end{equation}
that is, when $k\le m_0$. Then, since $\#N_k=\#\Phi_k-\#\Phi_{k+1}$, it follows that
\begin{align*}
\sum_{\substack{\vv q\in\Z^n\setminus\{\vv0\}\\[0.5ex] |q_j|\le T_j\ (1\le j\le n)}}\ &\min\Big\{L,\big\|\sum_{j=1}^n\alpha_{j}q_j\big\|^{-1}\Big\}\
\ge
\sum_{k=1}^\infty\sum_{\vv q\in N_k}\ \min\{L,b^{k}\}\\
&\ge
\sum_{k=1}^\infty\sum_{\vv q\in N_k}\ \min\{b^m,b^{k}\}\\
&=
\sum_{k=1}^m\sum_{\vv q\in N_k}\ b^{k}+\sum_{k=m+1}^\infty\sum_{\vv q\in N_k}\ b^{m}\\
&=\sum_{k=1}^m\#N_k b^{k}+\#\Phi_{m+1}\ b^{m}\\
&=\sum_{k=1}^m(\#\Phi_k-\#\Phi_{k+1})b^{k}~+~\#\Phi_{m+1}\ b^{m}\\
&
=\sum_{k=1}^m\#\Phi_kb^{k}-\sum_{k=2}^{m+1}\#\Phi_kb^{k-1}~+~\#\Phi_{m+1}\ b^{m}\\
&= \#\Phi_1b+\sum_{k=2}^m\#\Phi_k(b^{k}-b^{k-1})\\
&=\#\Phi_1b+(b-1)\sum_{k=2}^mb^{k-1}\#\Phi_k\\
&
=\#\Phi_1b+(b-1)\sum_{k=2}^{m_0}b^{k-1}\#\Phi_k+(b-1)\sum_{k=m_0+1}^mb^{k-1}\#\Phi_k\\
&~ \hspace*{-5ex} \stackrel{\eqref{w1}\,\&\,\eqref{w2}}{ \ge} ~~  \!\! 2^nTb+(b-1)2^nT\sum_{k=2}^{m_0}b^{k-1}+(b-1)\sum_{k=m_0+1}^mb^{k-1}(2b^{-k}T-2)\\
%\end{align*}
%\begin{align*}
&
=2^nTb+2^nT(b^{m_0}-b)+2(b-1)b^{-1}\sum_{k=m_0+1}^m(T-b^k)\\
&\stackrel{\eqref{m_0}}{\ge} ~~  2^{n+1}Tb^{-1}+2(b-1)b^{-1}\sum_{k=m_0+1}^m(T-b^k)\\
&=2^{n+1}Tb^{-1}+2(b-1)b^{-1}\Big((m-m_0)T-\frac{b^{m+1}-b^{m_0+1}}{b-1}\Big)\\[1ex]
&=2^{n+1}Tb^{-1}+2(b-1)b^{-1}(m-m_0)T-2b^{m}+2b^{m_0}\\[1ex]
&\stackrel{\eqref{m}}{=} ~~ 2^{n+1}Tb^{-1}+2(b-1)b^{-1}(m-m_0)T-2\min\{L,T\}+2b^{m_0}\,.
\end{align*}
Now since
$$
m=\frac{\min\{\log L,\log T\}}{\log b},\qquad m_0\le \frac{\log 2}{\log b}\qquad\text{and}\qquad
2b^{-1}<b^{m_0}\,,
$$
we have that the above is
$$
\ge 2^{n+1}Tb^{-1}+2T\frac{b-1}{b\log b}\big(\min\{\log L,\log T\}-\log 2\big)-2\min\{L,T\}+4b^{-1}\,.
$$
Since this estimate holds for $b$ arbitrarily close to $1$ and since $\lim\limits_{b\to1^+}\frac{b-1}{b\log b}=1$, we obtain that the above sum is
\begin{align*}
&\ge 2^{n+1}T+2T\big(\min\{\log L,\log T\}-\log 2\big)-2\min\{L,T\}+4\\[1ex]
&= 2T\min\{\log L,\log T\}+(2^{n+1}-\log4)T-2\min\{L,T\}+4\,,
\end{align*}
whence the required estimate immediately follows.
\end{proof}

\bigskip

\begin{remark}
It is worth mentioning that a similar argument can be given for the  sums of products investigated in~\cite{Vaaler}.
\end{remark}

Theorem~\ref{T3} together with  formula \eqref{pg},   yields  the following  lower bound estimates for $R_N(\alpha,0)$ and $S_N(\alpha,0)$, that are sharper  than those given by  Corollary~\ref{coro4+} and Corollary~\ref{coro2}.  We leave the details of the proof to the reader.

\begin{corollary}\label{coro2++}
For any $\alpha\in\R\setminus\Q$ and any integer $N\ge2$
$$
R_N(\alpha,0)  \ \ge \ N\log N+N\log (e/2)+2
$$
and
$$
S_N(\alpha,0)  \ \ge \  \tfrac12(\log N)^2\,.
$$
\end{corollary}

\begin{remark}
Another consequence of Theorem~\ref{T3} is that
\begin{align}
\nonumber \sum_{\substack{\vv q\in\Z^n\setminus\{\vv0\}\\[0.5ex] |q_j|\le T_j\ (1\le j\le n)}}\! \frac{1}{\Pi_+(\vv q)\big\|q_1\alpha_{1}+\dots+q_n\alpha_n\big\|} ~\gg~ (\log T)(\log T_1)\cdots(\log T_n)\,,
\end{align}
where $\alpha_1,\dots,\alpha_n$ and $T_1,\dots,T_n$ are as above in Theorem \ref{T3}  and
$$
\Pi_+(\vv q)=\prod_{q_i\neq 0}|q_i|\,,\qquad\text{where}\quad \vv q=(q_1,\dots,q_n)\,.
$$
We omit the details of the proof since in  \S\ref{linforms2} below we will prove a  more general inhomogeneous statement.
\end{remark}

\medskip

By analogy with the one-dimensional case (see also \cite[Theorem 2]{Schmidt1964}),  it makes sense to make the following upper bound conjecture.

\begin{conjecture}
Let $T_1,\dots,T_n$ be positive integer parameters and $T=T_1\cdots T_n$. Then, as $T\to\infty$, for almost every $(\alpha_1,\dots,\alpha_n)\in\R^n$ we have that
  \begin{align}
\sum_{\substack{\vv q\in\Z^n\setminus\{\vv0\}\\[0.5ex] |q_j|  \, \le T_j\ (1\le j\le n)}}\! \frac{1}{\Pi_+(\vv q)\big\|q_1\alpha_{1}+\dots+q_n\alpha_n\big\|} ~\ll~ (\log T)(\log T_1)\cdots(\log T_n)\,.\label{kron}
\end{align}
\end{conjecture}

\medskip

\begin{remark}
Bounds such as \eqref{kron} are instrumental in the theory of uniform distribution to study the discrepancy of Kronecker sequences $\big(n(\alpha_1,\dots,\alpha_n)\big)_{n\in\N}$,   see for example \cite{Beck1994}, \cite[\S1.4.2]{Drmota-Tichy}, \cite{Hua-Wang} or \cite[\S6]{Korobov}. Note that Lemma~4.4 in \cite{Beck1994} (see also \cite[Lemma~1.95]{Drmota-Tichy}) implies that the left hand side of \eqref{kron} is $\ll (\log T)^{n+2}$.  When $T_1=\dots=T_n$, this is  $\log T$ times bigger than the  conjectured bound of $(\log T)^{n+1}$. It may well be that the proof of Lemma~4.1 in \cite{Beck1994} (see also \cite[Lemma~1.93]{Drmota-Tichy}) can be adapted to prove the conjecture in the symmetric case. It is worth highlighting the fact that compared to the sum in  \eqref{kron} the range of  summation in \cite[Lemma~4.1]{Beck1994} is restricted by the addition condition that
$$
\Pi_+(\vv q) \, \big\|q_1\alpha_{1}+\dots+q_n\alpha_n\big\|  \, \ll \ (\log T)^{20n}  \, .
$$
%??MAYBE WE SHOULD THINK IF WE CAN DO THIS IN SYMMETRIC CASE
\end{remark}

\medskip

\subsubsection{Almost sure behaviour of $R_N(\alpha,0)$  \label{ae} }

Corollary \ref{coro4+} completely describes the almost sure behaviour of the sum $S_N(\alpha,0)$; namely that for almost all $\alpha \in \R $ we have that
\begin{equation*}
\tfrac1{2}\,(\log N)^2\ \le \  S_N(\alpha,0)  \ \le \ 34\,(\log N)^2  \, ,
\end{equation*}
for all sufficiently large $N\in\N$.   In this section, we make use of Theorem \ref{T2} to investigate the almost sure behaviour of the sum $R_N(\alpha,0)$.

\vspace{1ex}

 In order to understand the almost sure behaviour of  inequalities  \eqref{eq4} and \eqref{eq2} associated with Theorem \ref{T2},  we  first recall from \S\ref{intro}, that by the Borel-Cantelli Lemma,  if $f:\N\to \Rp$ is any function such that the sum \eqref{intro.eqn.Khinsum}  converges,
  then for almost all $\alpha\in\R$ one has that $qf(q)\|q\alpha\|\ge 1$ for all sufficiently large $q$. In particular, taking $q=q_K$, where $q_K=q_K(\alpha)$ is as before the sequence of denominators of the principle  convergents  of $\alpha$, we have that $\|q_K\alpha\|\ge (q_Kf(q_K))^{-1}$ for sufficiently large $K$. In addition, we have the following well known inequalities from the theory of continued fractions \cite{Khintchine}:
$$
(2q_{K+1})^{-1}<\|q_K\alpha\|<q_{K+1}^{-1},
$$
see also \eqref{vb7} below. Hence, whenever \eqref{intro.eqn.Khinsum} converges, for almost every $\alpha$ we have that $$q_{K+1}<q_Kf(q_K)$$ for sufficiently large $K$.

On the other hand, the main substance of Khintchine's Theorem, the divergent part (namely, \eqref{neweq1} below with $k=1$),  implies that if $f$ is monotonic and the sum \eqref{intro.eqn.Khinsum} diverges, then for almost all $\alpha$ we have that $qf(q)\|q\alpha\|<1$ for infinitely many $q\in\N$. Since  convergents are best approximations, for almost all $\alpha$ we have that $2q_Kf(q_K)\|q_K\alpha\|<1$ for infinitely many $K$. Thus, in the case $f$  is monotonic and the sum \eqref{intro.eqn.Khinsum} diverges, it follows that for almost all $\alpha$ we have that $$q_{K+1}>q_Kf(q_K)$$ for infinitely many $K$.

The  above observations regarding the size of $q_{K+1}$  together with  \eqref{eq4} of Theorem~\ref{T2} give rise to the following statement.

\begin{corollary}\label{coro3+}
Let  $f:\N\to(0,\infty)$ be any increasing function. If \eqref{intro.eqn.Khinsum} converges, then for almost all $\alpha \in \R $ we have that
$$
\sum_{\substack{1\le n\le N\\[0.2ex] n\equiv0,\,q_{K-1}\, ({\operatorname{mod}{q_K})}}}\hspace*{-4ex}\frac{1}{\|n\alpha\|}\ \ll \ q_Kf(q_{K})\log(1+N/q_K)\ll Nf(N),\
$$
for all sufficiently large $N\in\N$.
On the other hand, if \eqref{intro.eqn.Khinsum} diverges, then for almost all $\alpha$ there are infinitely many $N\in\N$ such that
$$
\sum_{\substack{1\le n\le N\\[0.2ex] n\equiv0,\,q_{K-1}\, ({\operatorname{mod}{q_K})}}}\hspace*{-4ex}\frac{1}{\|n\alpha\|}\ \gg \ Nf(N)\,.
$$
\end{corollary}

\begin{example} \label{eg1}
Let $f(q)=\log q\log\log q$. Then, Corollary~\ref{coro3+} implies that for almost all $\alpha \in \R$
$$
R_N(\alpha,0) \  \gg \ N\,\log N\,\log\log N  \quad \mbox{ for infinitely many $N\in\N$.}
$$
However, if $f(q)=\log q \,(\log\log q)^{1+\ve}$, then Corollary~\ref{coro3+} together with Corollary~\ref{coro1} (and the well known fact that the set of Liouville numbers are a set of measure zero), shows that
for almost all $\alpha \in \R$
$$
R_N(\alpha,0) \ \ll \ N\,\log N\,(\log\log N)^{1+\ve}   \quad \mbox{for all sufficiently large $N\in\N$.}
$$
\end{example}

Finally, we analyze the almost all behavior of the sum appearing in (\ref{eq2}), in which  the possible spikes of $\|n\alpha\|^{-1}$ are ``trimmed''  to be no more than $cN$. Once again,  we appeal to the Borel-Cantelli Lemma, which implies that if \eqref{intro.eqn.Khinsum} converges then, for almost all $\alpha$, we have $q_{K+1}<q_Kf(q_K)$ for sufficiently large $K$. Since $q_{K+1}=a_{K+1}q_K+q_{K-1}\ge a_{K+1}q_K$, we have that $a_{K+1}< f(q_K) \le f(N) $ for large $K$. This observation with $f$ replaced by $\frac{1}{144}  f $,  together with (\ref{eq2}) implies the following statement.

\begin{corollary}\label{coro3}
Let   $f:\N\to\Rp$ be an increasing function such that \eqref{intro.eqn.Khinsum} converges.
Then, for almost all $\alpha \in \R$ and for all sufficiently large $N$ and  $c>0$
\begin{equation}\label{eq3}
\sum_{\substack{1\le n\le N\\[0.2ex] n\equiv0,\,q_{K-1}\, ({\operatorname{mod}{q_K})}}}\min\left\{cN,\frac{1}{\|n\alpha\|}\right\}\le N\,(cf(N))^{1/2}\,.
\end{equation}
\end{corollary}

\medskip

\begin{example}
Taking $f(q)=\log q\,(\log\log q)^{1+ \ve}$ with $\ve>0$ in the above corollary, gives the upper bound estimate
\[  {\rm l.h.s. \ of \ } \eqref{eq3}   \  \leq \  N    \big(   c  \, \log N (\log\log N)^{1+ \ve} \big)^{1/2} \] for almost  all  $\alpha \in \R$, $c > 0 $  and for all sufficiently large $N$.
\end{example}

\begin{example}
Take $f$ as the previous example and $c=\dfrac{\log N}{(\log\log N)^{1+\ve}}$ with $\ve~>~0$.   Then   Corollary \ref{coro3}  together with the upper bound  appearing in \eqref{eq1} of Theorem~\ref{T2}  implies that \begin{equation}\label{zvb2+}
\sum_{\substack{1\le n\le N}}\min\left\{\dfrac{N\log N}{(\log\log N)^{1+\ve}},\frac{1}{\|n\alpha\|}\right\}\ll N\log N
\end{equation} for almost all $\alpha   \in \R$ and for all sufficiently large $N$.
\end{example}

\subsubsection{The homogeneous estimates for algebraic numbers}

In this section we investigate the implication of the estimates obtained above for $R_N(\alpha,0)$ and $S_N(\alpha,0)$  on specific classes of numbers.  In particular, we consider the case that $\alpha$ is an  algebraic irrational.

To begin with, observe that the task at hand is much simplified if we know the continued fraction expansion of the  number $\alpha$ under consideration. For instance, it is well known that
$$
    e = [2; 1, 2, 1, 1, 4, 1, 1, 6, 1, 1, 8, 1, 1, 10, \dots] \,.
$$
Now, for any irrational number we have that
\begin{eqnarray*}
{\rm r.h.s. \ of \ (\ref{eq4})}  \ & \ll  & \ q_{K+1}(1+\log(2N/q_K))\le 2a_{K+1}q_K (1+\log(2N/q_K))   \\[1ex]
&\ll&  N a_{K+1}q_K (1+\log(2N/q_K))/N \ll N a_{K+1}\,
\end{eqnarray*}
and note  that  for  $e$ we also have that $a_k\ll k$.   Hence,  Theorem~\ref{T2} implies that for all sufficiently large $N$
$$
R_N(e,0)   :=  \sum_{1\le n\le N}\frac{1}{\|ne\|}\ \asymp \ N\log N \,.
$$
Similarly, it is readily seen from Corollary~\ref{coro4} that
$$
S_N(e,0)  :=  \sum_{1\le n\le N}\frac{1}{n\|ne\|}\ \asymp \ (\log N)^2 \,.
$$

The sums $R_N(\alpha,0)$ and $S_N(\alpha,0)$ are of course just as easily estimated when $\alpha$ is a quadratic irrational. One just exploits the fact that the   continued fraction of $\alpha$ is periodic.   The  behaviour of
the sums for quadratic irrationals  is similar to that for $\alpha=e$.

The case of algebraic numbers $\alpha$ of degree $\ge3$ is problematic, although some bounds can be obtained using, for instance, Roth's Theorem. Indeed, Roth's Theorem implies that  the  denominators $q_{k} $  of the  convergents of an algebraic number  $\alpha$   satisfy $q_{k+1}< q_k^{1+\ve}$ for every  $\ve>0$  and for all $k$ sufficiently large.
Hence,  Theorem~\ref{T2} implies that if $\alpha$ is an algebraic irrational, then
$$
N \log N  \ \ll   \  R_N(\alpha ,0)  \ \ll \ N^{1+\ve}\,.
$$
Lang's Conjecture \cite{Lang}, dating back to 1965,  implies  that $q_{k+1}\ll q_k(\log q_k)^{1+\ve}$ for any $\ve>0$. Together with Theorem~\ref{T2}, this would in turn imply the  following strengthening of the previous statement: if $\alpha$ is an algebraic irrational, then
$$
N\log N\ \ll \  R_N(\alpha ,0)  \ \ll N(\log N)^{1+\ve} \, .
$$
 Using \eqref{pg}, one can deduce similar inequalities in relation to $S_N(\alpha,0)$ for  algebraic irrational $\alpha$ . Furthermore, it is not unreasonable to believe in the truth of the following statement which,   in view of Corollary  \ref{coro4+},  it implies that $S_N(\alpha,0)$ for algebraic irrationals  behaves in the same way as for almost all irrationals.

\begin{conjecture}\label{conj2}
For any algebraic $\alpha\in\R\setminus\Q$
$$
  S_N(\alpha ,0)  \ \asymp   \ (\log N)^2\,.
$$
\end{conjecture}

\noindent In order to establish Conjecture~\ref{conj2},   it is not necessary to know that all the partial quotients associated with an  algebraic irrational grow relatively slowly. Instead, all that is required  is that the growth ``on average'' is relatively slow. Indeed,
Conjecture~\ref{conj2} is  equivalent to the following statement.

\begin{conjecture}
For any algebraic $\alpha\in\R\setminus\Q$,  there exists a constant $c_\alpha> 0$ such that for all $k\in\N$
$$
A_k(\alpha)\ \le \ c_\alpha k^2\, .
$$
\end{conjecture}

\medskip

%\begin{remark}
%Computational evidence for specific algebraic numbers does support this conjecture -- see the appendix.
%\end{remark}

\subsection{Inhomogeneous results and corollaries}

\subsubsection{Upper bounds}

We concentrate our attention on the sums $S_N(\alpha,\gamma)$. As we have seen from the homogeneous case, these sums behave more predictably than the sums $R_N(\alpha,\gamma)$.

Our first inhomogeneous  result removes the `epsilon' term  in  the upper bound appearing in  (\ref{schmidt}) obtained by Schmidt.

\begin{theorem}\label{T4}
For each $\gamma\in\R$ there exists a set $\mathcal{A}_\gamma\subset\R$ of full Lebesgue measure such that for all $\alpha\in \mathcal{A}_\gamma$ and all sufficiently large $N$
\begin{equation}\label{sum1}
S_N(\alpha,\gamma):=\sum_{1\le n\le N}\frac{1}{n\|n\alpha-\gamma\|}\ \ll\ (\log N)^2 \,.
\end{equation}
\end{theorem}

In view of the lower bound given by (\ref{schmidt}), this theorem  is  best possible up to the implied constant. Furthermore,
the set $\mathcal{A}_\gamma$ cannot be made independent of $\gamma$, as demonstrated by the following result.

\begin{theorem}\label{T5}
For any $\alpha\in\R\setminus\Q$ and any  increasing function $f:\N\rar\Rp$ satisfying $f(N)=o(N)$, there exist continuum many $\gamma\in\R$, satisfying \eqref{contfrac.eqn.|na-g|}, and such that
\begin{equation*}
\limsup_{N\rar\infty}~f(N)^{-1}\left(S_N(\alpha,\gamma)-\max_{1\le n\le N}\frac{1}{n\|n\alpha-\gamma\|}\right)=\infty.
\end{equation*}
\end{theorem}

Despite Theorem~\ref{T5}, the main `bulk' of the sum in (\ref{sum1}) can be estimated on a set $\A$ of $\alpha$'s independent of $\gamma$. The following can be viewed as the inhomogeneous analogue of the upper bound appearing in Corollary \ref{coro4}.

\begin{theorem}\label{T6}
There exists a set $\A\subset \R$ of full Lebesgue measure  such that for any $\alpha\in\A$, any $\gamma\in\R$ and any $c>0$, we have that for all sufficiently large $N$,
\begin{equation*}
\sum_{\substack{1\le n\le N \\[0.5ex] n\|n\alpha-\gamma\|\ge c}}\frac{1}{n\|n\alpha-\gamma\|} \ \ll_{\alpha,\gamma, c} \ (\log N)^2  \,.
\end{equation*}
\end{theorem}

We now move on to describing lower bound estimates.

\subsubsection{Lower bounds}\label{lowerbounds}

All our lower bound results will apply to any irrational number $\alpha$ which is not Liouville, that is $\alpha\not\in\mathfrak{L}$. Recall,  that $\mathfrak{L}$ is a small set in the sense that it has Hausdorff dimension zero (see \eqref{Liou}).

\begin{theorem}\label{T7}
Let $\alpha\in\R\setminus(\mathfrak{L}\cup\Q)$ and let $0<v<1$. There exist constants $r,C_1,$ and $N_0$, depending on $\alpha$ and $v$ only, with $0<r,C_1<1$ and $N_0>1$, and such that, for any $\gamma\in\R$ and all $N>N_0$,
\begin{equation*}
\sum_{\substack{rN<n\le N \\[0.5ex] \|n\alpha-\gamma\|\ge N^{-v}}}\frac{1}{\|n\alpha-\gamma\|} \ \ge \ C_1N\log N  \,.
\end{equation*}
\end{theorem}

Using \eqref{pg} we readily obtain the following result.

\begin{corollary}\label{cor8}
Let $\alpha\in\R\setminus(\mathfrak{L}\cup\Q)$ and let $0<u,v<1$. There exist constants $r,C_2,$ and $N_1$,  depending on $\alpha$, $u$ and $v$ only, with $0<r,C_2<1<N_1$, and such that, for any $\gamma\in\R$ and all $N>N_1$,
$$
\sum_{\substack{N^u\le n\le N \\[0.5ex] n^{v}\|n\alpha-\gamma\|\ge r^v}}\frac{1}{n\|n\alpha-\gamma\|}
\ge C_2(\log N)^2 \, .
$$
\end{corollary}

\begin{proof}
Let $r$ be as in Theorem~\ref{T7} and set
\[\ell_0:=\left\lfloor\frac{u\log N}{\log r}-1\right\rfloor.\]
Then, it follows that
\begin{align*}
\sum_{\substack{N^u\le n\le N \\[0.5ex] n^{v}\|n\alpha-\gamma\|\ge r^v}}\frac{1}{n\|n\alpha-\gamma\|}
&~\ge~ \sum_{\ell=0}^{\ell_0}
~~\sum_{\substack{r^{\ell+1} N<n\le r^{\ell}N \\[0.5ex] \|n\alpha-\gamma\|\ge (r^{\ell}N)^{-v}}}\frac{1}{n\|n\alpha-\gamma\|}\\[1ex]
\end{align*}
\begin{align*}
&~\ge~ \sum_{\ell=0}^{\ell_0}~~\frac{1}{r^\ell N}
\sum_{\substack{r^{\ell+1} N<n\le r^{\ell}N \\[0.5ex] \|n\alpha-\gamma\|\ge (r^{\ell}N)^{-v}}}\frac{1}{\|n\alpha-\gamma\|}\\[0ex]
&~\ge~ \sum_{\ell=0}^{\ell_0}~~\frac{1}{r^\ell N}\cdot C_1r^\ell N\log(r^{\ell} N)
\\[1ex]
&~\ge~ C_1\ell_0\log(r^{\ell_0} N)\ge C_2(\log N)^2,
\end{align*}
provided that $\ell_0\ge 1$ and $rN^u>N_0$. These conditions determine the quantity  $N_1$ appearing in the statement of the corollary.
\end{proof}

\medskip

Corollary~\ref{cor8} complements the upper bounds associated with Theorems~\ref{T4} and \ref{T6}. It also implies the following statement.

\begin{corollary}\label{cor8+}
Let $\alpha\in\R\setminus(\mathfrak{L}\cup\Q)$ and $c>0$. Then,  for all sufficiently large $N$ and any $\gamma\in\R$,  we have that
$$
S_N(\alpha,\gamma)\ge \sum_{\substack{1\le n\le N \\[0.5ex] n\|n\alpha-\gamma\|\ge c}}\frac{1}{n\|n\alpha-\gamma\|}
\gg (\log N)^2 \, .
$$
\end{corollary}

\vspace*{1ex}

\begin{remark}
Recall, that Schmidt's inhomogeneous result \eqref{schmidt} is true on a set of $\alpha \in \R$ of full measure, which may depend on $\gamma$. Corollary~\ref{cor8+} shows that the lower bound holds for all $\alpha \in \R$ except possibly for Liouville numbers  - an explicit set of Hausdorff dimension zero that is independent of $\gamma$.
\end{remark}

\vspace*{2ex}

\begin{remark}
We do not  know if the condition that $\alpha$ is not a Liouville number is necessary
for the conclusion of Corollary~\ref{cor8+} to hold for all $\gamma$. However
the next result, which can be viewed as the inhomogeneous analogue of Corollary \ref{coro2}, demonstrates the necessity of the condition  that $\alpha$ is not a Liouville number in the statement of Theorem~\ref{T7}.
\end{remark}

\begin{theorem}\label{T8}
Let $\alpha\in\R\setminus\Q$. Then, $\alpha\not\in\mathfrak{L}$ if and only if for any $\gamma\in\R$,
$$
R_N(\alpha,\gamma):=\sum_{1\le n\le N}\frac{1}{\|n\alpha-\gamma\|}\gg N\log N\quad\text{for}~ N\ge 2. \,
$$
\end{theorem}

\subsection{Linear forms revisited}\label{linforms2}

In this section we deduce a few further corollaries via the results of \S\ref{lowerbounds}. These provide inhomogeneous generalisations of the results of \S\ref{linforms} and also give an alternative proof of Theorem~\ref{T3} and its corollaries, albeit with weaker constants. The constants are not explicitly given but can be computed from the proofs given below and in
\S\ref{lowerbounds}.

\begin{theorem}\label{T9}
Let $A=(\alpha_1,\dots,\alpha_n)$ be any $n$-tuple of real numbers such that $\alpha_1$ is irrational but not a  Liouville number. Then, for any positive integers $T_1,\dots,T_n$ such that $T_1$ is sufficiently large, and for any $\gamma\in\R$,
\begin{equation}\label{cc5}
\sum_{\substack{\vv q\in\Z^n\\[0.5ex] 1\le q_j\le T_j\ (1\le j\le n)}}\! \frac{1}{\,\big\|q_1\alpha_{1}+\dots+q_n\alpha_n-\gamma\big\|\,}
~\gg~ T_1\cdots T_n\log T_1\,.
\end{equation}
\end{theorem}

\begin{proof}
By Theorem~\ref{T7}, for any $T_1>T_0$ and any $\gamma_1\in\R$,
\begin{equation*}
\sum_{1\le q_1\le T_1 }\frac{1}{\|q_1\alpha_1-\gamma_1\|} \ \ge \ C_1T_1\log T_1\,,
\end{equation*}
where $T_0$ and $C_1$ do not depend on $\gamma_1$. Applying this inequality with
$\gamma_1=\gamma-(q_2\alpha_{2}+\dots+q_n\alpha_n)$,  and summing over  $1\le q_i\le T_i$ ($2\le i\le n$), we obtain the desired statement.
\end{proof}

\begin{theorem}\label{T10}
Let $A=(\alpha_1,\dots,\alpha_n)$ be any $n$-tuple of real numbers such that $\alpha_1$ is irrational but not a  Liouville number. Then for any positive integers $T_1,\dots,T_n$ such that $T_1$ is sufficiently large, and for any $\gamma\in\R$,
\begin{equation}\label{cc7}
\sum_{\substack{\vv q\in\Z^n\\[0.5ex] 1\le q_j\le T_j\ (1\le j\le n)}}\! \frac{1}{q_1\cdots q_n \big\|q_1\alpha_{1}+\dots+q_n\alpha_n-\gamma\big\|\,}
~\gg~ \log T_1\prod_{i=1}^n\log T_i\,.
\end{equation}
\end{theorem}

\begin{proof}
By Corollary~\ref{cor8}, for any $T_1>T_0$ and any $\gamma_1\in\R$
\begin{equation}\label{cc8}
\sum_{1\le q_1\le T_1 }\frac{1}{q_1\|q_1\alpha_1-\gamma_1\|} \ \ge \ C_2(\log T_1)^2\,,
\end{equation}
where $T_0$ and $C_2$ do not depend on $\gamma_1$. Now consider \eqref{cc8} with
$\gamma_1=\gamma-(q_2\alpha_{2}+\dots+q_n\alpha_n)$, where $q_2,\dots,q_n$ are integers. Then, on  multiplying the resulting inequality  by $(q_2\cdots q_n)^{-1}$ we obtain that
\begin{equation*}
\sum_{1\le q_1\le T_1 }\frac{1}{q_1\cdots q_n\|q_1\alpha_1+\dots+q_n\alpha_n-\gamma\|} \ \ge \ C_2(\log T_1)^2\frac{1}{q_2\cdots q_n}\,.
\end{equation*}
Summing this over $1\le q_i\le T_i \ $ ($2\le i\le n$), we obtain the desired statement.
\end{proof}

\section{Multiplicative Diophantine approximation}\label{MDA}

\subsection{Background}\label{intro2}

Many problems in multiplicative Diophantine approximation can be phrased in terms of   the set
$$
\cS^\times_k(\psi):=\Big\{\,(x_1,\dots,x_k)\in \R^k\,:\,\prod_{i=1}^k\|nx_i\|<\psi(n)\text{ for i.m. }n\in\N\,\Big\},
$$
where $k \ge 1$ is an integer, `i.m.' means `infinitely many,' and $\psi:\N\to\Rp$ is a function which, for obvious reasons, is referred to as an approximating function. The famous conjecture of Littlewood from the nineteen thirties  states  that $\cS^\times_2(n\mapsto \ve n^{-1})=\R^2$ for any $\ve>0$, or equivalently that
\begin{equation}\label{littleor}
\liminf_{n\rar\infty}n\|n\alpha\|\|n\beta\|=0  \quad \text{for all }   \  (\alpha,\beta) \in \R^2\, .
\end{equation}
Littlewood's conjecture has attracted much attention-- see  \cite{Badziahin-Velani-MAD, Einsiedler-Katok-Lindenstrauss-06:MR2247967, Pollington-Velani-00:MR1819996,vent} and references within.   Despite some recent remarkable progress, the Littlewood Conjecture remains  very much open. For instance, we are unable to show that (\ref{littleor}) is valid for the pair $(\sqrt2,\sqrt3)$.

On the contrary, the measure theoretic description of $\cS^\times_2(n\mapsto \ve n^{-1})$, and indeed $\cS^\times_2(\psi)$ is well understood.   For $\psi:\N\rar\Rp$ monotonic\footnote{When $\psi$ is not monotonic the convergence condition reads $\sum_{n=1}^\infty  \psi(n)  \, |\log \psi(n)|^{k-1} <\infty$; see \cite{BHV} for more details.}, a  simple `volume' argument together with the Borel-Cantelli Lemma from probability theory implies that
$$
|\cS^\times_k(\psi)|= 0   \quad {\rm \ if \  }  \quad \sum_{n=1}^\infty  \,\,\psi(n)  \;  (\log n)^{k-1} <\infty  \, .
$$
Throughout,  $|X|$  denotes the
$k$-dimensional Lebesgue measure of the set $X\subset\R^{k} $.

 For the case when the above sum diverges we have the following non-trivial result due to Gallagher~\cite{Gallagher1962}.

\begin{theoremnoname}[Gallagher]
Let $\psi:\N\rar\Rp$ be monotonic. Then
\begin{equation}\label{neweq1}
|\R^k \setminus \cS^\times_k(\psi)|=0   \quad \mbox{ if \ }  \quad \sum_{n=1}^\infty  \,\,\psi(n) \;  (\log n)^{k-1}=\infty  \ .
\end{equation}
\end{theoremnoname}

%\begin{theorem}[Gallagher]
%Let $\psi:\N\rar\Rp$ be monotonic. Then
%\begin{equation}\label{neweq1}
%|\R^k \setminus \cS^\times_k(\psi)|=0   \quad \mbox{ if \ }  \quad \sum_{n=1}^\infty  \,\,\psi(n) \;  (\log n)^{k-1}=\infty  \ .
%\end{equation}
%\end{theorem}

\begin{remark}
In fact Gallagher~\cite{Gallagher1962} established the above statement for $k>1$. The $k=1$ case is the  famous classical (divergent) result of Khintchine \cite{Khintchine-1924} dating back to 1924. The monotonicity condition  in this classical statement cannot in general be relaxed, as was shown by Duffin and Schaeffer  \cite{DuffinSchaeffer1941}  in 1941.  In short, they constructed a non-monotonic function  $\psi$ for which $ \sum_{n=1}^\infty \psi(n) $ diverges but $ |\cS^\times_1(\psi)|=0$.  In the same paper they formulated the following alternative statement for arbitrary approximating functions.

\noindent\textbf{The Duffin-Schaeffer Conjecture.} {\em Let $\psi:\N\to\Rp$.  Then  for  almost all $\alpha \in\R$  the inequality $$|n\alpha-r| < \psi(n)$$  holds for infinitely many  coprime pairs $(n,r) \in \N\times \Z $  provided that}
\begin{equation}\label{divcond1}
\sum_{n=1}^{\infty} \frac{\varphi (n) }{n}  \, \psi (n) =\infty \, .
\end{equation}
Here  $\varphi $ is the Euler phi function. Note for comparison that, for monotonic $\psi$, condition (\ref{divcond1}) is equivalent to  $ \sum_{n=1}^\infty \psi(n) = \infty$. Although various partial results have been obtained, this conjecture represents a key unsolved problem in number
theory.  For background and recent developments regarding this fundamental problem see \cite{Beresnevich-Bernik-Dodson-Velani-Roth,Harman1998,haypollvel}.
However, since we shall make use of it later on, it is worth highlighting the following `partial'  result established in \cite{DuffinSchaeffer1941}.
\end{remark}

\medskip

\begin{theoremnoname}[Duffin \& Schaeffer]
The above conjecture is true if, in addition to the divergent sum condition \eqref{divcond1}, we  also have that
\begin{equation}\label{dst}
 \limsup_{N\to\infty}\left(\sum_{n=1}^{ N} \frac{\varphi (n)}{n}\psi (n)\right)  \left(\sum_{n=1}^{ N}  \psi (n) \right)^{-1} \, >  \, 0  \ .
\end{equation}
\end{theoremnoname}
Note that  condition \eqref{dst} implies that the convergence/divergence behavior of the sum in (\ref{divcond1}) and the sum $ \sum_{n=1}^\infty \psi(n) $  are equivalent.

\bigskip

\begin{remark}
In the case $k > 1$, we expect to be able to remove the condition that $\psi$ is monotonic in the hypothesis of Gallagher's Theorem (without imposing any other condition such as coprimality) if  we replace  the  divergence condition appearing in \eqref{neweq1}  by  $\sum_{n=1}^\infty  \psi(n)  |\log \psi(n)|^{k-1}~=~\infty \, $; see \cite{BHV} for more details.
For the  multiplicative analogue of the Duffin-Schaeffer Theorem see  \cite[Theorem~2]{BHV}.
\end{remark}

\subsection{Problems and main results}

 With $k=2$, observe that for almost all $(\alpha,\beta) \in \R^2$, Gallagher's Theorem improves upon Littlewood's Conjecture by a factor of $(\log n)^2$. More precisely, it implies that
\begin{equation}\label{log2}
\liminf_{n\rar\infty}n(\log n)^2\|n\alpha\|\|n\beta\|=0 \quad \mbox{for almost all }   (\alpha, \beta) \in \R^2  .
\end{equation}
Note that this is beyond the scope of what Khintchine's Theorem  (i.e., \eqref{neweq1} with $k=1$)  can tell us; namely  that
\begin{equation}\label{log}
\liminf_{n\rar\infty}n\log n\|n\alpha\|\|n\beta\|=0  \quad \forall   \  \alpha \in \R  \ \quad  \mbox{and}   \quad \  \mbox{for almost all } \beta \in \R  \, .
\end{equation}
However, the extra $\log$ factor in \eqref{log2} comes at a cost of having to  sacrifice a set of measure zero on the $\alpha$  side. As a consequence, unlike with (\ref{log}) which is valid for any $\alpha$, we are unable to claim that the stronger `$\log$ squared' statement (\ref{log2}) is true for example when $\alpha = \sqrt2$.    Obviously, the role of $\alpha$ and $\beta$ in (\ref{log}) can be reversed. This raises the natural question of whether (\ref{log2}) holds for every $\alpha$. If true, it would mean  that for any $\alpha$ we still beat Littlewood's Conjecture by `log squared' for almost all $\beta$.   In general, this line of thought leads to following problem.

\begin{problem}\label{rubber} Let $k \ge 2$. Prove that for any real numbers $\alpha_1,\dots,\alpha_{k-1}\in\R$ one has that
$$
\liminf_{n\rar\infty}n(\log n)^k \,\|n\alpha_1\|\cdots\|n\alpha_{k-1}\|\,\|n\alpha_k\|   = 0  \quad \  \mbox{for almost all }  \alpha_k \in \R \, .
$$
\end{problem}

\noindent Problems of this nature fall within the scope of the theory of  multiplicative Diophantine approximation on manifolds \cite{BV-07,BKM}. In short, the approximated points $ (\alpha_1,\dots,\alpha_{k}) \in \R^k$ associated with Problem~\ref{rubber} are confined to lie on the manifold, or rather the line,  given by  $(\alpha_1,\dots,\alpha_{k-1}) \times \R$. More generally, one can pose a similar problem for arbitrary submanifolds $\mathcal{M}$  of $\R^k$ -- see \cite{BV-07} for details in the case the manifold is a planar curve.

In this paper we resolve Problem~\ref{rubber} in the two-dimensional case by obtaining the following fiber version of Gallagher's Theorem.  Basically, given $\alpha \in \R$,  the points $ (\alpha_1,\alpha_{2}) \in \R^2$ of interest are forced to lie on the line  given by $ \{ \alpha \} \times  \R  $.

\begin{theorem}\label{T11}
Let $\alpha \in \R$ and let $\psi:\N\to\Rp$ be a monotonic function such  that

\begin{equation}\label{yy}
\sum_{n=1}^{\infty} \,  \psi (n)\,\log n
\end{equation}
diverges and such that for some $\ve>0$,
\begin{equation}\label{lb}
  \liminf_{\ell\to\infty}\,q_\ell^{3-\ve}\psi(q_\ell)\ge1\,  ,
\end{equation}
where  $q_\ell=q_\ell(\alpha)$ denotes the denominators of the principal convergents of  $\alpha$.
Then for almost all $\beta\in\R$, there exists infinitely many $n\in\N$ such that
\begin{equation}\label{ineq}
\|n\alpha\| \,  \|n\beta\| <\psi(n)  \, .
\end{equation}
\end{theorem}

\medskip

\begin{remark}  Condition \eqref{lb} is not particularly restrictive. Irrespective of $\psi$, it holds for all $\alpha$ with exponent of approximation  $w(\alpha)<3$ and it follows from the Jarn\'ik-Besicovitch Theorem \cite{Bes34,Ja29} (alternatively see \cite{BDV-06}) that the complement is of relatively small Hausdorff dimension; namely
$
\dim\{\alpha\in\R:w(\alpha)\ge 3\}=\frac{1}{2}\,.
$
\end{remark}

\medskip

\begin{remark} Theorem \ref{T11} is applicable with $\psi(n)=(n(\log n)^2\log\log n)^{-1}$ and arbitrary irrational $\alpha$ and   thus resolves Problem~\ref{rubber} for $k=2$.
\end{remark}

\medskip

\begin{remark}
We have  reason to believe that  the approach taken in \S\ref{multiproof}  to prove Theorem \ref{T11}  could be  the foundation  for  developing a general  multiplicative framework of regular/ubiquitous systems. We hope to explore this in the near future.  For details of the current  systems and their role in establishing measure theoretic statements for a general class of $\limsup$ sets associated with the  simultaneous  approximation  see \cite{BBD-Baker, BDV-06}.
\end{remark}

\medskip

\begin{remark} For the sake of completeness, we mention that a  strengthening of Khintchine's   simultaneous   theorem  to fibers,   akin to  the  above strengthening of Gallagher's multiplicative theorem, has recently been obtained. Recall that the basic object in the (homogeneous) theory of simultaneous Diophantine approximation is the set
$$
\mathcal{W}_k(\psi):=\Big\{\,(x_1,\dots,x_k)\in \R^k\,:\,\max_{1 \le i \le k}\|nx_i\|<\psi(n)\text{ for i.m. }n\in\N\,\Big\}   \, ,
$$
and under the assumption that  $\psi:\N\to\Rp$ is monotonic, Khintchine's  simultaneous  theorem states that  $$
|\R^k \setminus \mathcal{W}_k(\psi)|=0   \quad \mbox{ if \ }  \quad \sum_{n=1}^\infty  \,\,\psi(n)^{k}=\infty  \ .
$$
We refer the reader to   \cite[\S4.5]{notes}  and references within  for further details.
\end{remark}

\medskip

We next consider the situation in which the sum \eqref{yy} converges.  In this case we are able to obtain the following  inhomogeneous statement that naturally complements the above theorem.  It can be viewed as the fiber analogue of the convergence results  established in \cite{BVV-11, BV-07} for (non-degenerate) planar curves.

\begin{theorem}\label{T12}
Let $\gamma,\delta\in\R$ and $\alpha \in \R$ be any irrational real number and let $\psi:\N\to\Rp$ be such that the sum \eqref{yy} converges. Furthermore,  assume either of the following two conditions\,{\rm:}\\[-5ex]
\begin{itemize}
  \item[{\rm(i)}] $n\mapsto n\psi(n)$ is decreasing and
\begin{equation}\label{cond}
S_N(\alpha;\gamma)\ll(\log N)^2\qquad\text{for all  $N \ge 2 $}\,;
\end{equation}
  \item[{\rm(ii)}] $n\mapsto \psi(n)$ is decreasing and
  \begin{equation}\label{cond+}
R_N(\alpha;\gamma)\ll N\log N\qquad\text{for all  $N \ge 2$}\,.
\end{equation}
\end{itemize}
Then for almost all $\beta\in\R$,   there exist only finitely many $n\in\N$ such that
\begin{equation}\label{ineq+}
\|n\alpha-\gamma\| \,  \|n\beta-\delta\| <\psi(n) \, .
\end{equation}

\end{theorem}

\medskip

\begin{remark}
Recall that in view of Theorem \ref{T4},  we know that  condition \eqref{cond} is satisfied for almost every real number.  This is  far from the truth  regarding condition \eqref{cond+},  as demonstrated by Example \ref{eg1}   in  \S\ref{ae}.   Also recall  that  in the homogeneous case,   we are able to give explicit examples of  real numbers $\alpha$ satisfying \eqref{cond}  and indeed \eqref{cond+},  such as  quadratic irrationals (and more generally badly approximable numbers)  and  $e$ (the base of natural logarithm).
\end{remark}

\begin{corollary}\label{cor10}
Let  $\alpha$ be  a badly approximable number, $\gamma=0$ and $\delta\in\R$. Let  $\psi:\N\to\Rp$  be monotonic such that \eqref{yy} converges.   Then,  for almost every $\beta  \in \R$  inequality  \eqref{ineq+} holds for only finitely many $n\in\N$.
\end{corollary}

\begin{remark}
Conjecture~\ref{conj2}, if true, implies the validity of Corollary~\ref{cor10} for any irrational algebraic number $\alpha$.
\end{remark}

%\begin{remark}
%We would like to point out that conditions \eqref{cond} and \eqref{cond+} cannot be significantly relaxed, if at all. Indeed, if $\alpha$ itself is $\psi$-approximable, that is to say $\|n\alpha\|<\psi(n)$ holds for i.m. $n\in\N$, then \eqref{ineq+} holds for i.m. $n\in\N$ for \emph{all} $\beta$.
%\end{remark}
%

Returning to the  case when  the  sum \eqref{yy}  diverges, we highlight the rather amazing fact  that currently nothing is known concerning the  natural inhomogeneous version of Gallagher's Theorem.

\begin{conjecture}[Inhomogeneous Gallagher] \label{rub}
Let $\gamma, \delta \in \R$ and let  $\psi:\N\rar\Rp$ be monotonic such that the sum \eqref{yy} is divergent.  Then,  for  almost all $(\alpha, \beta)  \in\R^2$  inequality \eqref{ineq+}
  holds for infinitely many   $n \in \N $.
\end{conjecture}

In this paper we establish the following partial result.

\begin{theorem}\label{T13}
Let $\gamma\in\R$, $\delta=0$ and $\psi:\N\rar\Rp$ monotonic, and assume that the sum \eqref{yy} is divergent. Then for almost all $(\alpha,\beta)  \in\R^2$  inequality \eqref{ineq+} holds for infinitely many  $n \in \N $.
\end{theorem}

\noindent  It will be  evident from the proof of Theorem \ref{T13} given in \S\ref{sect13},  that the same  argument would establish Conjecture \ref{rub} in full if we had the inhomogeneous version of the Duffin-Schaeffer Theorem  at hand.  Unfortunately and somewhat surprisingly, such a statement does not seem to be in the existing literature and we have not been able to prove it.

\begin{problem}\label{crazy}
Let $\gamma  \in \R$  and $\psi:\N\rar\Rp$ . Prove, that for  almost all $\alpha \in\R$  the inequality $$|n\alpha-r - \gamma | < \psi(n)$$  holds for infinitely many  coprime pairs $(n,r) \in \N\times \Z $  provided that the divergent sum condition  \eqref{divcond1} and the limsup condition \eqref{dst} hold.
\end{problem}

\smallskip

Of course a version of Conjecture~\ref{rub} in which the points $(\alpha,\beta)$  are restricted to a fiber in $\R^2$ or a (non-degenerate) curve is of interest but currently seems well out of reach. In particular, it is  natural to consider  the following two-dimensional inhomogeneous version of Problem~\ref{rubber}.

\begin{problem}\label{prob2}
Let $\gamma, \delta  \in \R$  and $\alpha\in\R\setminus(\mathfrak{L}\cup\Q)$. Prove that
\begin{equation}\label{y}
\liminf_{n\rar\infty}n\,(\log n)^2 \|n\alpha-\gamma\|\,\|n\beta-\delta\|   = 0  \quad \  \mbox{for almost all }  \beta \in \R \, .
\end{equation}
\end{problem}

\smallskip

Within this paper we obtain the following conditional partial result.

\begin{theorem}\label{T14}
Let $\gamma \in \R, \delta=0$  and $\alpha\in\R\setminus(\mathfrak{L}\cup\Q)$. Then,   under the assumption that the Duffin-Schaeffer Conjecture is true,    \eqref{y} holds.
\end{theorem}

\begin{remark}
 It is worthwhile comparing Problem~\ref{prob2} with the inhomogeneous version of Littlewood's Conjecture.  The latter  states:  {\em  for any $\alpha,\beta\in\R$ such that $1,\alpha,\beta$ are linearly independent over $\Q$
\begin{equation}\label{iLC}
\liminf_{n\rar\infty}n\|n\alpha-\gamma\|\,\|n\beta-\delta\|   = 0 \qquad\text{for any $\gamma,\delta\in\R^2$} \,.
\end{equation} }
The existence of a single pair $(\alpha,\beta)$ satisfying \eqref{iLC} was conjectured by Cassels in the 1950s and remained  open until the  recent work of Shapira \cite{Shapira}. He showed the validity of \eqref{iLC} for almost all pairs $(\alpha,\beta)$ and, in particular, for the basis $1,\alpha,\beta$ of any real cubic number field. A higher dimensional version of these findings have subsequently been established by Lindenstrauss and Shapira \cite{Lindenstrauss-Shapira}.
\noindent In short, Problem~\ref{prob2} implies a $(\log n)^2$ strengthening of the  inhomogeneous  Littlewood's Conjecture at a cost of `losing' a null set of $\beta\in\R$, which almost certainly will depend on $\delta$.
  Note that without the $(\log n)^2$ strengthening, Shapira's result shows that for almost every $\alpha \in \R$,   the null set is independent of $\delta$.  Recently, Gorodnik and Vishe \cite{Gorodnik-Vishe} have obtained a $(\log \log \log \log \log  n)^\lambda$ strengthening of Shapira's result. Here $\lambda$ is some positive constant.
\end{remark}

\newpage
\part{\Large Developing techniques and establishing the  main results}

\bigskip
\bigskip

\section{Ostrowski expansion}\label{sec.contfrac}

The Ostrowski expansion of real numbers  is very much at the heart of our approach towards obtaining good estimates for the size of $\|n\alpha-\gamma\|$.  In short, the  Ostrowski expansion  naturally expresses real numbers in terms of the basic parameters associated with the theory of continued fraction.

\subsection{Continued fractions: the essentials}

To begin with we recall various standard facts from the theory of continued fractions, which can for instance be found in \cite{Khintchine} or \cite{RockettSzusz1992}.
Let $\alpha$ be a real irrational  number. Throughout, the expression
\[
\alpha=[a_0;a_1,a_2,\ldots\,]\,=\,a_0+\dfrac{1}{a_1+\dfrac{1}{a_2+\dfrac{1}{a_3+\raisebox{-2ex}{$\ddots$}}}}
\]
will denote the \emph{simple continued fraction expansion} of $\alpha$.  The integers  $a_k$ are called the \emph{partial quotients} of $\alpha$ and  for $k \ge 1$ satisfy $ a_k \geq 1$.  The reduced rationals
\[
\frac{p_k}{q_k}=[a_0;a_1,\ldots ,a_k]\qquad (k\ge 0)
\]
obtained by truncating the infinite continued fraction expansion   are called the \emph{principal convergents} of $\alpha.$ We will often make use of the quantities
\begin{equation}\label{D_m}
D_k=q_k\alpha-p_k\qquad (k\ge0).
\end{equation}
The following well known relations can be found in \cite{RockettSzusz1992}:
\begin{align}
p_{k+1}=a_{k+1}p_k+p_{k-1},\qquad q_{k+1}=a_{k+1}q_k+q_{k-1}& \qquad (k\ge 1), \label{vb7-}\\[2ex]
D_0=\{\alpha\},\ \quad\qquad  D_k=(-1)^k\|q_k\alpha\| &  \qquad (k\ge1),\label{vb3}\\[2ex]
 a_{k+1}D_k=D_{k+1}-D_{k-1}&   \qquad (k\ge1),\label{vb2}\\[2ex]
  \tfrac12\le q_{k+1}|D_k|\le1 & \qquad (k\ge0).\label{vb7}
\end{align}
Using (\ref{vb3}) and (\ref{vb2}) it is easily verified that
\begin{equation}\label{vb9}
    \rule{0ex}{3.5ex}a_{k+2}|D_{k+1}|+|D_{k+2}|=|D_k|\qquad(k\ge0),\ \ \
\end{equation}
 \begin{equation}\label{vb0}
\sum_{i=1}^\infty a_{k+2i}|D_{k+2i-1}|=|D_k|\qquad(k\ge0)
\end{equation}
and therefore
\begin{equation}\label{vb1}
\sum_{i=k+1}^\infty a_{i+1}|D_i|=|D_k|+|D_{k+1}|\qquad(k\ge0).\qquad
\end{equation}

\noindent Recall that $a_0=\lfloor\alpha\rfloor$, $p_1=a_0a_1+1$, $q_1=a_1$, and therefore
\[D_1=q_1\alpha-p_1=a_1\{\alpha\}-1.\]
In view of \eqref{vb3}, $D_1<0$  and so it follows that $|D_1|=1-a_1\{\alpha\}$.   Now, since $|D_0|=\{\alpha\}$, we have that $a_1|D_0|+|D_1|=1$ and on  using \eqref{vb1} with $k=0$ we obtain that
\begin{equation}\label{vb1zx}
\sum_{i=0}^\infty a_{i+1}|D_i|=|D_0|+1 \,  .
\end{equation}

We will also appeal to the following well known and useful consequence of (\ref{vb7-}), which can easily be verified  by induction:
\begin{equation}\label{zvb}
\max\{2^{\frac{m-1}2},a_1\cdots a_m\}\ \le \ q_m\ \le \ 2^{m-1}\cdot a_1\cdots a_m \qquad \forall   \  m\ge 1 \,.
\end{equation}
Finally, it is easy to show that
\begin{equation}\label{monot}
    |D_k|> |D_{k+1}|\qquad (k\ge0).
\end{equation}

\subsection{The Ostrowski expansion}

The continued fractions framework offers a natural and convenient way of encoding real numbers (both integers and irrationals) via the continued fraction expansion of a given irrational number. The following lemma explicitly determines the  expansion of a positive integer in terms of the denominators of the  principal convergents of a given irrational number.

\begin{lemma}\label{Ostrowski1}
Let $\alpha\in\R\setminus\Q$ and let $a_k=a_k(\alpha)$ and $q_k=q_k(\alpha)$ be as above. Then, for every $n\in\N$ there is a unique integer $K\ge0$ such that ~$$ q_K\le n< q_{K+1},$$ and a unique sequence $\{c_{k+1}\}_{k=0}^\infty$ of integers such that
\begin{equation}\label{contfrac.eqn.Ostrowski1}
n=\sum_{k=0}^\infty c_{k+1}q_k,
\end{equation}
\begin{equation}\label{contfrac.eqn.Ostrowski1+1}
0\le c_1<a_1\quad\text{and}\quad 0\le c_{k+1}\le a_{k+1} \quad \forall   \ k\ge 1,
\end{equation}
\begin{equation}\label{contfrac.eqn.Ostrowski1+2}
\rule{0ex}{3ex}c_k=0 \quad \text{whenever}\quad c_{k+1}=a_{k+1} \quad  \text{with}     \ k\ge 1,
\end{equation}
\begin{equation}\label{contfrac.eqn.Ostrowski1+3}
\rule{0ex}{3ex}c_{k+1}=0  \quad \forall   \ k>K.
\end{equation}
\end{lemma}

\noindent The representation given by Lemma~\ref{Ostrowski1} is  called the \emph{Ostrowski expansion} (or \emph{Ostrowski numeration}) \emph{of the integers}. In view of (\ref{contfrac.eqn.Ostrowski1+3}) the sum (\ref{contfrac.eqn.Ostrowski1}) is  finite.   The next lemma shows that if we work with the quantities $D_k$ rather than simply the denominators  $q_k$, then  a similar representation is valid for irrational numbers.

\begin{lemma}\label{Ostrowski2}
Let $\alpha\in[0,1)\setminus\Q$ and let $a_k=a_k(\alpha)$, $q_k=q_k(\alpha)$ and $D_k=D_k(\alpha)$ be as above, and suppose that $\gamma\in [-\alpha,1-\alpha)$.
Then there is a unique sequence $\{b_{k+1}\}_{k=0}^{\infty}$ of integers such that
\begin{equation}\label{contfrac.eqn.Ostrowski2}
\gamma=\sum_{k=0}^\infty b_{k+1}D_k,
\end{equation}
\begin{equation}\label{contfrac.eqn.Ostrowski2+1}
0\le b_1<a_1\quad\text{and}\quad 0\le b_{k+1}\le a_{k+1}\ \quad \forall   \ k\ge 1,
\end{equation}
\begin{equation}\label{contfrac.eqn.Ostrowski2+}
\rule{0ex}{3ex}b_k=0 \quad \text{whenever}\quad b_{k+1}=a_{k+1} \quad \text{with}   \ k\ge 1.
\end{equation}
\end{lemma}

\noindent Observe that in view of  (\ref{vb1zx}) and (\ref{contfrac.eqn.Ostrowski2+1}), the series in (\ref{contfrac.eqn.Ostrowski2}) is absolutely convergent.

\medskip

Further details of the expansion of real numbers  given by  Lemmas~\ref{Ostrowski1} and \ref{Ostrowski2} can be found in \cite[pp. 24, 33]{RockettSzusz1992}.  As already mentioned,
the Ostrowski expansion of real numbers  is a key tool in our approach towards obtaining  estimates for the size of $\|n\alpha-\gamma\|$, which will be the subject of the next section.

\section{Estimates for $\|n\alpha-\gamma\|$}\label{estim}

In this section we give estimates for the size of $\|n\alpha-\gamma\|$ using the Ostrowski expansions of $n$ and $\gamma$. Before we proceed with the general case of this fairly elaborate task, we consider the much easier homogenous case; that is, the case when $\gamma=0$. Note that if (\ref{contfrac.eqn.|na-g|}) is not satisfied, $\|n\alpha-\gamma\|=\|(n-l)\alpha\|$ for some fixed $l\in\N$ and so we are within the homogeneous setup.

\subsection{The homogeneous case}

We begin by stating a lemma which is a slightly simplified version of \cite[Theorem II.4.1]{RockettSzusz1992}.

\begin{lemma}\label{contfrac.homoglem.homoglem}
Let $\alpha\in[0,1)\setminus\Q$, $n\in\N$, and, with reference to Lemma~\ref{Ostrowski1}, let $m$ be the smallest integer such that $c_{m+1}\not= 0.$ If $m\ge 2$ then
\begin{equation}\label{hom}
\left\|n\alpha\right\|=\left|\sum_{k=m}^\infty c_{k+1}D_k\right|=\sgn (D_m)\cdot\sum_{k=m}^\infty c_{k+1}D_k\,.
\end{equation}
Also if $m=1$ and $\{\alpha\}<1/2$, then we have \eqref{hom}. In all other cases, $\|n\alpha\|\ge|D_2|$, which is a positive constant depending on $\alpha$ only.
\end{lemma}

Lemma~\ref{contfrac.homoglem.homoglem}  underpins the following two-sided estimate for $\|n\alpha\|$.

\begin{lemma}\label{contfrac.lemhombnd.lemhombnd}
Let $\alpha\in[0,1)\setminus\Q$ and $n\in\N$, and let $m$ be as in Lemma~\ref{contfrac.homoglem.homoglem}. If $\|n\alpha\|<|D_2|$ then $m\ge 2,$ and if $m\ge 2$ then
\begin{equation}\label{hom2}
(c_{m+1}-1)|D_m|+(a_{m+2}-c_{m+2})|D_{m+1}|\ \le \ \|n\alpha\| \ \le \ (c_{m+1}+1)|D_{m}|.
\end{equation}
\end{lemma}

\begin{proof}
By Lemma~\ref{contfrac.homoglem.homoglem}, if $\|n\alpha\|<|D_2|$ or $m\ge2$, we have that (\ref{hom}) holds and $m\ge 1$.
Then using the fact that $\sgn(D_k)=(-1)^k$ we find that
\begin{eqnarray}
\|n\alpha\|&=& c_{m+1}|D_m|-c_{m+2}|D_{m+1}|+c_{m+3}|D_{m+2}|-c_{m+4}|D_{m+3}|+\cdots \label{slv11} \\[2ex]
&\ge & c_{m+1}|D_m|+(a_{m+2}-c_{m+2})|D_{m+1}|  \nonumber\\[2ex]
&&\hspace*{18.3ex}-a_{m+2}|D_{m+1}|-a_{m+4}|D_{m+3}|-\cdots  \nonumber \\[2ex]
&\stackrel{\eqref{vb0}}{=} & c_{m+1}|D_m|+(a_{m+2}-c_{m+2})|D_{m+1}|-|D_m| \label{vb11}
\end{eqnarray}
%\begin{equation}\label{vb11}
%\begin{array}[b]{rcl}
%\|n\alpha\|&=& c_{m+1}|D_m|-c_{m+2}|D_{m+1}|+c_{m+3}|D_{m+2}|-c_{m+4}|D_{m+3}|+\ldots\\[2ex]
%&\ge & c_{m+1}|D_m|+(a_{m+2}-c_{m+2})|D_{m+1}|\\[2ex]
%&&\hspace*{18.3ex}-a_{m+2}|D_{m+1}|-a_{m+4}|D_{m+3}|-\cdots\\[2ex]
%&\stackrel{\eqref{vb0}}{=} & c_{m+1}|D_m|+(a_{m+2}-c_{m+2})|D_{m+1}|-|D_m|,\\[1ex]
%\end{array}
%\end{equation}
where the above inequality is obtained by removing the non-negative terms $c_{m+3}|D_{m+2}|$, $c_{m+5}|D_{m+4}|$, \ldots, from (\ref{slv11}) and using \eqref{contfrac.eqn.Ostrowski1+1}. This establishes the l.h.s. of (\ref{hom2}).

\noindent By (\ref{contfrac.eqn.Ostrowski1+2}) and the fact that $c_{m+1}\not=0$, we have that $c_{m+2}<a_{m+2}$ and so if $m$ was equal to $1$, then the l.h.s. of (\ref{hom2}) would imply the inequality $\|n\alpha\|\ge |D_2|$.
Thus, $\|n\alpha\|<|D_2|$ implies that $m\ge 2$.

To obtain the r.h.s. of (\ref{hom2}) we argue similarly to the proof of the lower bound, this time removing the negative terms appearing in  (\ref{slv11}). It follows that
\begin{align*}
\|n\alpha\| & \le c_{m+1}|D_m|+c_{m+3}|D_{m+2}|+c_{m+5}|D_{m+4}|+\cdots\\[1ex]
 & \hspace*{-2.5ex}\stackrel{\eqref{vb0} \& \eqref{contfrac.eqn.Ostrowski1+1}}{\le}
c_{m+1}|D_m|+|D_{m+1}| \stackrel{\eqref{monot}}{\le} (c_{m+1}+1)|D_m|\,.
\end{align*}

The proof is now complete.
\end{proof}

\subsection{The inhomogeneous case \label{qw} }

It is worth pointing out that the inhomogeneous  estimates obtained in this section are valid with $\gamma =0$. We begin by establishing an inhomogeneous analogue of Lemma~\ref{contfrac.homoglem.homoglem}.

\begin{lemma}\label{contfrac.inhomoglem.inhomoglem}
Let $\alpha\in [0,1)\setminus\Q$ and $\gamma\in [-\alpha,1-\alpha)$, and suppose that {\rm(\ref{contfrac.eqn.|na-g|})} holds. Further,  let $n\in\N$ and, with reference to expansions {\rm(\ref{contfrac.eqn.Ostrowski1})} and {\rm(\ref{contfrac.eqn.Ostrowski2})}, let
\begin{equation}\label{delta}
\delta_{k+1}:=c_{k+1}-b_{k+1}\qquad (k\ge0)\,.
\end{equation}
Then, there exists a smallest integer $m=m(n,\alpha,\gamma)$ such that $\delta_{m+1}\not=0$ and
\begin{equation}\label{Sigma}
\Sigma=\Sigma(n,\alpha,\gamma):=\sum_{k=m}^\infty \delta_{k+1}D_k
\end{equation}
satisfies the equations
\begin{equation}\label{vb6A}
\|n\alpha-\gamma\|=\|\Sigma\|=\min\Big\{\left|\Sigma\right|,1-|\Sigma|\Big\}
\end{equation}
and
\begin{equation}\label{vb6B}
|\Sigma|=\sgn(\delta_{m+1}D_m)\Sigma\,.
\end{equation}
\end{lemma}

\begin{proof}
Since $\|\cdot\|$ is invariant under integer translation, the l.h.s. of (\ref{vb6A}) is a straightforward consequence of (\ref{D_m}), (\ref{contfrac.eqn.Ostrowski1}) and (\ref{contfrac.eqn.Ostrowski2}), and the existence of $m$ follows trivially from \eqref{contfrac.eqn.|na-g|}.
Furthermore, by (\ref{contfrac.eqn.Ostrowski1+1}) and (\ref{contfrac.eqn.Ostrowski2+1}), we have that $|\delta_{2}|\le a_{2}$ and $|\delta_{1}|\le a_{1}-1$. Then,
\begin{equation}\label{vb8}
|\Sigma|\le \sum_{k=0}^\infty |\delta_{k+1}D_k| \ \le \
\sum_{k=0}^\infty a_{k+1}|D_k|-|D_0|\ \stackrel{\eqref{vb1zx}}{=}1.
\end{equation}
This implies that $\|\Sigma\|=\min\{|\Sigma|,1-|\Sigma|\}$ and establishes  the r.h.s. of (\ref{vb6A}).

In order to prove (\ref{vb6B}) first note that $|\delta_{k+1}|\le a_{k+1}$ for all $k\ge0$. Then, for any integer $l\ge1$
\begin{equation}\label{vb5}
\begin{array}[b]{rcl}
\displaystyle\left|\sum_{k=l}^\infty \delta_{k+1}D_k\right|
&\stackrel{}{\le} & |\delta_{l+1}|\cdot|D_{l}|+ \displaystyle\sum_{k=l+1}^\infty a_{k+1}|D_{k}|\\[2ex]
& \stackrel{\eqref{vb1}}{=} & |\delta_{l+1}|\cdot|D_{l}|+|D_{l}|+|D_{l+1}|\\[2ex]
& \stackrel{\eqref{vb9}}{=} & (|\delta_{l+1}|+1-a_{l+1})|D_{l}|+|D_{l-1}|.\\[1ex]
\end{array}
\end{equation}
First, consider the case when $|\delta_{m+2}|\le a_{m+2}-1$. Then, by (\ref{vb5}) with $l=m+1$, we have that $\left|\sum_{k=m+1}^\infty \delta_{k+1}D_k\right|\le |D_m|$. On the other hand, since $\delta_{m+1}\not=0$ and it is an integer, we have that $|\delta_{m+1}D_m|\ge |D_m|$. Therefore, the first term in the sum $\sum_{k=m}^\infty\delta_{k+1}D_k$ dominates and so (\ref{vb6B}) holds.

Now assume that $\delta_{m+2}=a_{m+2}$. Then, by (\ref{contfrac.eqn.Ostrowski1+1}), (\ref{contfrac.eqn.Ostrowski2+1}) and (\ref{delta}), we have that $c_{m+2}=a_{m+2}$ and $b_{m+2}=0$. Consequently, by (\ref{contfrac.eqn.Ostrowski1+2}), we get
$c_{m+1}=0$, and so $\delta_{m+1}=c_{m+1}-b_{m+1}<0$.
Then
$\sgn(\delta_{m+1}D_m)=\sgn(\delta_{m+2}D_{m+1})$,
which implies that
$$
|\delta_{m+1}D_m+\delta_{m+2}D_{m+1}|\ge|D_m|+|D_{m+1}|.
$$
In turn, by (\ref{vb5}), we have that
$$
\left|\sum_{k=m+2}^\infty \delta_{k+1}D_k\right|\le
|D_{m+1}|+|D_{m+2}|<|D_{m}|+|D_{m+1}|.
$$
Therefore the first two terms in the sum $\sum_{k=m}^\infty\delta_{k+1}D_k$ (which have the same sign) dominate and so (\ref{vb6B}) holds again.

The remaining case $\delta_{m+2}=-a_{m+2}$ is established in a similar manner  to the case $\delta_{m+2}=a_{m+2}$.  We leave the details to the reader.
\end{proof}

\bigskip

Lemma~\ref{contfrac.inhomoglem.inhomoglem} enables us to compute the value of $\|n\alpha-\gamma\|$ via $|\Sigma|$. The next two lemmas, akin to Lemma~\ref{contfrac.lemhombnd.lemhombnd}, provide accurate estimates for  $|\Sigma|$ and $1-|\Sigma|$.

\begin{lemma}\label{contfrac.leminhombnd.leminhombnd}
Let $\alpha$, $\gamma$, $n$, $\delta_{k+1}$, $m$ and $\Sigma$ be  as in Lemma~\ref{contfrac.inhomoglem.inhomoglem} and let $K$ be as in Lemma~\ref{Ostrowski1}. Then there exists some $\ell\in\N$, with $\ell\le \max\{2,K-m+1\}$, such that
    \begin{equation}\label{vb12}
        \delta_{m+2+i}=(-1)^{i}\sgn(\delta_{m+1})a_{m+2+i}\qquad(1\le i\le \ell-1),
    \end{equation}
    \begin{equation}\label{vb12+}
        \delta_{m+2+\ell}\not=(-1)^{\ell}\sgn(\delta_{m+1})a_{m+2+\ell}\phantom{\qquad(1\le i\le \ell-1),},
    \end{equation}
and we have that
\begin{equation}\label{contfrac.leminhombnd.eqn1}
\begin{array}[b]{rcl}
%\|n\alpha-\gamma\|
\displaystyle |\Sigma| & = & (|\delta_{m+1}|-1)|D_m|\\[2ex]
 &  & + \ (a_{m+2}-1-\sgn(\delta_{m+1})\delta_{m+2})|D_{m+1}|\\[2ex]
 &  & + \ (a_{m+2+\ell}-(-1)^\ell\sgn(\delta_{m+1})\delta_{m+2+\ell})|D_{m+1+\ell}|\\[2ex]
 & & +\ \Delta\,,\\[1ex]
\end{array}
\end{equation}
where
\begin{eqnarray}
    &&0~\le~ |\delta_{1}|-1 ~\le~ a_{1}-2~~~~~~~~~~~(\text{if }~m=0)\,,\label{delta1}\\[1ex]
    &&0~\le~ |\delta_{m+1}|-1 ~\le~ a_{m+1}-1~~~~(\text{if }~m\ge1)\,,\label{deltam}\\[1ex]
    &&0~\le~ a_{m+2}-1-\sgn(\delta_{m+1})\delta_{m+2} ~\le~ 2a_{m+2}-1\,,\label{eta1}\\[1ex]
    &&1~\le~ a_{m+2+\ell}-(-1)^\ell\sgn(\delta_{m+1})\delta_{m+2+\ell} ~\le~ 2a_{m+2+\ell}\label{eta2}\,,\quad\text{and}\\[1ex]
    &&0~\le~ \Delta~\le~ 2|D_{m+1+\ell}|+2|D_{m+2+\ell}|~<~ 4|D_{m+1+\ell}|\label{eta}\,.
\end{eqnarray}
\end{lemma}

\vspace*{3ex}

\begin{remark}
In the above statement $\ell$ can in principle be $1$ in which case condition \eqref{vb12} is empty and thus automatically holds.
\end{remark}

\begin{proof}
By (\ref{contfrac.eqn.Ostrowski1+3}) and (\ref{contfrac.eqn.Ostrowski2+1}), we have that $\delta_{m+2+\ell}=-b_{m+2+\ell}\le 0$ for all $\ell\ge \ell_0:=\max\{1,K-m\}$. Obviously, the expression  $$(-1)^{\ell}\sgn(\delta_{m+1})a_{m+2+\ell}$$ takes one positive and one negative value when $\ell\in\{\ell_0,\ell_0+1\}$. Therefore, \eqref{vb12+} holds for some $\ell\le \ell_0+1=\max\{2,K-m+1\}$. On letting $\ell$ to be the minimal positive integer satisfying \eqref{vb12+} we also ensure  the validity of (\ref{vb12}).

Equations \eqref{delta1} and \eqref{deltam} follow immediately from \eqref{delta} together with he assumption that $\delta_{m+1}\neq0$, and properties \eqref{contfrac.eqn.Ostrowski1+1} and \eqref{contfrac.eqn.Ostrowski2+1} of the Ostrowski expansion.

Next, we consider \eqref{eta1}. If $\delta_{m+1}>0$, then $c_{m+1}=\delta_{m+1}+b_{m+1}>0$ and, by (\ref{contfrac.eqn.Ostrowski1+1}), we have that $c_{m+2}\le a_{m+2}-1$. This means that $\sgn(\delta_{m+1})\delta_{m+2}=\delta_{m+2}=c_{m+2}-b_{m+2}\le a_{m+2}-1$ and implies the l.h.s. of \eqref{eta1}. Similarly, if $\delta_{m+1}<0$, then $b_{m+1}=c_{m+1}-\delta_{m+1}>0$ and, by (\ref{contfrac.eqn.Ostrowski2+1}), we have that $b_{m+2}\le a_{m+2}-1$. This means that $\sgn(\delta_{m+1})\delta_{m+2}=-\delta_{m+2}=b_{m+2}-c_{m+2}\le a_{m+2}-1$ and again implies the l.h.s. of \eqref{eta1}. In either case $|\sgn(\delta_{m+1})\delta_{m+2}|\le a_{m+2}$ and so the r.h.s. of \eqref{eta1} is straightforward.

Equation \eqref{eta2} is a consequence of (\ref{vb12+}) and the inequality
$|\delta_{m+2+\ell}|\le a_{m+2+\ell}$, which is valid for the same reason as the r.h.s. of \eqref{deltam}.

What remains is to prove (\ref{eta}). By (\ref{Sigma}), (\ref{vb6B}) and the fact that $\sgn(D_i)=(-1)^i$, we obtain that
\begin{align*}
|\Sigma| & =  \sgn(\delta_{m+1}D_m)\cdot\sum_{k=m}^\infty \delta_{k+1}D_k\\[2ex]
 & =  |\delta_{m+1}|\,|D_m|-\sgn(\delta_{m+1})\delta_{m+2}|D_{m+1}|+\\[2ex]
  & \qquad+\sum_{k=m+2}^\infty (-1)^{k-m}\sgn(\delta_{m+1})\delta_{k+1}|D_k|\\[2ex]
  & \stackrel{\eqref{vb12}}{=}  |\delta_{m+1}|\,|D_m|-\sgn(\delta_{m+1})\delta_{m+2}|D_{m+1}|-\sum_{k=m+2}^{m+\ell} a_{k+1}|D_k|-\\[2ex]
  & \qquad- (-1)^{\ell}\sgn(\delta_{m+1})\delta_{m+\ell+2}|D_{m+\ell+1}|+\\[2ex]
  & \qquad+ \sum_{k=m+\ell+2}^\infty (-1)^{k-m}\sgn(\delta_{m+1})\delta_{k+1}|D_k|\,.
\end{align*}
Therefore, on combining the implicit expression  for $\Delta$ from (\ref{contfrac.leminhombnd.eqn1}) with the above equation for $|\Sigma|$, we obtain  that
\begin{equation}
\Delta= |D_m|+|D_{m+1}|-\sum_{k=m+1}^{m+\ell+1}a_{k+1}|D_{k}|+\sum_{k=m+\ell+2}^\infty (-1)^{k-m}\sgn(\delta_{m+1})\delta_{k+1}|D_k|\,.
\end{equation}
Therefore, since $|\delta_{k+1}|\le a_{k+1}$, we get that
$$
\Delta\ge |D_m|+|D_{m+1}|-\sum_{k=m+1}^{\infty}a_{k+1}|D_{k}|\stackrel{\eqref{vb1}}{=}0
$$
and
$$
\begin{array}{rcl}
\Delta&\le& \displaystyle|D_m|+|D_{m+1}|-\sum_{k=m+1}^{m+1+\ell}a_{k+1}|D_{k}|+
\sum_{k=m+2+\ell}^{\infty}a_{k+1}|D_{k}|\\[4ex]
 & = & \underbrace{\displaystyle|D_m|+|D_{m+1}|- \sum_{k=m+1}^{\infty}a_{k+1}|D_{k}|}_{\stackrel{\phantom{\,\eqref{vb1}}|\hspace*{0.05ex}|\,\eqref{vb1}}{\rule{0ex}{2ex} \textstyle0}} \quad +\quad
\displaystyle2\sum_{k=m+2+\ell}^{\infty}a_{k+1}|D_{k}|\\[9.5ex]
&\stackrel{\eqref{vb1}}{=} & 2|D_{m+1+\ell}|+2|D_{m+2+\ell}|< 4|D_{m+1+\ell}|\,.
\end{array}
$$
This completes the proof.
\end{proof}

\begin{corollary}\label{cor1}
Under the conditions of Lemma~\ref{contfrac.leminhombnd.leminhombnd},  we have that
$$
\|n\alpha-\gamma\|< (|\delta_{m+1}|+2)|D_m|.
$$
\end{corollary}

\begin{proof}
By (\ref{contfrac.leminhombnd.eqn1}) and (\ref{eta}), it follows that
\begin{align*}
\|n\alpha-\gamma\| \ & \le\ (|\delta_{m+1}|-1)|D_m|+(2a_{m+2}-1)|D_{m+1}|+(2a_{m+3}+4)|D_{m+2}|\\[1ex]
& \stackrel{\eqref{vb9}}{=}  (|\delta_{m+1}|+1)|D_m|+|D_{m+1}|< (|\delta_{m+1}|+2)|D_m|.
\end{align*}
\end{proof}

\begin{lemma}\label{contfrac.leminhombnd.leminhombnd2}
Let $\alpha$, $\gamma$, $n$, $\delta_{k+1}$, $m$ and $\Sigma$ be  as in Lemma~\ref{contfrac.inhomoglem.inhomoglem} and let $K$ be as in Lemma~\ref{Ostrowski1}.  Then there exists a positive integer $L\le K+2$ such that
    \begin{equation}\label{vb12zx}
        \delta_{i+1}=(-1)^{i+m}\sgn(\delta_{m+1})a_{i+1}\qquad(1\le i\le L-1),
    \end{equation}
    \begin{equation}\label{vb12+zx}
        \delta_{L+1}\not=(-1)^{L+m}\sgn(\delta_{m+1})a_{L+1}\phantom{\qquad(1\le i\le L-1)}
    \end{equation}
and
\begin{equation}\label{contfrac.leminhombnd.eqn1zx}
\begin{array}[b]{rcl}
%\|n\alpha-\gamma\|
\displaystyle 1-|\Sigma| & = & (a_1-1-(-1)^m\sgn(\delta_{m+1})\delta_1)|D_0|\\[2ex]
 &  & + \ \big(a_{L+1}-(-1)^{L+m}\sgn(\delta_{m+1})\delta_{L+1}\big)|D_L|\\[2ex]
 & & +\ \widetilde\Delta\,,\\[1ex]
\end{array}
\end{equation}
where all the three terms in the r.h.s.\,of \eqref{contfrac.leminhombnd.eqn1zx} are non-negative, and
\begin{equation}\label{etazx}
    0\le \widetilde\Delta< 4|D_{L}|\,.
\end{equation}
\end{lemma}

\begin{proof}
By (\ref{contfrac.eqn.Ostrowski1+3}) and (\ref{contfrac.eqn.Ostrowski2+1}), we have that $\delta_{i+1}\le 0$ for all $i\ge K+1$. This readily implies the existence of $L\in\N$ such that $L\le K+2$ and (\ref{vb12zx}) and (\ref{vb12+zx}) hold -- see the proof of Lemma~\ref{contfrac.leminhombnd.leminhombnd} for a similar and more detailed argument. The non-negativity of the first two terms in the r.h.s. of \eqref{contfrac.leminhombnd.eqn1zx} follows immediately from the facts that $|\delta_1|\le a_1-1$ and $|\delta_{k+1}|\le a_{k+1}$ for $k\ge1$.

Now we proceed to the proof of (\ref{etazx}). We have that
\begin{align*}
1-|\Sigma| \ & \stackrel{\eqref{Sigma} \& \eqref{vb6B}}{=} \  1-\sgn(\delta_{m+1}D_m)\cdot\sum_{k=m}^\infty \delta_{k+1}D_k\\[0ex]
           & \stackrel{\eqref{vb1zx}}{=} \  (a_1-1)|D_0|+\sum_{k=1}^\infty a_{k+1}|D_k| -\sgn(\delta_{m+1}D_m)\cdot\sum_{k=m}^\infty \delta_{k+1}D_k\\[2ex]
           &= \  (a_1-1-(-1)^m\sgn(\delta_{m+1})\delta_1)|D_0|\\[2ex]
           &  \qquad+ \sum_{k=1}^\infty \big(a_{k+1}-(-1)^{k+m}\sgn(\delta_{m+1})\delta_{k+1}\big)|D_k|\\[2ex]
           & \stackrel{\eqref{vb12zx}}{=} \  (a_1-1-\sgn(\delta_{m+1}D_m)\delta_1)|D_0|\\[2ex]
           &\qquad+\big(a_{L+1}-(-1)^{L+m}\sgn(\delta_{m+1})\delta_{L+1}\big)|D_L|\\[2ex]
           &  \qquad+ \sum_{k=L+1}^\infty \big(a_{k+1}-(-1)^{k+m}\sgn(\delta_{m+1})\delta_{k+1}\big)|D_k|\,.
\end{align*}
The latter sum is by definition $\widetilde\Delta$. Since $|\delta_{k+1}|\le a_{k+1}$ for all $k$, we trivially have that
$$
0\ \le \ \widetilde\Delta \ \le \ 2 \sum_{k=L+1}^\infty a_{k+1}|D_k| \ \stackrel{\eqref{vb1}}{=} \ 2(|D_L|+|D_{L+1}|)\ \le \ 4 |D_L|\,.
$$
This establishes  (\ref{etazx}) and completes the proof.
\end{proof}

\section{A gaps lemma and its consequences}\label{gaps}

In the course of establishing Theorems~\ref{T4} -- \ref{T8} we will group together natural numbers $n$ such that $\|n\alpha-\gamma\|$ is of comparable size. Akin to Lebesgue integration, this natural idea will prove to be extremely fruitful for our purposes. We shall see that the `grouping' in question leads to investigating  subsets $A(d_1,\dots,d_{m+1}) $ of natural numbers that have, up to certain point, pre-determined digits  $d_1,\dots,d_{m+1}$  in their Ostrowski expansions. The following lemma reveals the structure of such sets in terms of the  size of the gap  between consecutive integers in $A(d_1,\dots,d_{m+1}) $.

\begin{lemma}[Gaps Lemma]\label{gapslemma1}
Let $\alpha\in\R\setminus\Q$, $m\ge0$ and $d_1,\dots,d_{m+1}$ be non-negative integers such that $d_1<a_1$, $d_{k+1}\le a_{k+1}$ for $k=1,\dots,m$ and $d_{k}=0$ whenever $d_{k+1}=a_{k+1}$ for $k\le m$. Let $A=A(d_1,\dots,d_{m+1})$ denote the set of all $n\in\N$ with Ostrowski expansions of the form
\begin{equation}\label{contfrac.sumbndlem.eqn1}
n \ = \ \sum_{k=0}^md_{k+1}q_k \ +\sum_{k=m+1}^\infty c_{k+1}q_k\,.
\end{equation}
Write $A=\{n_1<n_2<\dots\}$ in increasing order. Then, for every $i\ge 1$ we have that:
\begin{itemize}
  \item[{\rm(i)}] if $d_{m+1}>0$ then $n_{i+1}-n_i\in\{q_{m+1},\ q_{m+1}+q_m\}$,  and \\[-1.5ex]
  \item[{\rm(ii)}] if $d_{m+1}=0$ then $n_{i+1}-n_i\in\{q_{m+1},\ q_m\}$. Furthermore, if $n_{i+1}-n_i=q_m$ then $c_{m+2}(n_i)=a_{m+2}$ and the gap $n_{i+1}-n_i$ is preceded by $a_{m+2}$ consecutive gaps of length $q_{m+1}$.
\end{itemize}
\end{lemma}

\begin{proof}
The Ostrowski expansion of a natural number $n$ can be defined via the following well known ``greedy algorithm''. First, choose the largest integer $M$ satisfying $q_M\le n$. This is possible since $\alpha$ is irrational and so the sequence $q_k$ is unbounded. Then, select $c_{M+1}$ to be the largest integer satisfying $c_{M+1}q_M \le n$. Next, subtract $c_{M+1}q_M$ from $n$ and then repeat this process on the remaining integer to determine $c_M$ and so on. Formally, with $n_M:=n$ we define
\begin{equation}\label{greedy}
c_{k+1}:=\left[\frac{n_k}{q_{k}}\right]\quad\text{and}\quad n_{k-1}:=n_k-c_{k+1}q_k\qquad\text{for}\quad k=M,M-1,\dots,0\,.
\end{equation}
Since $q_0=1$, this process will terminate with the remainder equal to $0$. Also we define $c_{k+1}=0$ for $k>M$.

Using (\ref{vb7-}) and (\ref{greedy}) one easily verifies properties (\ref{contfrac.eqn.Ostrowski1})--(\ref{contfrac.eqn.Ostrowski1+3}); that is, the above algorithm produces the Ostrowski expansion of $n$. Consequently, the set $A(d_1,\dots,d_{m+1})$ can be put in increasing order by using the reverse lexicographic ordering of the Ostrowski coefficients. For the rest of the proof let
\begin{equation}\label{n'}
n':=\sum_{k=0}^md_{k+1}q_k.
\end{equation}
In view of the conditions imposed on $d_1,\dots,d_{m+1}$ in the statement of the lemma, we have that $A(d_1,\dots,d_{m+1})\not=\emptyset$. Furthermore, if $n'>0$ then $n'\in A(d_1,\dots,d_{m+1})$ and in this case $n'$ is the smallest element of $A(d_1,\dots,d_{m+1})$. Now we fix some  $i\in\N$ and suppose that
$$
n_{i+1}=n'+\sum_{k=m'}^\infty c_{k+1}q_k,
$$
where $c_{k+1}=c_{k+1}(n_{i+1})$ are the Ostrowski coefficients of $n_{i+1}$ and $m'\ge m+1$ with $c_{m'+1}\not=0$. Since $i+1>1$, we have that $n_{i+1}>n_1\ge n'$, and such an $m'$ exists. We now consider three separate scenarios.

\medskip

\noindent\textbf{Case (1)}. Suppose that $m'=m+1$. Then, by the reverse lexicographic ordering of $A(d_1,\dots,d_{m+1})$, we have that
  $$
  n_i=n'+(c_{m+2}-1)q_{m+1}+\sum_{k=m+2}^\infty c_{k+1}q_k\,.
  $$
  In this case,  we obviously have that $n_{i+1}-n_i=q_{m+1}$.

\medskip

  \noindent\textbf{Case (2)}. Suppose that $m'-m$ is positive and even. Then again appealing to the reverse lexicographic ordering we find that
  $$
\begin{array}{r}
  n_i=n'+a_{m+3}q_{m+2}+a_{m+5}q_{m+4}+\dots+ a_{m'}q_{m'-1}\\[2ex] \displaystyle +(c_{m'+1}-1)q_{m'}+ \sum_{k=m'+1}^\infty c_{k+1}q_k\,.
\end{array}
  $$
 Using (\ref{vb7-}) one readily verifies that $n_{i+1}-n_i=q_{m+1}$ in this case.

\medskip

  \noindent\textbf{Case (3)}. Suppose that $m'-m-2$ is positive and odd. Then the form of $n_i$ depends on whether or not $d_{m+1}>0$. If $d_{m+1}>0$, then
  $$
  \begin{array}{r}
  n_i=n'+(a_{m+2}-1)q_{m+1}+a_{m+4}q_{m+3}+a_{m+6}q_{m+5}+\dots+a_{m'}q_{m'-1}\\[2ex] \displaystyle +(c_{m'+1}-1)q_{m'}+ \sum_{k=m'+1}^\infty c_{k+1}q_k\,.
\end{array}
  $$
  Again using (\ref{vb7-}) one verifies that $n_{i+1}-n_i=q_{m+1}+q_{m}$. Finally, if
 $d_{m+1}=0$ then
  $$
  \begin{array}{r}
  n_i=n'+a_{m+2}q_{m+1}+a_{m+4}q_{m+3}+a_{m+6}q_{m+5}+\dots+a_{m'}q_{m'-1}\\[2ex] \displaystyle +(c_{m'+1}-1)q_{m'}+ \sum_{k=m'+1}^\infty c_{k+1}q_k\,,
\end{array}
  $$
  and using (\ref{vb7-}) gives $n_{i+1}-n_i=q_m$. The `furthermore' part of (ii) easily follows from the above explicit form of $n_i$ and the reverse lexicographic ordering of $A(d_1,\dots,d_{m+1})$. This completes the proof.
\end{proof}

\bigskip

Using Lemma~\ref{gapslemma1},  we are able to  prove various useful counting results.

\begin{lemma}\label{contfrac.sumbndlem.sumbndlem}
Let $\alpha\in\R\setminus\Q$, $N\ge3$ and $m\ge0$. Furthermore, let $d_1,\dots,d_{m+1}$ and $A(d_1,\dots,d_{m+1})$ be as in Lemma~\ref{gapslemma1}, let $n'$ be given by \eqref{n'} and let
$$
A_N(d_1,\dots,d_{m+1}):=A(d_1,\dots,d_{m+1})\cap[1,N].
$$
Then
\begin{equation*}
\sum_{\substack{n\in A_N(d_1,\ldots , d_{m+1})\\n\not= n'}}\frac{1}{n}\le \frac{5\log N}{q_{m+1}}.
\end{equation*}
\end{lemma}

\begin{proof}
We will assume that $A_N(d_1,\ldots , d_{m+1})$ contains at least one element different from $n'$, as otherwise the sum under consideration is zero and there is nothing to prove.
By Lemma~\ref{gapslemma1}, for any distinct $n_1,n_2\in A_N(d_1,\dots,d_{m+1})$ with $c_{m+2}(n_i)<a_{m+2}$ ($i=1,2$) we have that $|n_1-n_2|\ge q_{m+1}$. In this case we find that
\begin{align*}
\sum_{\substack{n\in A_N(d_1,\ldots , d_{m+1})\\c_{m+2}<a_{m+2},\ n\not= n'}}\frac{1}{n}\ & \le\ \sum_{1\le\ell\le N/q_{m+1}}\frac{1}{n'+\ell q_{m+1}} \\
& \le \ \frac{1}{q_{m+1}}\sum_{1\le\ell\le N/q_{m+1}}\frac{1}{\ell}\ \le \ \frac{2\log N}{q_{m+1}}.
\end{align*}
On the other hand, any distinct $n_1,n_2\in A_N(d_1,\dots,d_{m+1})$ with $c_{m+2}(n_i)=a_{m+2}$ ($i=1,2$) lie in $A_N(d_1,\dots,d_{m+1},a_{m+2})$. Then, by Lemma~\ref{gapslemma1}, it follows that $|n_1-n_2|\ge q_{m+2}$ and so similarly
\begin{align*}
\sum_{\substack{n\in A_N(d_1,\ldots , d_{m+1})\\c_{m+2}=a_{m+2}}}\frac{1}{n}\ &\le\sum_{0\le\ell\le N/q_{m+2}}\frac{1}{n'+a_{m+2}q_{m+1}+\ell q_{m+2}} \\[1ex]
& \le\ \frac{1}{a_{m+2}q_{m+1}}+\frac{2\log N}{q_{m+2}}\ \le\ \frac{3\log N}{q_{m+1}}\,.\\[-2ex]
\end{align*}
This together with the previous  displayed estimate implies the desired statement.
\end{proof}

\begin{lemma}\label{cortogapslem1}
Under the conditions of Lemma~\ref{contfrac.sumbndlem.sumbndlem}, let us assume that
$$
\#A_N(d_1,\ldots , d_{m+1})\ge 1.
$$
Then
\begin{equation}\label{gaps2}
\frac{N}{3q_{m+1}}\ \le \ \#A_N(d_1,\ldots , d_{m+1}) \ \le \ \frac{3N}{q_{m+1}}+1\,.
\end{equation}
\end{lemma}

\begin{proof}
In order to establish the lower bound, first observe  that since we are assuming that  $\#A_N(d_1,\ldots , d_{m+1})\ge 1$, the lower bound inequality in  (\ref{gaps2}) is trivial if $N\le3q_{m+1}$. Without loss of generality,  assume that $N>3q_{m+1}$. It is readily verified  that the minimal element of $A(d_1,\ldots , d_{m+1})$ is $\le q_{m+1}$. Hence, by Lemma~\ref{gapslemma1}, any block of $2q_{m+1}$ consecutive positive integers contains at least one element of $A(d_1,\ldots ,d_{m+1})$. Therefore,
\begin{align*}
\#A_N(d_1,\ldots , d_{m+1})\ge 1+\left\lfloor\frac{N-q_{m+1}}{2q_{m+1}}\right\rfloor\ge \frac{N-q_{m+1}}{2q_{m+1}}\stackrel{N>3q_{m+1}}{\ge}\frac{N}{3q_{m+1}}\,.
\end{align*}
This verifies the lower bound of (\ref{gaps2}).

For the upper bound write $A(d_1,\ldots,d_{m+1})=\{n_1<n_2<\cdots\}$ in increasing order.
Let $t=\#A_N(d_1,\ldots , d_{m+1})$. Thus,  $n_t\le N<n_{t+1}$.  Observe, that without loss of generality, we can  assume that $t\ge 2$ as otherwise the upper bound in (\ref{gaps2}) is trivial.  With this in mind,
we consider separately the two cases appearing in  the statement of Lemma \ref{gapslemma1}.

\noindent $\bullet \ $ In case (i), we have that $n_{i+1}-n_i\ge q_{m+1}$ for all $i\ge 1$ and therefore $N\ge n_1+(t-1)q_{m+1}$. Hence
\begin{equation}\label{gaps3}
\frac{N}{\#A_N(d_1,\ldots , d_{m+1})}=\frac{N}{t}\ge \frac{n_1+(t-1)q_{m+1}}{t}\ \stackrel{t\ge2}{\ge}\ \frac{q_{m+1}}{2},
\end{equation}
and the upper bound in (\ref{gaps2}) readily follows.

\noindent $\bullet \ $ In  case (ii), we have that  any gap of length $q_m$ among $\{n_1,n_2,\ldots\}$ must be preceded by $a_{m+2}$ consecutive gaps of length $q_{m+1}$.
Therefore, if $t-1< a_{m+2}+1$, then all the gaps $n_{i}-n_{i-1}$ with $i\le t$ must be $q_{m+1}$ and we have that $N\ge n_1+(t-1)q_{m+1}$. Then, by (\ref{gaps3}), we again conclude that the upper bound in (\ref{gaps2}) holds. It remains to consider the case $t-1\ge a_{m+2}+1$. In this case,  it is obvious that $t\ge3$. Further, by the division algorithm, $t-1=Q(a_{m+2}+1)+R$ with $0\le R<a_{m+2}+1$. Then, there are at most $Q$ gaps of length $q_m$, each preceded by $a_{m+2}$ consecutive gaps of length $q_{m+1}$, plus $R$ gaps of length $q_{m+1}$. Then,  we have that
\begin{align*}
N &\ge n_1+\frac{t-1-R}{a_{m+2}+1}\,(a_{m+2}q_{m+1}+q_m)+Rq_{m+1}\\[2ex]
&\ge  (t-1-R)\frac{q_{m+1}}{2}+Rq_{m+1}\\[2ex]
&\ge (t-1-R)\frac{q_{m+1}}{2}+R\frac{q_{m+1}}2=\frac{t-1}2\,q_{m+1}\\[2ex]
&\stackrel{t\ge3}{\ge} \frac t3\,q_{m+1}
 =\frac{q_{m+1}\cdot\#A_N(d_1,\ldots , d_{m+1})}3\,,
\end{align*}
whence the upper bound in (\ref{gaps2}) follows.
\end{proof}

\section{Counting solutions to $\|n\alpha-\gamma\|<\ve$}

Given an $\alpha,\gamma\in\R$, $N\in\N$ and $\ve>0$, consider the  set
\begin{equation} \label{ayesha}
N_\gamma(\alpha,\ve):=\{n\in\N:\|n\alpha-\gamma\|<\ve,\ n\le N\}\,.
\end{equation}
In the homogeneous case ($\gamma=0$),  we will simply  write $N(\alpha,\ve)$ for $N_0(\alpha,\ve)$.
The main goal of this section is to estimate the cardinality of $N_\gamma(\alpha,\ve)$.

\subsection{The homogeneous case}

We begin by observing  that when $\ve N\ge1$, Minkowski's Theorem for convex bodies, see \cite[p.71]{CasselsGofN}, implies that
\begin{equation}\label{vbf}
\#N(\alpha,\ve)\ge \lfloor \ve N\rfloor\,.
\end{equation}
In fact, we have already used this estimate in the proof of Theorem~\ref{T3}. We shall investigate under which conditions this bound can be reversed with some positive multiplicative constant. The main result established  in this section is the following
 statement.
\begin{lemma}\label{lem12++}
Let $\alpha\in\R\setminus\Q$ and $(q_\ell)_{\ell\ge0}$ be the sequence of denominators of the principal convergents of $\alpha$. Let $N\in\N$ and $\ve>0$ such that $0<2\ve<\|q_2\alpha\|$. Suppose that
\begin{equation}\label{ghj}
\frac{1}{2\ve}\le q_{\ell}\le N
\end{equation}
for some integer $\ell$. Then
\begin{equation}\label{vbe}
 \lfloor  \ve  N\rfloor  \ \le \  \#N(\alpha,\ve) \ \le \  32\, \ve N\,.
\end{equation}
\end{lemma}

\begin{proof}
Without loss of generality,  we will assume that $\ell$ is the smallest integer satisfying \eqref{ghj}.
Since $2\ve<\|q_2\alpha\|$, by \eqref{vb7} and \eqref{ghj}, we have that $q_{\ell}>q_2$, so $\ell\ge3$.
By the minimality of $\ell$,
\begin{equation}\label{vb+1231}
q_{\ell-1}< \frac{1}{2\ve}\le q_{\ell}.
\end{equation}
Let $n\in\N$ satisfy the condition $\|n\alpha\|<\ve$ and let $m$ be the same as in Lemma~\ref{contfrac.lemhombnd.lemhombnd}.
Then, by \eqref{vb7} and \eqref{hom2}, we obtain that
\begin{equation}\label{hom2+}
\frac{c_{m+1}-1}{2q_{m+1}}+\frac{a_{m+2}-c_{m+2}}{2q_{m+2}}\ \le \ \|n\alpha\| \ \le \ \frac{c_{m+1}+1}{q_{m+1}}.
\end{equation}
Since, by definition, $c_{m+1}\neq0$, condition \eqref{contfrac.eqn.Ostrowski1+2} implies that $a_{m+2}-c_{m+2}\ge1$. Hence, by \eqref{hom2+}, we obtain that
$$
\frac{1}{2q_{m+2}}\ \le \ \|n\alpha\|<\ve
$$
and thus $q_{m+2}\ge 1/(2\ve)$. Since $\ell$ is defined to be the smallest integer satisfying $q_{\ell}\ge 1/(2\ve)$, we conclude that $m+2\ge \ell$. Consider the following  three cases.

\noindent\textbf{Case (1):} $m\ge \ell$. Then $n$ lies in the set $A_N(d_1,\ldots,d_{\ell})$ with $d_1=\dots=d_{\ell}=0$. By Lemma~\ref{cortogapslem1},
$$
\#A_N(\underbrace{0,\ldots,0}_{\ell}\,)\ \le \ \frac{3N}{q_{\ell}}+1~\stackrel{\eqref{vb+1231}}{\le}~ 6\,\ve N+1\,.
$$
By \eqref{ghj}, we have that $2\ve N\ge1$. Therefore, the number of $n\in N(\alpha,\ve)$ that correspond to this case is
$$
\le 8\ve N\,.
$$

\noindent\textbf{Case (2):} $m+1=\ell$. Then $q_{m+1}=q_\ell\ge 1/(2\ve)$ and consequently,
by \eqref{hom2+} and the assumption $\|n\alpha\|<\ve $, we have that
\begin{equation}\label{x2}
c_{m+1}< 1+2\ve q_{m+1}\le 4\ve q_{m+1}.
\end{equation}
It follows that each $n\in N(\alpha,\ve)$ with $m+1=\ell$, must lie in
$A_N(d_1,\ldots,d_{m+1})$ with $d_1=\dots=d_{m}=0$ and $d_{m+1}=c_{m+1}$, for some $c_{m+1}$ satisfying \eqref{x2}.
By Lemma~\ref{cortogapslem1},
$$
\#A_N(\underbrace{0,\ldots,0}_{\ell-1}, c_{m+1}) \ \le \ \frac{3N}{q_{m+1}}+1\,.
$$
Since $q_{m+1}=q_{\ell}<N$, we have that
$$
\#A_N(\underbrace{0,\ldots,0}_{\ell-1}, c_{m+1}) \ \le \frac{4N}{q_{m+1}}\,.
$$
Therefore, using \eqref{x2}, we conclude that the number of $n\in N(\alpha,\ve)$ that correspond to this case is
$$
\le
\frac{4N}{q_{m+1}}\cdot4\ve q_{m+1}=16\,\ve N\,.
$$

\noindent\textbf{Case (3):} $m+2=\ell$. Then, by \eqref{vb+1231}, we have that
$$
q_{m+1}=q_{\ell-1}< 1/(2\ve)\le q_\ell=q_{m+2}.
$$
Consequently, by \eqref{hom2+} we obtain that $c_{m+1}=1$ and
$$
\frac{a_{m+2}-c_{m+2}}{2q_{m+2}}\ \le \ \|n\alpha\| \ < \ve\,.
$$
Thus, we have that
$$
q_{m+2} > \frac1{2\ve}\qquad\text{and}\qquad
{a_{m+2}-c_{m+2}}<2\ve{q_{m+2}}\,.
$$
It follows that each $n\in N(\alpha,\ve)$ with $m+2=\ell$ must lie in $A_N(d_1,\ldots,d_{m+2})$ with $d_1=\dots=d_{m}=0$, $d_{m+1}=c_{m+1}=1$ and $d_{m+2}=c_{m+2}$.
By Lemma~\ref{cortogapslem1},
$$
\#A_N(\underbrace{0,\ldots,0}_{\ell-2},1,c_{m+2}) \ \le \ \frac{3N}{q_{m+2}}+1\le \frac{4N}{q_{m+2}}.
$$
Therefore, we conclude that the number of $n\in N(\alpha,\ve)$ that correspond to this case is
$$
\le
\frac{4N}{q_{m+2}}\times 2\ve q_{m+2}= 8\ve N\,.
$$

On summing the estimates obtained in  each of the above three cases yields  the desired  result.
\end{proof}

\bigskip

The following corollary of Lemma~\ref{lem12++} is phrased in terms of the  exponent of approximation $w(\alpha)$ of $\alpha \in \R$,  -- see \eqref{cors} for the definition. Recall that $\mathcal{L}$ denotes the set of Liouville numbers; that is, the set of real numbers $\alpha$ such that $w(\alpha)=\infty$.

\begin{corollary}\label{lem12}
Let $\alpha\not\in\mathcal{L}\cup\Q$ and let $\nu\in\R$ satisfy the inequalities
\begin{equation}\label{e10.5}
0<\nu<\frac{1}{w(\alpha)}\,.
\end{equation}
Then, there exists a constant $\ve_0=\ve_0(\alpha)>0$ such that for any sufficiently large $N$ and any $\ve$ with $N^{-\nu}<\ve<\ve_0$, estimate \eqref{vbe} is satisfied.
\end{corollary}

\begin{proof}
We will assume that $2\ve_0<\|q_2\alpha\|$. Let $\ell$ be the smallest integer such that $(2\ve)^{-1}\le q_{\ell}$. In particular, we have that $\ell\ge3$ and \eqref{vb+1231} is satisfied.
By Lemma~\ref{Liouvillelem1} and the condition $w(\alpha)<1/\nu$, we have that $q_{m+1}\le q_m^{1/\nu}$ for all sufficiently large $m$. In particular, for sufficiently small $\ve_0$ the parameter $\ell$ will be sufficiently large and hence we will have that $q_{\ell}\le q_{\ell-1}^{1/\nu}$. Then, using \eqref{vb+1231}, we readily verify that
$$
q_{\ell}\le q_{\ell-1}^{1/\nu}< 1/(2\ve)^{1/\nu}< 1/\ve^{1/\nu}<N\,.
$$
Since, by the definition of $\ell$, we also have that $(2\ve)^{-1}\le q_{\ell}$, the inequalities associated with \eqref{ghj} are satisfied. Therefore, Lemma~\ref{lem12++} is applicable and the conclusion of  the corollary follows.
\end{proof}

\bigskip

Now we discuss the situation not covered by Lemma~\ref{lem12++}, namely  when \eqref{ghj} is not satisfied. To begin with, we give another `trivial' lower bound for $\#N(\alpha,\ve)$. Suppose that $K$ is the largest integer such that $q_K\le N$. Then $q_{K+1}>N$ and, by \eqref{vb7}, we have that
$$
\|q_K\alpha\|\le |D_K|<1/q_{K+1}\,.
$$
Observe that for any positive integer $s$ we trivially have that $\|sq_K\alpha\|<s\|q_K\alpha\|$. Hence, as long as $sq_K\le N$ and $s/q_{K+1}\le \ve$, we have that $sq_K\in N(\alpha,\ve)$. Hence,
\begin{equation}\label{vbm}
\#N(\alpha,\ve)\ge \min\big\{\lfloor\ve q_{K+1}\rfloor,\lfloor N/q_K\rfloor\big\}\,.
\end{equation}
The next result shows that, up to a constant multiple,  this rather trivial lower bound combined with the other trivial lower bound given by \eqref{vbf} is  best possible.

\begin{lemma}\label{lem12+++}
Let $\alpha\in\R\setminus\Q$ and $(q_\ell)_{\ell\ge0}$ be the sequence of denominators of the principal convergents of  $\alpha$, $N\in\N$ and $\ve>0$ such that $N^{-1}<2\ve<\|q_2\alpha\|$. Let
$K$ be the largest non-negative integer such that $q_K\le N$, and let
\begin{equation}\label{M}
M:=
\max\left\{\ve N,\min\Big\{\ve q_{K+1},\frac{N}{2q_K}\Big\}\right\}=
\min\left\{\ve q_{K+1},\max\Big\{\ve N ,\frac{N}{2q_K}\Big\}\right\}.
\end{equation}
Then
\begin{equation}\label{vbp}
\lfloor M\rfloor ~\le~   \#N(\alpha,\ve) ~\le~ 32M\,.
\end{equation}
\end{lemma}

\begin{proof}
The equality in \eqref{M} between the min-max and the max-min expressions is a consequence of the fact that $\ve N<\ve q_{K+1}$, which follows from the definition of $K$. The lower bound in \eqref{vbp} is a consequence of \eqref{vbf} and \eqref{vbm}. Thus, it remains to prove the upper bound. Suppose that $q_K<1/(2\ve)$. Then $\ve<\ve':=1/(2q_K)$ and we have that $N(\alpha,\ve)\subset N(\alpha,\ve')$. Clearly, Lemma~\ref{lem12++} is applicable to $N(\alpha,\ve')$ and thus it follows  that
$$
\#N(\alpha,\ve)\le 32\ve' N= 32N/2q_K.
$$
Again, by Lemma~\ref{lem12++}, in the case $q_K\ge1/(2\ve)$ we have that $\#N(\alpha,\ve)\le32\ve N$. Hence
$$
\#N(\alpha,\ve)\le32\max\{\ve N ,N/2q_K\}\,.
$$
Next, let $N'=q_{K+1}$. Note that $N'>N$ and therefore
$N(\alpha,\ve)\subset N'(\alpha,\ve)$.
Also condition \eqref{ghj} is satisfied with $\ell=K+1$. Hence, it follows via  Lemma~\ref{lem12++} that
$$
\#N(\alpha,\ve)\le \#N'(\alpha,\ve)\le 32\ve N'= 32q_{K+1}\ve.
$$
Thus
$$
\#N(\alpha,\ve)\le 32\min\big\{\ve q_{K+1},~\max\{\ve N ,N/2q_K\}\big\}\,.
$$
This is the upper bound in \eqref{vbp} and thereby completes the proof of the lemma.
\end{proof}

\subsection{The inhomogeneous case}

We prove a statement that relates the cardinality of  the inhomogeneous set  $N_\gamma(\alpha,\ve)$ to that of the homogeneous set $N(\alpha,\ve)$. The upshot is that the estimates obtained in the previous section can be exploited to provide estimates for $\#N_\gamma(\alpha,\ve)$.

\begin{lemma}\label{inhom2}
For any $\ve>0$ and $N\in\N$,  we have that
\begin{equation}\label{vbv1}
\#N_\gamma(\alpha,\ve) ~\le~ \#N(\alpha,2\ve)+1.
\end{equation}
Furthermore, if $N'_\gamma(\alpha,\ve')\neq\emptyset$, where $N':=\tfrac12N$ and $\ve':=\tfrac12\ve$, then
\begin{equation}\label{vbv2}
\#N_\gamma(\alpha,\ve) ~\ge~ \#N'(\alpha,\ve')+1\,.
\end{equation}
\end{lemma}

\begin{proof}
First we prove the upper bound \eqref{vbv1}. Suppose that $\#N_\gamma(\alpha,\ve)\ge2$, as otherwise there is nothing to prove. Let $n_0$ be the smallest element of $N_\gamma(\alpha,\ve)$. Then for any $n\in N_\gamma(\alpha,\ve)$ such that $n>n_0$ we have that
$$
\|(n-n_0)\alpha\|=\|(n\alpha-\gamma)-(n_0\alpha-\gamma)\|\le\|n\alpha-\gamma\|+\|n_0\alpha-\gamma\|<2\ve\,.
$$
Hence, $n-n_0\in N(\alpha,2\ve)$ and \eqref{vbv1} immediately follows.

Now we prove the lower bound \eqref{vbv2}. In view of the assumptions of the lemma, fix any $n_1\in N'_\gamma(\alpha,\ve')$. Then, for any $n\in N'(\alpha,\ve')$ we obtain that
$$
\|(n_1+n)\alpha-\gamma\|\le\|n_1\alpha-\gamma\|+\|n\alpha\|<\ve'+\ve'=\ve.
$$
Also, note that $1\le n_1+n\le N$. Therefore,for any $n\in N'(\alpha,\ve')$ we have that $n_1+n\in N_\gamma(\alpha,\ve)$. The lower bound  \eqref{vbv2} now readily follows.
\end{proof}

\section{Establishing the homogeneous results on  sums of reciprocals}\label{sec.thm3.ub}

In this section we prove Theorems  \ref{T1} and  \ref{T2}. Before beginning the proofs, we establish various useful auxiliary inequalities. As before, $q_k=q_k(\alpha)$ denotes the denominators of the principal convergents of $\alpha$, and $K=K(N,\alpha)$ denotes the largest non-negative integer such that $q_K\le N$, so
\begin{equation}\label{K(N)}
    q_K\le N < q_{K+1}\,.
\end{equation}
By (\ref{zvb}) and (\ref{K(N)}), we also have that
\begin{equation}\label{K(N)2}
    K-1\le 2\log N/\log2\,.
\end{equation}
Furthermore,  it is easily verified that
\begin{equation}\label{enq2}
\max\{\log(\tfrac{a_1}2\cdots \tfrac{a_K}2),K\}\ge \tfrac13\log q_K\,.
\end{equation}
Indeed, by (\ref{zvb}), we have that
\begin{equation*}
q_K\le 2^{K-1}\cdot a_1\cdots a_K<4^{K}\cdot \tfrac{a_1}2\cdots \tfrac{a_K}2\,.
\end{equation*}
On taking  logarithms,  we obtain that
\begin{equation*}
\log q_K \le K\log 4+\log(\tfrac{a_1}2\cdots \tfrac{a_K}2)\le
2\log 4\max\{K,\log(\tfrac{a_1}2\cdots \tfrac{a_K}2)\}\,
\end{equation*}
whence (\ref{enq2}) follows.

\subsection{Proof of Theorem \ref{T2}}

\subsubsection{The upper bound in  (\ref{eq1})}

For each pair of integers $(u,a)$ satisfying $1\le u\le K+1$ and $1\le a\le a_{u+1}$, let $B_N(u,a)$ be the set of integers $n$ with $1\le n\le N$ such that the Ostrowski coefficients $\{c_{k+1}\}_{k=0}^{\infty}$ of $n$ satisfy
$$
c_{1}=\dots=c_{u-1}=0
$$
and one of the following two conditions:
\begin{enumerate}
\item[(i)] $c_u=0$ and $c_{u+1}=a+1$, or
\item[(ii)] $c_u=1$ and $c_{u+1}=a_{u+1}-a$.
\end{enumerate}
In view of the uniqueness of Ostrowski expansion, the sets $B_N(u,a)$ are pairwise disjoint. By definition, we always have that $B_N(u,a)\subseteq\{1,\dots,N\}$. We claim that
\begin{equation}\label{claim1}
    \big\{n\in\N: n\le N,~\|n\alpha\|<|D_2|\,\big\}~\subset~\bigcup_{u=2}^{K+1}\ \bigcup_{a=1}^{a_{u+1}} B_N(u,a)\,.
\end{equation}
To see this, take any $n$ from the l.h.s. of \eqref{claim1} and, with reference to the Ostrowski expansion of $n$, let $m$ be the smallest integer such that $c_{m+1}\not= 0$. By (\ref{K(N)}) and Lemma~\ref{contfrac.lemhombnd.lemhombnd}, we have that $2\le m\le K$. If $c_{m+1}>1$ then $n$ lies in the set $B_N(m,c_{m+1}-1)$. If $c_{m+1}=1$ then $n$ lies in the set $B_N(m+1,a_{m+2}-c_{m+2})$. Note that, by (\ref{contfrac.eqn.Ostrowski1+2}) and (\ref{contfrac.eqn.Ostrowski1+3}), in this case $1\le a_{m+2}-c_{m+2}\le a_{m+2}$.

By Lemma \ref{contfrac.lemhombnd.lemhombnd}, if $n$ lies in the l.h.s. of \eqref{claim1} then
\begin{equation}\label{bndB_N(m,a)}
\|n\alpha\|\ge a|D_u| \quad\text{ whenever }~n\in B_N(u,a)
\end{equation}
for some $u$ and $a$ from the r.h.s. of \eqref{claim1}.
Clearly, the set $B_N(u,a)$ is the union of
\begin{equation}\label{vbx1}
A_N(\underbrace{0,\dots,0}_{u\text{ times}},a+1) \qquad\text{and}\qquad
A_N(\underbrace{0,\dots,0}_{u-1\text{ times}},1,a_{u+1}-a)\,,
\end{equation}
where the sets $A_N(\cdots)$ are defined in \S\ref{gaps}.
Then, by Lemma \ref{cortogapslem1} and the inequalities given by  (\ref{K(N)}), we find that
\begin{equation}\label{bound3}
\#B_N(u,a)\le \frac{8N}{q_{u+1}}\qquad\text{if }u\le K-1\,.
\end{equation}
Putting this together with (\ref{bndB_N(m,a)}) implies that \\[1ex]
\begin{align}
\sum_{u=2}^{K-1}\sum_{a=1}^{a_{u+1}}\sum_{n\in B_N(u,a)}\frac{1}{\|n\alpha\|} \ & \le \
\sum_{u=2}^{K-1}\sum_{a=1}^{a_{u+1}}\frac{8N}{aq_{u+1}|D_u|} \nonumber \\[2ex]
 & \stackrel{\eqref{vb7}}{\le} \ 16N\sum_{u=2}^{K-1}\sum_{a=1}^{a_{u+1}}\frac{1}{a}\nonumber \\[2ex]
&\le  16N\sum_{u=2}^{K-1}(1+\log a_{u+1}) \nonumber \\[2ex]
 &\le  16N(K-2+\log (a_1\dots a_{K})) \nonumber \\[2ex]
&\stackrel{\eqref{zvb}}{\le}  64N\log q_K\,. \label{Tmp1}
\end{align}

\medskip

\noindent Considering the case $\|n\alpha\|\ge|D_2|$, we get that
\begin{equation}\label{zvb3}
\sum_{\substack{n\le N\\[0.5ex] \|n\alpha\|\ge |D_2|}}\frac{1}{\|n\alpha\|}\le
N|D_2|^{-1}~\stackrel{\eqref{vb7}}{\le}~2q_3N\,.
\end{equation}
It remains to note that if
$$
n\in \bigcup_{u=K}^{K+1} \bigcup_{a=1}^{a_{u+1}} B_N(u,a)
$$
then
\begin{equation}\label{vbt}
n\equiv0~\text{or}~q_{K-1}\, ({\operatorname{mod}{q_K})}\,.
\end{equation}
Indeed, if $n\in B_N(K,a)$ then, since $q_{K+1}>N\ge n$, by definition, the Ostrowski expansion of $n$ is either $(a+1)q_K$ or $q_{K-1}+(a_{K+1}-a)q_K$. In both cases \eqref{vbt} is satisfied. Similarly, if $n\in B_N(K+1,a)$ then, by definition, the Ostrowski expansion of $n$ can only be $q_{K}+(a_{K+2}-a)q_{K+1}$ and since $n\le N<q_{K+1}$, we have that $n=q_K$. Clearly, \eqref{vbt} is satisfied again. The upshot is that
the $n$'s which fall into $B_N(K,a)$ or $B_N(K+1,a)$ are irrelevant in estimating \eqref{eq1} and therefore on combining (\ref{claim1}), \eqref{Tmp1} and (\ref{zvb3}) implies the upper bound in (\ref{eq1}).

\subsubsection{The lower bound in  (\ref{eq1})}

For each $m,a\in\N$, define $B^*_N(m,a)$ to be the collection of positive integers $n\le N$ with Ostrowski expansions satisfying $c_{k+1}=0$ for all $k<m$ and $c_{m+1}=a$. If $m\le K-1$ and $1\le a\le a_{m+1}$,  the set $B^*_N(m,a)$ is clearly non-empty. Then, by Lemma \ref{cortogapslem1}, it follows  that
\begin{equation}\label{eq5}
\#B^*_N(m,a)\ge \frac{N}{3q_{m+1}}\,.
\end{equation}
Assume that $n\in B^*_N(m,a)$ and that  one of the following is satisfied

$\bullet$ ~~$1\le a\le a_{m+1}$ and $m\le K-2$, or

$\bullet$ ~~$2\le a\le a_{m+1}$ and $m=K-1$.

\noindent Then it is easily seen that $n\not\equiv0,q_{K-1}\,(\operatorname{mod}\,q_K)$. Therefore, the sum in (\ref{eq1}) is bounded below by
$$
X:=\sum_{m=2}^{K-2}\sum_{a=1}^{a_{m+1}}\sum_{n\in B^*_N(m,a)}\frac{1}{\|n\alpha\|}+\sum_{a=2}^{a_{K}}\sum_{n\in B^*_N(K-1,a)}\frac{1}{\|n\alpha\|}\,.
$$
By Lemma~\ref{contfrac.lemhombnd.lemhombnd}, $\|n\alpha\|\le(a+1)|D_m|$ for every $n\in B^*_N(m,a)$ with $m\ge2$.
Then this together with  (\ref{eq5}) implies that
\begin{equation}\label{Tmp2}
\begin{array}[b]{rcl}
\displaystyle X & \ge &
\displaystyle \sum_{m=2}^{K-2}\sum_{a=1}^{a_{m+1}}\frac{N}{3q_{m+1}(a+1)|D_m|}+\sum_{a=2}^{a_{K}} \frac{N}{3q_{K}(a+1)|D_{K-1}|} \\[4.5ex]
 &\stackrel{\eqref{vb7}}{\ge} & \displaystyle \tfrac13N\left(\sum_{m=2}^{K-2}\sum_{a=1}^{a_{m+1}}+\sum_{a=2}^{a_{K}}\right)\frac{1}{a+1}
 \ = \ \tfrac13N\sum_{m=2}^{K-1}\sum_{a=1}^{a_{m+1}}\frac{1}{a+1}-\tfrac16N\\[4ex]
&\ge & \displaystyle \tfrac13N\sum_{m=2}^{K-1}\log \tfrac{a_{m+1}+2}2-\tfrac16N\\[4ex]
& \ge & \displaystyle  \tfrac13N\sum_{m=0}^{K-1}\log \tfrac{a_{m+1}+2}2-\tfrac13N(\log \tfrac{a_{1}+2}2+\log \tfrac{a_{2}+2}2)-\tfrac16N\,.
\end{array}
\end{equation}
Since $3a_i\ge a_i+2>a_i$,
$$
\log \tfrac{3a_{m+1}}2\ge \log \tfrac{a_{m+1}+2}2\ge \max\{\log(3/2),\log \tfrac{a_{m+1}}2\},
$$
and we obtain from \eqref{Tmp2} that
\begin{equation}\label{Tmp2+}
\begin{array}[b]{rcl}
\displaystyle X & > & \tfrac13N\max\{K\log(3/2),\log(\tfrac{a_1}2\dots\tfrac{a_K}2)\}-\tfrac13N\log(9a_1a_2/4)-\tfrac16N.
\end{array}
\end{equation}
Trivially, $\tfrac13\log(3/2)>\tfrac18$. Then, using \eqref{zvb} and \eqref{enq2} we obtain from \eqref{Tmp2+} that
\begin{equation*}
\begin{array}[b]{rcl}
\displaystyle X & \ge & \tfrac{1}{24} N\log q_K-\tfrac13N\log q_2-\tfrac13N\log(9/4)-\tfrac16N\\[2ex]
 & \ge & \tfrac{1}{24} N\log q_K-(\tfrac13\log q_2+\tfrac12)N\,.
\end{array}
\end{equation*}
This completes the proof of (\ref{eq1}).

\subsubsection{Proof of (\ref{eq4}) and (\ref{eq2})} \label{5.4}

First observe that since $N\ge q_3$ we have that $K\ge 3$. Furthermore,
$1\le n_1\le N$ and $n_1\equiv0\,(\operatorname{mod}\,q_K)$ if and only if $n_1=aq_K$ for some $a$ with $1\le a\le N/q_K$. Since $N<q_{K+1}=a_{K+1}q_K+q_{K-1}$, we also have that $a\le a_{K+1}$. Thus, the expression $n_1=aq_{K}$ is the Ostrowski expansion of such an integer $n_1$ and then, by  Lemma~\ref{contfrac.homoglem.homoglem} (with $m=K\ge3$), we have that
$$
\|n_1\alpha\|=a|D_K|\,.
$$
Similarly, $1\le n_2\le N$ and $n_2\equiv q_{K-1}\,(\operatorname{mod}\,q_K)$ if and only if $n_2=q_{K-1}+(a_{K+1}-b)q_K$ for some $b$ with $a_{K+1}-(N-q_{K-1})/q_K\le b\le a_{K+1}$. Again since $N<a_{K+1}q_K+q_{K-1}$, we have that $a_{K+1}-(N-q_{K-1})/q_K>0$, which implies that $b\ge1$. By Lemma~\ref{contfrac.homoglem.homoglem} (with $m=K-1\ge2$), we have that
$$
\|n_2\alpha\|=|D_{K-1}|-(a_{K+1}-b)|D_K|\stackrel{\eqref{vb9}}{=}|D_{K+1}|+b|D_K|\,.
$$
By (\ref{vb7}),
\begin{equation*}\label{eq13}
\frac{q_{K+1}}{a}\le \frac{1}{\|n_1\alpha\|}\le \frac{2q_{K+1}}{a}\qquad\text{and}\qquad \frac{q_{K+1}}{b+1}\le \frac{1}{\|n_2\alpha\|}\le \frac{2q_{K+1}}{b}\,.
\end{equation*}
Consequently,
\begin{align*}
\sum_{\substack{1\le n\le N\\[0.2ex] n\equiv0,\,q_{K-1}\, ({\operatorname{mod}{q_K})}}}\hspace*{-4ex}\frac{1}{\|n\alpha\|} \ \ & \le & & \hspace*{-4ex} \sum_{1\le a\le N/q_K}\frac{2q_{K+1}}{a} \ \  +
\sum_{a_{K+1}-(N-q_{K-1})/q_K\le b\le a_{K+1}}\frac{2q_{K+1}}{b}\\[1ex]
& \le & & \hspace*{-4ex}  4q_{K+1}\sum_{1\le a\le 1+N/q_K}\frac1a\ \le\ 4q_{K+1}(1+\log(1+N/q_K))
\end{align*}
and
\begin{align*}
\sum_{\substack{1\le n\le N\\[0.2ex] n\equiv0,\,q_{K-1}\, ({\operatorname{mod}{q_K})}}}\hspace*{-4ex}\frac{1}{\|n\alpha\|} \ \ \ & \ge & &  \hspace*{-10ex} \sum_{1\le a\le N/q_K}\frac{q_{K+1}}{a}\ \ge\ q_{K+1}\log(1+N/q_K)\,.
\end{align*}
These upper and lower bound inequalities prove (\ref{eq4}). Similarly,
\begin{align}
\nonumber \sum_{\substack{1\le n\le N\\[0.2ex] n\equiv0,\,q_{K-1}\, ({\operatorname{mod}{q_K})}}}\hspace*{-4ex}\min\left\{cN,\frac{1}{\|n\alpha\|}\right\} \ \  &  \le\ 2\sum_{1\le a\le 2N/q_K} \min\left\{cN,\frac{2q_{K+1}}a\right\}\\[0ex]
&  \le\ 2\sum_{1\le a\le 2N/q_K} \min\left\{cN,\frac{4a_{K+1}q_K}a\right\}\,.\label{vbvb1}
\end{align}
Here we have  used the fact that $q_{K+1}=a_{K+1}q_K+q_{K-1}\le 2a_{K+1}q_K$. Now split the sum on the r.h.s. of  (\ref{vbvb1}) into two sub-sums: one with $a\le 2(a_{K+1}/c)^{\frac12}$ and the other with $a> 2(a_{K+1}/c)^{\frac12}$. The  first sub-sum is trivially bounded by
$$
2\cdot c\,N\cdot\,2\,(a_{K+1}/c)^{\frac12}= 4\,N\,(ca_{K+1})^{\frac12}
$$
while the second sub-sum is bounded by
$$
2\cdot\frac{2N}{q_K}\cdot \frac{4a_{K+1}q_K}{2(a_{K+1}/c)^{\frac12}}= 8\,N\,(ca_{K+1})^{\frac12}\,.
$$
On combining  the previous  two estimates with  (\ref{vbvb1}) we obtain (\ref{eq2}).

\subsection{Proof of Theorem~\ref{T1}}

Let $N$ be a sufficiently large integer and $K=K(N,\alpha)$ be as in statement of Theorem~\ref{T1}.
By the definition of $K$, for any integer $0\le i\le K$ we have that $q_i\le N$. By \eqref{vb7}, we have that $q_{i+1}\|q_i\alpha\|<1$ and since $q_{i+1}=a_{i+1}q_i+q_{i-1}>a_{i+1}q_i$ ($q_{-1}:=-1$), it follows  that
$a_{i+1}q_{i}\|q_i\alpha\|<1$. Hence
$$
S_N(\alpha,0)\ge
\sum_{i=0}^K\frac{1}{q_i\|q_i\alpha\|}\ \ge \ \sum_{i=0}^K a_{i+1}=A_{K+1}\,.
$$
Combining this with Corollary~\ref{coro2++} establishes the lower bound in \eqref{eq12}.  It remains to prove  the upper bound.

By the partial summation formula \eqref{psum1}, for any function $g: \N \to \R$, we have that
\begin{equation}\label{psum}
\sum_{n=1}^N\frac{g(n)}{n}= \sum_{n=1}^N\frac{G(n)}{n(n+1)}
+\frac{G(N)}{N+1}   \ \qquad \text{where } \quad ~G(n):=\sum_{k=1}^ng(k)  \, .
\end{equation}
Now consider the specific function $g$ defined as follows: given an integer $k\ge0$, for $q_{k}\le n<q_{k+1}$, let
\begin{equation*}
g(n):=\left\{\begin{array}{ccl}
 \|n\alpha\|^{-1}  & \text{if} & \hspace*{-2ex} n\not\equiv0,\,q_{k-1}\, (\operatorname{mod}{q_k})\\[2ex]
0 & \quad\text{otherwise}\, ,  &
\end{array}
\right.
\end{equation*}
where $q_{-1}:=0$. Theorem~\ref{T2} implies that
\begin{equation}\label{dfdf}
G(m)\le 64m\log m+O(m)\,,
\end{equation}
where the implied constant depends only on $\alpha$. Indeed, to see that this is so,  let $M$ be the largest integer such that $q_M\le m$. Then for $n<q_{M}$,
we have that $n\not\equiv 0\pmod{q_{M}}$ while $n\equiv q_{M-1}\pmod{q_{M}}$ means that $n=q_{M-1}$ and so $g(n)=0$. Hence, on making use of  \eqref{eq1},  we have  that \begin{align*}
G(m) \ \ & = \ \sum_{k=0}^{M-1}\,\sum_{\substack{q_k\le n<q_{k+1}\\[0.5ex] n\not\equiv0,q_{k-1}\,(\mathrm{mod}\,{q_k})}}\frac{1}{\|n\alpha\|} \ \
+\sum_{\substack{q_M\le n\le m\\[0.5ex] n\not\equiv0,q_{M-1}\,(\mathrm{mod}\,{q_M})}}\frac{1}{\|n\alpha\|} \\[1ex]
& \le \sum_{\substack{1\le n\le m\\[0.5ex] n\not\equiv0,q_{M-1}\,(\mathrm{mod}\,{q_M})}} \!\!\frac{1}{\|n\alpha\|} \ \le  \ 64 m\log q_M +O(m)
\\[1ex]
&
\le  \ 64 m\log m +O(m) \,
 \end{align*}
which is precisely \eqref{dfdf}. On combining \eqref{psum} and \eqref{dfdf}, it follows that
\begin{eqnarray}\label{eq10}
\sum_{\substack{1\le n \le N\\g(n) \neq 0}}\frac{1}{n\|n\alpha\|} &  =  &  \sum_{n=1}^N \frac{g(n)}{n} \ \le \  64\sum_{n=1}^N\frac{\log n}{n}
+O(\log N) \nonumber \\[1ex]
& \le &   32(\log N)^2+O(\log N)\,.
\end{eqnarray}
It remains to consider the sum
$$
\sum_{\substack{1\le n \le N\\g(n)=0}}\frac{1}{n\|n\alpha\|}\,.
$$

We will use similar  arguments  to those appearing in  \S\ref{5.4}. Fix any integer $k\ge 3$. A positive integer $n_1\in[q_k,q_{k+1})$ satisfies $n_1\equiv0\,(\operatorname{mod}\,q_k)$ if and only if $n_1=aq_k$ for some $a$ with $1\le a\le (q_{k+1}-1)/q_k$. Similarly, a positive integer $n_2\in[q_k,q_{k+1})$ satisfies $n_2\equiv q_{k-1}\,(\operatorname{mod}\,q_k)$ if and only if $n_2=q_{k-1}+(a_{k+1}-b)q_k$ for some $b$ with $a_{k+1}-(q_{k+1}-q_{k-1}-1)/q_k\le b< a_{k+1}$. It is easily seen that $b\ge1$. Then, by Lemma~\ref{contfrac.homoglem.homoglem}, we have that
$$
\|n_1\alpha\|=a|D_k|
$$
and
$$
 \|n_2\alpha\|=|D_{k-1}|-(a_{k+1}-b)|D_k|\stackrel{\eqref{vb9}}{=}|D_{k+1}|+b|D_k|\,.
$$
By (\ref{vb7}), we have that
$$
\frac{1}{\|n_1\alpha\|}\le \frac{2q_{k+1}}{a}\qquad\text{and}\qquad
\frac{1}{\|n_2\alpha\|}\le \frac{2q_{k+1}}{b}\,.
$$
These together with the inequality $q_{k+1}<2a_{k+1}q_k$,  imply that
$$
\frac{1}{\|n_1\alpha\|}\le \frac{4a_{k+1}q_k}{a}
\qquad\text{and}\qquad
\frac{1}{\|n_2\alpha\|}\le \frac{4a_{k+1}q_k}{b}\,.
$$
Consequently, it follows that
\begin{align*}
\sum_{\substack{q_k\le n<q_{k+1}\\[0.2ex] g(n)=0}}\frac{1}{n\|n\alpha\|} \ \ & \le \ \sum_{a\ge1}\frac{4a_{k+1}q_k}{q_ka^2}+
\sum_{1\le b\le a_{k+1}-1}\frac{4a_{k+1}q_k}{b(a_{k+1}-b)q_k}\\[0ex]
& = \ 4a_{k+1}\sum_{a\ge1}\frac1{a^2}+4\sum_{1\le b\le a_{k+1}-1}\left(\frac1b+\frac{1}{a_{k+1}-b}\right)\\
& = \ 4a_{k+1}\sum_{a\ge1}\frac1{a^2}+8\sum_{1\le b\le a_{k+1}-1}\frac1b\\
&\le\ \frac{2\pi^2}3a_{k+1}+1+\log a_{k+1}< 9a_{k+1}\,.
 \end{align*}
Hence, for $N$ (and therefore $K$)  sufficiently large we have that
$$
\sum_{\substack{1\le n\le N\\[0.2ex] g(n)=0}}\frac{1}{n\|n\alpha\|} \  \le \  \ \sum_{k=0}^K  \ \sum_{\substack{q_k\le n<q_{k+1}\\[0.2ex] g(n)=0}}\frac{1}{n\|n\alpha\|} \ \le \ 9A_{K+1}+O(1)  \ \le  \ 10A_{K+1}  \, .
$$
This together with (\ref{eq10}) establishes  the upper bound appearing in  (\ref{eq12}) and thus completes the proof of Theorem~\ref{T1}.

\section{Establishing the inhomogeneous results on  sums of reciprocals}

In this section we prove Theorems \ref{T4}, \ref{T5}, \ref{T6} and \ref{T8}. Since $\|x\|$ is invariant under integer translations of $x$, without loss of generality, we can assume that $\alpha \in [0,1)\setminus \Q$  and  that $\gamma\in [-\alpha,1-\alpha)$.
Throughout,  we will use the notation and language introduced in \S\ref{sec.contfrac} -- \S\ref{gaps}.
Also, we assume that (\ref{contfrac.eqn.|na-g|}) holds.

\subsection{Proof of Theorem~\ref{T6}}\label{inhom}

Take $0<\ve<1$ and let
$$
\cA:=\Big\{\alpha\in[0,1)\setminus\Q:\log(a_1\cdots a_n)\ll n,\ \sum_{i=1}^na_i\ll n^{1+\ve}\text{ for all } n\in\N\Big\}\,,
$$
where the implied constants may depend on $\alpha$. It follows from the Khintchine-Levy Theorem \cite[Chapter~V]{RockettSzusz1992}, Khintchine's Theorem \cite[Theorem~2.2]{Harman1998}, and a theorem of Diamond and Vaaler \cite{DiamondVaaler1986} that the set $\cA$ has Lebesgue measure one. We shall show that Theorem~\ref{T6} holds with this particular choice of $\cA$.

In order to estimate the relevant sum we will make use of expression
\eqref{vb6A} for $\|n\alpha-\gamma\|$.  Namely,
$$
\|n\alpha-\gamma\|=\min\Big\{\left|\Sigma\right|,1-|\Sigma|\Big\}\,,
$$
where $\Sigma=\Sigma(n,\alpha,\gamma)$ is given by \eqref{Sigma} and $m=m(n,\alpha,\gamma)$ is  defined as in Lemma~\ref{contfrac.inhomoglem.inhomoglem}.
We will use the estimates for $|\Sigma|$ and $1-|\Sigma|$ given by Lemma~\ref{contfrac.leminhombnd.leminhombnd} and Lemma~\ref{contfrac.leminhombnd.leminhombnd2}.

\subsubsection{The case $m>K$}

Unlike our treatment of the corresponding homogeneous case, where $m$ was always bounded by $K$, in the inhomogeneous case the parameter $m$ (which is determined by the $b_{k+1}$'s as well as by the $c_{k+1}$'s) can be arbitrarily large if the Ostrowski expansion of $\gamma$ contains a large number of consecutive zeros. More precisely, if $b_{K+1}=\dots=b_{K+T}=0$, $b_{K+T+1}\neq0$, then for $n=\sum_{k=0}^Kb_{k+1}q_k$ we will have that $m=K+T$. However, as is clear from the definition of $m$ and the coefficients $\delta_{k+1}$ given by \eqref{delta}, whenever $m>K$, we have that $c_{k+1}=b_{k+1}$ for all $k\le K$ and the associated integer $n$ is uniquely defined by $\gamma$. Hence
\begin{equation}\label{m>K}
\sum_{\substack{n\le N\,: ~m>K\\[0.5ex]
c\le n\|n\alpha-\gamma\|}}~\frac{1}{n\|n\alpha-\gamma\|}  \ \le \ \frac{1}{c}\,.
\end{equation}

\subsubsection{The case $|\Sigma|\le1/2$, $m\le K$}

By Lemma~\ref{contfrac.inhomoglem.inhomoglem}, in the case $|\Sigma|\le1/2$, we have that $\|n\alpha-\gamma\|=|\Sigma|$ and therefore we can call upon Lemma~\ref{contfrac.leminhombnd.leminhombnd} when required. With this in mind, for each pair of integers $(u,a)$ satisfying
\begin{equation}\label{am}
0\le u\le K+1\quad\text{and}\quad 1\le a\le 2a_{u+1}-1,
\end{equation}
let $C_N(u,a)$ be the set of all positive integers $n\le N$ such that $|\Sigma|\le 1/2$ and one of the following two conditions is satisfied:
\begin{itemize}
  \item[(i)] ~~$\delta_{1}=\dots=\delta_{u}=0$, ~$|\delta_{u+1}|-1=a$\,, or
  \item[(ii)] ~~$u\ge1$, ~$\delta_{1}=\dots=\delta_{u-1}=0$, ~$|\delta_u|=1$,
~$a_{u+1}-1-\sgn(\delta_{u})\delta_{u+1}=a$.
\end{itemize}
In the first case we have that $u=m$ while in the second case $u=m+1$, where $m=m(n,\alpha,\gamma)$ is defined as in Lemma~\ref{contfrac.inhomoglem.inhomoglem}. Since we are dealing with the case $m\le K$, the condition $u\le K+1$ within \eqref{am} is natural and non-restrictive. Then, applying
Lemmas~\ref{contfrac.inhomoglem.inhomoglem} and \ref{contfrac.leminhombnd.leminhombnd}, we obtain that
\begin{equation}\label{bound1}
\|n\alpha-\gamma\|\ge a|D_u|
\end{equation}
for any $n\in C_N(u,a)$.
It is easily seen that  $C_N(u,a)$ lies within the union of the following four sets:
\begin{equation}\label{sets1}
A_N(b_1,\dots,b_{u},b_{u+1}\pm(a+1))
\end{equation}
and
\begin{equation}\label{sets2}
A_N(b_1,\dots,b_{u-1},b_{u}\pm1,b_{u+1}\pm(a_{u+1}-a-1)),
\end{equation}
where $b_1,b_2,\dots$ are the Ostrowski coefficients of $\gamma$ (see Lemma~\ref{Ostrowski2}) and are fixed.

Let $M_N(u,a)$ be the collection of the minimal elements of the non-empty sets given by  (\ref{sets1}) and (\ref{sets2}). Trivially, we have that $\#M_N(u,a)\le4$. Further, using Lemma~\ref{contfrac.sumbndlem.sumbndlem}, we obtain
\begin{equation*}
\sum_{\substack{n\in C_N(u,a)\setminus M_N(u,a)}}\frac{1}{n}\ll\frac{\log N}{q_{u+1}}\,.
\end{equation*}
On combining this with (\ref{bound1}),  we obtain that
\begin{align}
\sum_{u=0}^{K+1}~&~\sum_{a=1}^{2a_{u+1}-1}  \sum_{\substack{n\in C_N(u,a)\\c\le n\|n\alpha-\gamma\|}}\frac{1}{n\|n\alpha-\gamma\|}  \nonumber\\[2ex]
& ~\le~
\sum_{u=0}^{K+1}~~\sum_{a=1}^{2a_{u+1}-1}\left(\sum_{n\in C_N(u,a)\setminus M_N(u,a)}\frac{1}{na|D_u|}
~+~
\sum_{\substack{n\in M_N(u,a)\\c\le n\|n\alpha-\gamma\|}}\frac{1}{n\|n\alpha-\gamma\|}\right) \nonumber\\[2ex]
&\ll \ \log N\cdot\sum_{u=0}^{K+1}\frac{1}{q_{u+1}|D_u|} \sum_{a=1}^{2a_{u+1}-1}\frac{1}{a}~+~ \sum_{u=0}^{K+1}~\sum_{a=1}^{2a_{u+1}-1}
\frac{4}{c} \nonumber\\[2ex]
&\ll \ \log N\cdot\sum_{u=0}^{K+1}~(1+\log a_{u+1})~+~ \sum_{u=0}^{K+1}~a_{u+1}
\, .\label{e1}
\end{align}
The last inequality is a consequence of \eqref{vb7} and the trivial fact that
$$
\sum_{a=1}^{2a_{u+1}-1}\frac1a~\ll 1+\log a_{u+1}\,.
$$
Recall that, by \eqref{zvb}, we have that $K\ll\log q_K\le \log N$ and, by the definition of $\cA$, we have that
$$
\sum_{u=0}^{K+1}~\log a_{u+1}=\log(a_1\cdots a_{K+2})\ll K+2\ll K\ll\log N,
$$
and
$$
\sum_{u=0}^{K+1}~a_{u+1}\ll K^{1+\ve}\ll (\log N)^{1+\ve}\,.
$$
Hence, the above two estimates together with \eqref{e1} imply that
\begin{equation}\label{x0}
  \sum_{u=0}^{K+1}~\sum_{a=1}^{2a_{u+1}-1}  \sum_{\substack{n\in C_N(u,a)\\c\le n\|n\alpha-\gamma\|}}\frac{1}{n\|n\alpha-\gamma\|}~\ll~(\log N)^2\,.
\end{equation}
Observe that this upper bound estimate is no bigger than the upper bound  appearing in Theorem~\ref{T6}. However, we are not yet done with the case under consideration,  since there may be integers $n$ with $1\le n\le N$ and $|\Sigma|\le1/2$ which do not fall in any of the sets $C_N(u,a)$. As is clear from the definition of $C_N(u,a)$ and the range of $(u,a)$ given by \eqref{am}, the remaining numbers $n$ correspond to the situation when the first two terms of \eqref{contfrac.leminhombnd.eqn1} vanish. In order to take these numbers into account, for each triple $(m,a,\ell)$ such that
$$
0\le m\le K,\quad 1\le \ell\le \max\{2,K-m+1\},\qquad 1\le a\le 2a_{m+2+\ell}\,,
$$
we define the set $C^+_N(m,a,\ell)$ to consist of integers $n$ with $1\le n\le N$ such that $|\Sigma|\le1/2$ and their Ostrowski expansions satisfy the following set of equalities:

\begin{equation}\label{c+}
\left\{\begin{array}{l}
\delta_1=\dots=\delta_m=0 \\[1ex]
|\delta_{m+1}|=1, \\[1ex] \sgn(\delta_{m+1})\delta_{m+2}=a_{m+2}-1,\\[1ex]
        \delta_{m+2+i}=(-1)^{i}\sgn(\delta_{m+1})a_{m+2+i} \qquad\text{($1\le i\le \ell-1$)}\\[1ex]
a_{m+2+\ell}-(-1)^{\ell}\sgn(\delta_{m+1})\delta_{m+2+\ell}=a\,.
\end{array}\right.
\end{equation}

\noindent Let $t=m+\ell+1$. Then we have that
\begin{equation}\label{Tt}
2\le t\le K+3\,.
\end{equation}
Since we are assuming  that $|\Sigma|\le1/2$, by Lemmas~\ref{contfrac.inhomoglem.inhomoglem} and \ref{contfrac.leminhombnd.leminhombnd}, we have that
\[\|n\alpha-\gamma\|\ge a\|q_{t}\alpha\| \qquad\text{for all }~n\in C^+_N(m,a,\ell).\]
In view of properties \eqref{contfrac.eqn.Ostrowski1+1}, \eqref{contfrac.eqn.Ostrowski1+2}, \eqref{contfrac.eqn.Ostrowski2+1} and \eqref{contfrac.eqn.Ostrowski2+}, the equalities associated with (\ref{c+}) can only occur if the coefficients $b_{m+1},\dots,b_{t}$ of the Ostrowski expansion of $\gamma$ have one of the following two patterns:\\[2ex]
\begin{equation}\label{pattern}
\begin{array}{|c|c|c|c|c|c|c|}
\hline
b_{m+1} & b_{m+2} & b_{m+3} & b_{m+4} &b_{m+5} &b_{m+6} & \dots  \\ \hline
<a_{m+1} & 0 & a_{m+3} & 0 & a_{m+5} & 0 &\dots \\ \hline
>0\ \ \ \  & a_{m+2}-1 & 0 & a_{m+4} &0 &a_{m+6} & \dots \\ \hline
\end{array}
\end{equation}

\medskip

\noindent
Therefore, for each  $t$ satisfying \eqref{Tt}, there are at most two corresponding pairs $(m,\ell)$ such that (\ref{c+}) is satisfied for some $n\le N$. In other words, for each such  $t$ there are at most two values of $(m,\ell)$ such that $C^+_N(m,a,\ell)$ is non-empty. Next, in view of  (\ref{c+}), observe that $C^+_N(m,a,\ell)$ is the union of two sets $A_N(d_1,\dots,d_{t+1})$ with\\[1ex]
\begin{equation*}
\left\{\begin{array}{l}
d_{k}=b_k \qquad \text{($1\le k\le m$)}\\[1.5ex]
d_{m+1}=b_{m+1}\pm1, \\[1.5ex]
d_{m+2}=b_{m+2}\pm(a_{m+2}-1),\\[1.5ex]
d_{m+2+i}=b_{m+2+i}\pm(-1)^{i}a_{m+2+i} \qquad\text{($1\le i\le \ell-1$)}\\[1.5ex]
d_{m+2+\ell}=b_{m+2+\ell}\pm(-1)^{\ell}(a_{m+2+\ell}-a)\,.
\end{array}\right.
\end{equation*}

\medskip

\noindent Then, on exploiting the same arguments used to establish (\ref{e1}) and \eqref{x0}, we obtain that
\begin{align}
\sum_{t=2}^{K+3}\sum_{\substack{(m,\ell)\\m+\ell+1=t}}&\sum_{a=1}^{2a_{t+1}}\sum_{\substack{n\in C^+_N(m,a,\ell)\\ c\le n\|n\alpha-\gamma\|}}\frac{1}{n\|n\alpha-\gamma\|}\nonumber\\[2ex]
&\ll\ (\log N)\sum_{t=2}^{K+3}(1+\log a_t)+\sum_{t=1}^{K+2}a_t ~\ll~ (\log N)^2\,.\label{e2}
\end{align}
Again, observe that this upper bound estimate is no bigger than the upper bound  appearing in Theorem~\ref{T6}.  On using Lemma~\ref{contfrac.leminhombnd.leminhombnd},  it can be readily verified that the sets $C_N(u,a)$ and $C^+_N(m,a,\ell)$ considered above account for all integers $n$ with $1\le n\le N$ such that $|\Sigma|\le1/2$ and $m=m(n,\alpha,\gamma)\le K$.

\subsubsection{The case $|\Sigma|>1/2$, $m\le K$}

 By Lemma~\ref{contfrac.inhomoglem.inhomoglem}, when $|\Sigma|>1/2$ we have that $\|n\alpha-\gamma\|=1-|\Sigma|$ and therefore we can call upon Lemma~\ref{contfrac.leminhombnd.leminhombnd2} when required.

 We consider two subcases. In subcase one, we assume that
\begin{equation}\label{c1}
\delta_1\not=(-1)^m\sgn(\delta_{m+1})(a_1-1).
\end{equation}
Then, by \eqref{contfrac.leminhombnd.eqn1zx}, $1-|\Sigma|\ge|D_0|=\{\alpha\}$ and
\begin{equation}\label{e++}
\sum_{\substack{1\le n\le N\\[0.5ex] |\Sigma|\ge1/2,~\eqref{c1}\text{~holds}}}\frac{1}{n\|n\alpha-\gamma\|}\ll\sum_{n=1}^N\frac{1}{n}\ll\log N\,.
\end{equation}
In subcase two, we assume that
\begin{equation}\label{c2}
\delta_1=(-1)^m\sgn(\delta_{m+1})(a_1-1).
\end{equation}
In view of properties \eqref{contfrac.eqn.Ostrowski1+1}, \eqref{contfrac.eqn.Ostrowski1+2}, \eqref{contfrac.eqn.Ostrowski2+1} and \eqref{contfrac.eqn.Ostrowski2+},  the equalities associated with (\ref{vb12zx}) can only occur if the coefficients $b_{1},\dots,b_{L}$ of the Ostrowski expansion of $\gamma$ have one of the following two patterns:\\[1ex]
\begin{equation}\label{pattern2}
\begin{array}{|c|c|c|c|c|c|c|}
\hline
b_{1} & b_{2} & b_{3} & b_{4} &b_{5} &b_{6} & \dots  \\ \hline
a_{1}-1 & 0 & a_{3} & 0 & a_{5} & 0 &\dots \\ \hline
0  & a_{m+2} & 0 & a_{4} &0 &a_{6} & \dots \\ \hline
\end{array}
\end{equation}

\medskip

\noindent
For positive integers $L\le K+2$ and for $1\le a\le 2a_{L+1}$, let $F_N(L,a)$ be the set of positive integers $n\le N$ such that $|\Sigma|>1/2$, $m=m(n,\alpha,\gamma)\le K$ and

\begin{equation}\label{c++}
\left\{\begin{array}{l}
\delta_1=(-1)^m\sgn(\delta_{m+1})(a_1-1), \\[1.5ex]
\delta_{i+1}=(-1)^{i+m}\sgn(\delta_{m+1})a_{i+1}\qquad(1\le i\le L-1),\\[1.5ex]
a_{L+1}-(-1)^{L+m}\sgn(\delta_{m+1})\delta_{L+1}=a\,.
        \end{array}\right.
\end{equation}

\noindent It is easily verified that the sets $F_N(L,a)$ account for the remaining positive integers $n\le N$, not accounted for elsewhere in the  proof of the theorem. By Lemmas~\ref{contfrac.inhomoglem.inhomoglem} and~\ref{contfrac.leminhombnd.leminhombnd2}, for $n\in F_N(L,a)$ we have that
$$
\|n\alpha-\gamma\|\ge a|D_L|\,.
$$
Next, on using   (\ref{c++}), it is readily  verified that   $F_N(L,a)$ is the set $A_N(d_1,\dots,d_{L+1})$ defined within Lemma~\ref{contfrac.sumbndlem.sumbndlem}, with\\[1ex]
\begin{equation*}
\left\{\begin{array}{l}
d_1=b_1+(-1)^m\sgn(\delta_{m+1})(a_1-1)\,,\\[1.5ex]
d_{i+1}=b_{i+1}+(-1)^{i+m}\sgn(\delta_{m+1})a_{i+1}\qquad(1\le i\le L-1)\,, \\[1.5ex]
d_{L+1}=b_{L+1}-(-1)^{L+m}\sgn(\delta_{m+1})(a-a_{L+1})\,.
\end{array}\right.
\end{equation*}

\medskip

\noindent On exploiting  the same arguments used to establish (\ref{e1}) and \eqref{x0}, we obtain that
\begin{align}
\sum_{L=1}^{K+2}&~\sum_{a=1}^{2a_{L+1}}\sum_{\substack{n\in F_N(L,a)\\ c\le n\|n\alpha-\gamma\|}}\frac{1}{n\|n\alpha-\gamma\|}\nonumber\\[2ex]
&\ll\ (\log N)\sum_{L=1}^{K+2}(1+\log a_L)+\sum_{L=1}^{K+2}a_L ~\ll~ (\log N)^2\,.\label{e2+}
\end{align}

\medskip

Combining the lower bound estimates given by  \eqref{m>K}, (\ref{x0}), (\ref{e2}), \eqref{e++} and (\ref{e2+}) completes the proof of Theorem~\ref{T6}.

\subsection{Proof of Theorem \ref{T4}}\label{sec.thm1}

Fix $\gamma\in\R$. Theorem \ref{T4} will follow from Theorem \ref{T6} if we can show that
\begin{equation}\label{proofthm1.eqn1}
\sum_{\substack{1\le n\le N\\n\|n\alpha-\gamma\|\le 1}}\frac{1}{n\|n\alpha-\gamma\|}\ll(\log N)^2~\text{ for almost all }~\alpha\in[0,1).
\end{equation}
Using the Borel-Cantelli Lemma in a standard way (e.g., see \cite{BDV-06} or \cite{Sprindzuk}) one can easily see that for almost all $\alpha\in\R$ the inequality
\begin{equation*}
\|n\alpha-\gamma\|\le\frac{1}{n(\log n)^2}
\end{equation*}
has only finitely many solutions $n\in\N$. Therefore, (\ref{proofthm1.eqn1}) will follow on showing that for some $\varepsilon\in(0,1)$
\begin{equation*}
\sum_{\substack{1\le n\le N\\1/(\log n)^2<n\|n\alpha-\gamma\|\le 1}}\frac{1}{n\|n\alpha-\gamma\|}\ll(\log N)^{1+\varepsilon}~\text{ for almost all}~\alpha\in[0,1).
\end{equation*}
Given an integer $n\ge 2$, define
$$
\A_n:=\left\{\alpha\in [0,1)~:~\frac{1}{n(\log n)^2}<\|n\alpha-\gamma\|\le \frac{1}{n}\right\}
$$
and
$$
J(n)=\int_{\A_n}\frac{d\alpha}{\|n\alpha-\gamma\|}.
$$
It is elementary to check that
\begin{align*}
J(n)  \  & =   \  \  \sum_{a=1}^n\left(\int_{(a+\gamma)/n-1/n^2}^{(a+\gamma)/n-1/(n\log n)^2}+\int_{(a+\gamma)/n+1/(n\log n)^2}^{(a+\gamma)/n+1/n^2}\right)\frac{d\alpha}{|n\alpha-a-\gamma|}\ \\[4ex]
& \sim    \  \ 4\log\log n\text{ \ \ as  \ \ $n\to\infty$. }\qquad
\end{align*}
This together with Fatou's lemma, implies that
\begin{align*}
\int_0^1\left(\sum_{\substack{n=2\\ (\log n)^{-2}<n\|n\alpha-\gamma\|\le 1}}^\infty\frac{1}{n(\log n)^{1+\varepsilon}\|n\alpha-\gamma\|}\right)~d\alpha ~\le~ \sum_{n=2}^\infty\frac{J(n)}{n(\log n)^{1+\varepsilon}}<\infty.
\end{align*}
The upshot is that the  above integrand  must be finite almost everywhere, therefore, for almost all $\alpha$  we have that

\begin{align*}
\sum_{\substack{1\le n\le N\\ (\log n)^{-2}<n\|n\alpha-\gamma\| \le 1}}\hspace*{-3ex}\frac{1}{n\|n\alpha-\gamma\|}\ &\le \ (\log N)^{1+\varepsilon}\hspace*{-5ex} \sum_{\substack{1\le n\le N\\ (\log n)^{-2}<n\|n\alpha-\gamma\|\le 1}}\hspace*{-3ex}\frac{1}{n(\log n)^{1+\varepsilon}\|n\alpha-\gamma\|}\\[4ex]
& \ll \ (\log N)^{1+\varepsilon}.
\end{align*}

\noindent This proves (\ref{proofthm1.eqn1}) and thereby
completes the proof of Theorem~\ref{T4}.

\subsection{Proof of Theorem \ref{T5}}\label{sec.thm2}

Let $\alpha$ and $\psi$ be as in the statement of Theorem~\ref{T5}. Without loss of generality, we will assume that $\alpha\in(0,1)$ and   we let $p_k/q_k$ denote the principal convergents of $\alpha$. Let $(\ve_i)_{i\in\N}$ be a sequence of elements from the set $\{0,1\}$ and define, inductively a strictly increasing sequence of integers $(k_i)_{i\in\N}$ as follows. Let $k_1\ge 1$ be the smallest integer satisfying
\begin{equation}\label{c1+}
  |D_{k_1-1}|+|D_{k_1}|<\min\{\alpha,1-\alpha\}\,.
\end{equation}
For  $i\ge1$, assuming $k_1,\dots,k_i$ are given, let $k_{i+1}$ be  taken large enough that \begin{equation}\label{cc2}
  k_{i+1}\ge k_i+5\,,
\end{equation}
\begin{equation}\label{cc3}
k_{i+1}-k_i\equiv \ve_i\pmod2\,,
\end{equation}
\begin{equation}\label{proofthm2.psieqn1}
\frac{q_{k_{i+1}+1}}{\psi(q_{k_{i+1}+1})}\ge (i+1)\cdot q_{k_i+1}
\end{equation}
and
\begin{equation}\label{proofthm2.psieqn2}
\frac{q_{k_i+1}}{q_{k_{i-1}+1}}\le \frac{q_{k_{i+1}+1}}{q_{k_i+1}}\,.
\end{equation}
Note that it is possible to choose $k_{i+1}$ in this manner  because of the assumption that $f(N)=o(N)$ as $N\to\infty$.
Now let
\begin{equation}\label{bb}
b_{k+1}:=\left\{\begin{array}{cl}
      1 & \text{if }~ k=k_i~\text{for some }~i\in\N\,,\\[1ex]
      0 & \text{otherwise}
     \end{array}
\right.
\end{equation}
and
$$
\gamma:=\sum_{k=0}^\infty b_{k+1}D_k=\sum_{i=1}^\infty D_{k_i}\,.
$$
It is readily seen that $\gamma\in(-\alpha,1-\alpha)$. Indeed, using \eqref{vb1} and \eqref{c1+} we have  that
$$
|\gamma|\le \sum_{k=k_1}^\infty a_{k+1}|D_{k}|=|D_{k_1-1}|+|D_{k_1}|<\min\{\alpha,1-\alpha\}\, .
$$
Thus, $(b_{k+1})_{k\ge0}$ is the sequence of Ostrowski coefficients associated with the expansion of  $\gamma$. In view of the uniqueness property of  the Ostrowski expansion of a real number, it follows that two different sequences $(k_i)_{i\in\N}$ and $(k'_i)_{i\in\N}$ give rise to two different
numbers $\gamma$ and $\gamma'$. Note that, by \eqref{cc3}, there are at least as many sequences $(k_i)_{i\in\N}$ as sequences $(\ve_i)_{i\in\N}$. Hence, there is a set of $\gamma$ as above of the cardinality of continuum satisfying \eqref{contfrac.eqn.|na-g|}. Now fix any one of them and
for each $i\in\N$ let
$$
n_i = \sum_{k=0}^{k_i}b_{k+1}q_k=\sum_{j=1}^iq_{k_i}.
$$
Clearly, $n_i\le q_{k_i+1}$. Using Corollary~\ref{cor1} (recall the results of \S\ref{qw} are valid even  in the case $\gamma=0$)  and (\ref{vb7}) we find that
$$\|n_i\alpha-\gamma\| \ll \dfrac{1}{q_{k_{i+1}+1}} \, . $$
Thus, it follows that
\begin{align*}
S_{n_i}(\alpha,\gamma)-\max_{1\le n\le n_i}\frac{1}{n\|n\alpha-\gamma\|}&\ge\min\left\{\frac{1}{n_{i-1}\|n_{i-1}\alpha-\gamma\|},\frac{1}{n_i\|n_i\alpha-\gamma\|}\right\}\\[2ex]
&\gg\min\left\{\frac{q_{k_i+1}}{q_{k_{i-1}+1}},\frac{q_{k_{i+1}+1}}{q_{k_i+1}}\right\}~\stackrel{\eqref{proofthm2.psieqn2}}{=}~\frac{q_{k_i+1}}{q_{k_{i-1}+1}}\\[2ex]
&\stackrel{\eqref{proofthm2.psieqn1}}{\ge}~ i\psi(q_{k_i+1})\, ,
\end{align*}
and thereby completes the proof of the theorem since $i$ can be taken arbitrarily large.

\subsection{Proof of Theorem~\ref{T7}}\label{sec.thm3.lb}

As before, without loss of generality we assume that $\gamma\in[-\alpha,1-\alpha)$. We can assume that
\eqref{contfrac.eqn.|na-g|} holds since, if this were not the case, there would exist $n_0\in\N$ and $m_0\in\Z$ such that $\gamma=n_0\alpha+m_0$. Then, on making the substitution $n'=n-n_0$ we would have that
\begin{align*}
\sum_{\substack{rN<n\le N \\[0.5ex] \|n\alpha-\gamma\|\ge N^{-v}}}\frac{1}{\|n\alpha-\gamma\|} & ~= \sum_{\substack{rN<n\le N \\[0.5ex] \|(n-n_0)\alpha\|\ge N^{-v}}}\frac{1}{\|(n-n_0)\alpha\|}\,.
\end{align*}
Since $n\in[rN,N]$ takes at least $ \lfloor(1-r)N\rfloor-1$ values the range of $|n'|=|n-n_0|$ will contain all positive integers up to $M:=\lceil\tfrac12\lfloor(1-r)N\rfloor-1\rceil$. Hence, assuming that $N$ is large enough so that $N^{-v}\ge M^{-v'}$, where $v<v':=(v+1)/2<1$, we have that
\begin{align*}
\sum_{\substack{rN<n\le N \\[0.5ex] \|n\alpha-\gamma\|\ge N^{-v}}}\frac{1}{\|n\alpha-\gamma\|} & ~\ge  \sum_{\substack{rM<n'\le M \\[0.5ex] \|n'\alpha\|\ge M^{-v'}}}\frac{1}{\|n'\alpha\|}\,,
\end{align*}
and  the desired result would follow from the homogeneous case.

Let $(b_{k+1})_{k\ge0}$ be the Ostrowski coefficients of $\gamma$ given within Lemma~\ref{Ostrowski2}. Let $K$ be the largest integer such that $q_K\le N$, and let
$$
n=\left\{\begin{array}{ll}
\sum_{k=0}^{K-1}b_{k+1}q_k &\text{if it is non-zero,}\\[2ex]
q_K &\text{otherwise}.
         \end{array}\right.
$$
It is readily seen that $1\le n\le q_K\le N$ and that the integer $m=m(n,\alpha,\gamma)$ defined in Lemma~\ref{contfrac.inhomoglem.inhomoglem} satisfies $m\ge K$. Then, by Corollary~\ref{cor1}
$$
\|n\alpha-\gamma\|< (|\delta_{m+1}|+2)|D_m|\le \frac{4a_{m+1}}{q_{m+1}}<
\frac{4a_{m+1}}{a_{m+1}q_m}=\frac{4}{q_m}\le\frac{4}{q_{K}}.
$$
Let $0<\nu<\min\{v,w(\alpha)^{-1}\}$, where $w(\alpha)$ is the exponent of  approximation of $\alpha$ -- see \eqref{cors}. Note that $w(\alpha)<\infty$ for $\alpha\notin\mathcal{L}$. Then, by Lemma~\ref{Liouvillelem1}, $q_K^{-1}<\tfrac18q_{K+1}^{-\nu}<\tfrac18N^{-\nu}$ if $N$ is bigger than some sufficiently large $N_0$ depending on $\alpha$.
Thus, for any $N>N_0$ and any $\gamma$ satisfying \eqref{contfrac.eqn.|na-g|}, there exists a positive integer $n\le N$ such that
$$
\|n\alpha-\gamma\|<\tfrac12 N^{-\nu}\,.
$$
Thus, for all $N>N_0$ and all $\gamma$ satisfying \eqref{contfrac.eqn.|na-g|}, we have that
$$
N_\gamma(\alpha,\tfrac12N^{-\nu})\neq\emptyset\,
$$
were $N_\gamma(\alpha,\ve)$ is given by \eqref{ayesha}.
By Corollary~\ref{lem12}
$$
\lfloor N\ve\rfloor\le N(\alpha,\ve)\le 32 N\ve\qquad\text{for $N^{-\nu}<\ve<\ve_0$ and $N>N_0$},
$$
provided $N_0=N_0(\alpha)$ is sufficiently large and $\ve_0=\ve_0(\alpha)$ is sufficiently small.
Then, by Lemma~\ref{inhom2} it follows that
\begin{equation}
\tfrac14N\ve ~\le~ \#N_\gamma(\alpha,\ve) ~\le~ 65N\ve
\end{equation}
provided that $4N^{-\nu}<\ve<\tfrac12\ve_0$.

Now let $R>1$ be an integer such that $R^\nu>1040$. For $k\in\N$, consider the  sets
$$
A_k:=\big\{n\in\N:NR^{-1}<n\le N,~R^{-(k+1)\nu}\le \|n\alpha-\gamma\|< R^{-k\nu}\big\}\,.
$$
Then,
\begin{align*}
\#A_k &~\ge~ \#N_\gamma(\alpha,R^{-k\nu})-\#(N/R)_\gamma(\alpha,R^{-k\nu})
-\#N_\gamma(\alpha,R^{-(k+1)\nu})\\[2ex]
&~\ge~ \tfrac14NR^{-k\nu}-65(N/R)R^{-k\nu}-65NR^{-(k+1)\nu}\\[2ex]
&~\ge~ (\tfrac14-130R^{-\nu}\big)N R^{-k\nu} ~\ge~ \tfrac18 N R^{-k\nu},
\end{align*}
 provided that $4N^{-\nu}<R^{-k\nu}<\tfrac12\ve_0$. The latter holds when $k_0\le k\le \log N/\log R-1$ for some fixed $k_0\in\N$. Clearly,  the sets $A_k$ are disjoint and on taking $r=R^{-1}$ we obtain that
\begin{align*}
\sum_{\substack{rN\le n\le N \\[0.5ex] \|n\alpha-\gamma\|\ge N^{-v}}}\frac{1}{\|n\alpha-\gamma\|}
&~\ge~ \sum_{k_0\le k\le \log N/\log R-1}~\sum_{n\in A_k} \frac{1}{\|n\alpha-\gamma\|} \\[1ex]
&~\ge~ \sum_{k_0\le k\le \log N/\log R-1}~R^{k\nu}\cdot\#A_k\\[1ex]
&~\ge~ \sum_{k_0\le k\le \log N/\log R-1}\tfrac18 N\\[1ex]
& ~\ge~
 \tfrac18N(\log N/\log R-k_0-1)  \, .
\end{align*}
This completes the proof of Theorem~\ref{T7} since  $R$ and $k_0$ depend on $\alpha$ and $v$ only.

\subsection{Proof of Theorem \ref{T8}}

If $\alpha\not\in \mathcal{L} \cup \Q$, then the desired conclusion follows immediately via
Theorem~\ref{T7}. To prove the other direction,  we suppose that $\alpha\in\mathcal{L}$ and  explicitly construct a real number $\gamma$ and a sequence $(N_i)_{i\in\N}$ such that
\begin{equation}\label{x125}
\sum_{1\le n\le N_i}\frac{1}{\|n\alpha-\gamma\|}=o(N_i\log N_i)\qquad\text{as }i\to\infty
\end{equation}
Without loss of generality, we will assume that $\{\alpha\}<\tfrac13$. Indeed,
for any integer $r$ we have the inequality $\|n(r\alpha)\|\le |r|\|n\alpha\|$, which implies that
\begin{equation}\label{bnw}
\sum_{1\le n\le N_i}\frac{1}{\|n\alpha\|}\le |r|\sum_{1\le n\le N_i}\frac{1}{\|n(r\alpha)\|}\,.
\end{equation}
Now, by Dirichlet's Theorem, we can find $r\in\{\pm1,\pm2,\pm3\}$ such that $\{r\alpha\}<\tfrac13$.  The upshot of this and \eqref{bnw}  is  that if $\{\alpha\} \ge \tfrac13$,  we simply replace  $\alpha$ by $\{r\alpha\}$ in the proof.

Recall that $w(\alpha)=\infty$ for $\alpha\in\mathcal{L}$. Then, by Lemma~\ref{Liouvillelem1}, we can find a sequence $\{K_i\}_{i\in\N}$ of positive integers such that
$$
\lim_{i\rar\infty}\frac{\log q_{K_i+1}}{\log q_{K_i}}=\infty.
$$
We will assume that $K_1>1$.
Since $q_{K_i+1}=a_{K_i+1}q_{K_i}+q_{K_i-1}$ we necessarily have that $a_{K_i+1}\to\infty$ as $i\to\infty$. Thus, we may also assume that $a_{K_i+1}\ge8$ for all $i$. Next, we define the sequence $\{b_{k+1}\}_{k\ge 0}$ by setting

\begin{equation}\label{bk+1}
b_{k+1}:=\begin{cases}\left\lceil\frac{a_{k+1}}{2}\right\rceil&\text{if } k=K_i\text{ for some $i$},\\
~~~0&\text{if }a_{k+1}=1\text{ or }2,\\
~~~1&\text{otherwise} , \end{cases}
\end{equation}
and we let
\begin{equation}\label{gamma}
\gamma :=\sum_{k=0}^\infty b_{k+1}D_k.
\end{equation}

\noindent It is easy to see that this expansion satisfies \eqref{contfrac.eqn.Ostrowski2+1} and \eqref{contfrac.eqn.Ostrowski2+}. Therefore, the sum (\ref{gamma}) is absolutely convergent, the real number $\gamma$ is well defined and the integers $b_{k+1}$, $k=0,1,\dots$ are the Ostrowski coefficients of $\gamma$.

\noindent Now, for each $i$ let
\begin{equation}\label{x127}
a'_i=\left\lfloor\frac{a_{K_i+1}}{4}\right\rfloor\qquad\text{and}\qquad N_i=a_i' q_{K_i}.
\end{equation}
Since $a_{K_i+1}\ge 8$, we have that $2\le a'_i<a_{K_i+1}$ and so
\begin{equation}\label{K_i}
q_{K_i}\le N_i<q_{K_i+1}.
\end{equation}
In order to prove \eqref{x125} we will use the techniques developed in \S\ref{inhom}.
First of all we observe  that for any positive integer $n\le N_i$ its Ostrowski coefficient $c_{K_i+1}$ satisfies $c_{K_i+1}\le a_i'\le a_{K_i+1}/4$, while the corresponding Ostrowski coefficient of $\gamma$ is $\ge a_{K_i+1}/2$. Hence
\begin{equation}\label{d+}
a_{K_i+1}/4\le |\delta_{K_i+1}|\le \lceil a_{K_i+1}/2\rceil<a_{K_i+1}-1,
\end{equation}
and the parameter $m=m(n,\alpha,\gamma)$ defined in Lemma~\ref{contfrac.inhomoglem.inhomoglem} satisfies
\begin{equation}\label{m+}
m\le K_i\,.
\end{equation}
Now, consider the sets $C_{N_i}(u,a)$ and  $C^+_{N_i}(m,a,\ell)$  as defined as in \S\ref{inhom} with $N=N_i$. By definition, these deal with the situation $|\Sigma|  \le 1/2$. By Lemma~\ref{cortogapslem1} and \eqref{K_i}, we have that
\begin{equation}\label{x126}
\#C_{N_i}(u,a)\le \frac{3N_i}{q_{u+1}}+1\ll \frac{N_i}{q_{u+1}}.
\end{equation}
In particular, for the case when $u\le K_i-1$, this implies  that

\begin{align*}
\sum_{u=0}^{K_i-1}\sum_{a=1}^{2a_{u+1}-1}\sum_{n\in C_{N_i}(u,a)}&\frac{1}{\|n\alpha-\gamma\|}\\[2ex]
&\stackrel{\eqref{bound1}}{\le} \sum_{u=0}^{K_i-1}\frac{1}{|D_u|}\sum_{a=1}^{2a_{u+1}-1}\frac{\#C_{N_i}(u,a)}{a}\\[2ex]
&\ll N_i\sum_{u=0}^{K_i-1}\sum_{a=1}^{2a_{u+1}-1}\frac{1}{a} \ll N_i\sum_{u=0}^{K_i-1}\log(a_{u+1})\\[2ex]
&= N_i \log(a_1\cdots a_{K_i})\le N_i\log q_{K_i}\,.
\end{align*}

\noindent Next we deal with the set  $C_{N_i}(u,a)$ when $u=K_i$. In this case, we have that
$$
\#C_{N_i}(K_i,a)\le 3N_i/q_{K_i+1}+1\stackrel{\eqref{x127}}{\le} 3\left\lfloor\frac{a_{K_i+1}}{4}\right\rfloor q_{K_i}/q_{K_{i+1}}+1<\tfrac34+1\, .
$$
Hence
\begin{equation}\label{x128}
\#C_{N_i}(K_i,a)\le 1,
\end{equation}
for all $a\ge1$.
Furthermore, since $a_i'\le a_{K_i+1}/4$ and $b_{K_i+1}\ge a_{K_i+1}/2$ the sets $C_{N_i}(K_i,a)$ will be empty whenever $a< a_{K_i+1}/4-1$.
Indeed, any $n\in C_{N_i}(K_i,a)$ has to  lie in one of the four sets given by \eqref{sets1} and \eqref{sets2}. If it lies within those  given by  \eqref{sets1}, then
$$
n\ge (b_{K_i+1}-(a+1))q_{K_i}> (a_{K_i+1}/2-a_{K_i+1}/4)q_{K_i}=(a_{K_i+1}/4)q_{K_i}\ge N_i
$$
and this is impossible for $n\in C_{N_i}(K_i,a)$.
If $n$ lies within those  given by \eqref{sets2}, then either
$$b_{K_i+1}+(a+1-a_{K_i+1})<1+a_{K_i+1}/2+(a_{K_i+1}/4-a_{K_i+1})\le 0,
$$
which is not allowed, or
$$
n\ge (b_{K_i+1}-(a+1-a_{K_i+1}))q_{K_i}> (a_{K_i+1}/2-a_{K_i+1}/4+a_{K_i+1})q_{K_i}> N_i\,,
$$
which is also impossible. Therefore, the contribution from the sets $C_{N_i}(K_i,a)$ to the l.h.s. of \eqref{x125} can be estimated as follows:
\begin{align}
\sum_{a=1}^{2a_{K_i+1}-1}\sum_{n\in C_{N_i}(K_i,a)}\frac{1}{\|n\alpha-\gamma\|}\nonumber
&\ \ \stackrel{\eqref{bound1}\&\eqref{x128}}{\le}\ \ \sum_{a_{K_i+1}/4\le a+1\le 2a_{K_i+1}}\frac{1}{a|D_{K_i}|}\\[2ex]
&~~~~~~~\stackrel{\eqref{vb7}}{\ll}\ \ q_{K_i+1}\stackrel{\eqref{vb7-}\&\eqref{x127}}{\ll} N_i.\label{x130}
\end{align}

\noindent Also note that the set  $C_{N_i}(u,a)$ in the remaining case of $u=K_i+1$ is always empty. Indeed, any $n\in C_{N_i}(K_i+1,a)$ has to  lie in one of the sets given by \eqref{sets1} and \eqref{sets2}. In either case, the Ostrowski coefficient of $q_{K_i}$, which is either $b_{K_i+1}$ or $b_{K_i+1}\pm1$ is at least $a_{K_i+1}/2-1$ and so if $n\in C_{N_i}(K_i+1,a)$ then
$$
n\ge (a_{K_i+1}/2-1)q_{K_i}> (a_{K_i+1}/4)q_{K_i}\ge N_i
$$
and this is impossible. Therefore, the contribution from the sets $C_{N_i}(K_i+1,a)$ to the l.h.s. of \eqref{x125} is zero.

\medskip

Now we analyse the contribution to the l.h.s. of \eqref{x125} arising from the sets $C^+_{N_i}(m,a,\ell)$. Recall that, by definition, we must have that the equalities associated with  \eqref{c+} are satisfied. In particular this means that for $k=0,\dots,m+\ell$ the quantities $|\delta_{k+1}|$ can only be $0$, $1$, $a_{k+1}-1$ or $a_{k+1}$. Then, by \eqref{d+}, $t:=m+\ell+1\le K_i$. By Lemma~\ref{cortogapslem1} and \eqref{x127}, we have that
\begin{equation}\label{x131}
\#C^+_{N_i}(m,a,\ell)\le \frac{3N_i}{q_{t+1}}+1\ll \frac{N_i}{q_{t+1}}\,.
\end{equation}
Recall that for a given $t$ there are at most two possible pairs of $(m,\ell)$ as above for which $C^+_{N_i}(m,a,\ell)$ is non-empty.
Furthermore, by \eqref{contfrac.leminhombnd.eqn1}, we obtain that
$$
\|n\alpha-\gamma\|\ge a|D_{t}|\,.
$$
Thus, the contribution from $C^+_{N_i}(m,a,\ell)$ with $t\le K_i-1$ is estimated as follows:
\begin{eqnarray*}
\sum_{t=2}^{K_i-1}~\sum_{m+\ell+1=t}~\sum_{a=1}^{2a_{t+1}-1} \hspace*{-6ex} & ~ &\sum_{n\in C^+_{N_i}(m,a,\ell)}\frac{1}{\|n\alpha-\gamma\|}  \\[1ex]
& \ll  &  \sum_{t=2}^{K_i-1}~\sum_{m+\ell+1=t}~\sum_{a=1}^{2a_{t+1}-1}~\frac{N_i}{a|D_t|q_{t+1}}  \\[1ex]
& \ll &  N_i\sum_{t=2}^{K_i-1}(1+\log a_{t+1}) \\[1ex]
& \ll  &  N_i\max\{K_i,\log(a_1\dots a_{K_i})\}\ll N_i\log q_{K_i}\,.
\end{eqnarray*}

\noindent If $t=K_i$, then
$$
0\le c_{K_i+1}=b_{K_i+1}\pm(-1)^{\ell}(a_{K_i+1}-a) \ \le  \  N_i/q_{K_i} \  = \ a'_i \ \le \  \frac{a_{K_i+1}}{4}\,.
$$
Since $b_{K_i+1}=\left\lceil\frac{a_{K_i+1}}{2}\right\rceil$, we then have that
$|a_{K_i+1}-a|\le b_{K_i+1}\le \tfrac12a_{K_i+1}+\tfrac12$, whence
$$
a\ge \tfrac12a_{K_i+1}-\tfrac12\ge \tfrac14a_{K_i+1}\,.
$$
In this case
$$
\|n\alpha-\gamma\|\ge a|D_{K_i}|\ge\tfrac14a_{K_i+1}/2q_{K_i+1}\gg 1/q_{K_i}\gg N_i^{-1}\,.
$$
Then, by \eqref{x127} and \eqref{x131}, it follows  that $\#C^+_{N_i}(m,a,\ell)\ll1$ and thus
$$
\sum_{m+\ell+1=K_i}~\sum_{a\asymp a_{K_i+1}}\sum_{n\in C^+_{N_i}(m,a,\ell)}\frac{1}{\|n\alpha-\gamma\|}\ll \sum_{a\asymp a_{t+1}}\frac{N_i}{q_{K_i+1}a|D_{K_i}|}\ll N_i\,.
$$

\bigskip

Finally, it remains to deal with the case $|\Sigma|>1/2$. Note that in view of our assumption $0< \{\alpha\}<\tfrac13$ imposed at the start of the proof, we have that $a_1\ge3$. Then, since $K_1>1$, we have that $b_1=1$ and as is easily seen $\delta_1=c_1-b_1\neq\pm(a_1-1)$. This means that the first term of \eqref{contfrac.leminhombnd.eqn1zx} is never zero. Hence $1-|\Sigma|\ge|D_0|=\{\alpha\}$. Then
\begin{equation}\label{e++00}
\sum_{\substack{1\le n\le N_i\\[0.5ex] |\Sigma|\ge1/2}}\frac{1}{\|n\alpha-\gamma\|}\ll N_i\,.
\end{equation}
Hence every positive integer $n\le N_i$ has been accounted for and on combining the various estimates above we obtain the upper  bound
\[\sum_{1\le n\le N_i}\frac{1}{\|n\alpha-\gamma\|}\ll N_i\log q_{K_i}.\]
By our choice of the sequences $\{K_i\}$ and $\{N_i\}$, this implies \eqref{x125} and thus completes the proof of Theorem \ref{T8}.

\section{Proofs of Theorems~\ref{T12}, \ref{T13} and \ref{T14}}

\subsection{Preliminaries \label{wer}}

We begin by describing our strategy and establishing various auxiliary statements. Instead of dealing with the multiplicative problem associated with (\ref{ineq+},  we consider the one-dimensional metrical problem associated with the inequality
\begin{equation}\label{yyy}
 \|n\beta - \delta\| < \frac{\psi(n)}{ \|n\alpha - \gamma\| } := \Psi^{\gamma}_{\alpha}(n)\, .
\end{equation}
If $\Psi^{\gamma}_{\alpha}(n)\ge\tfrac12$, then \eqref{yyy} holds automatically. Hence, to avoid this trivial pathological situation we will assume that
\begin{equation}\label{ass}
\Psi^{\gamma}_{\alpha}(n)<\tfrac12.
\end{equation}
Let $E_n$ be the set of $\beta\in[0,1]$ satisfying \eqref{yyy}.
It is easily verified that
\begin{equation*}
  |E_n|=2\Psi^{\gamma}_{\alpha}(n)\,.
\end{equation*}
If
\begin{equation}\label{sum12}
\sum_{n=1}^\infty\Psi^{\gamma}_{\alpha}(n)
\end{equation}
converges, the Borel-Cantelli Lemma implies that the set of $\beta$'s such that \eqref{yyy} (and hence \eqref{ineq+}) is satisfied infinitely often is of Lebesgue measure zero. Hence, Theorem~\ref{T12} will follow on proving that the conditions of the theorem ensure the convergence of $\sum_{n=1}^\infty\Psi^{\gamma}_{\alpha}(n)$. Note that in this case \eqref{ass} holds for all sufficiently large $n$ and thus our above argument is justified.

In order to prove Theorem~\ref{T13} we will have to ensure that the sum  \eqref{sum12} diverges. However, the approximating function $ \Psi^{\gamma}_{\alpha}$ in not monotonic so we cannot apply the inhomogeneous version of Khintchine's theorem\footnote{The inhomogeneous version of Khintchine's theorem states that for any monotonic $\psi:\N\rar\Rp$  and  $\delta \in \R$, the set $\cS^{\delta} (\psi):= \{ \beta \in \R:\|n \beta - \delta \|<\psi(n)\text{ for i.m. }n\in\N\,   \}  $ is of full measure if $\sum \psi(n) = \infty$.}. Instead we will attempt to make use of known results regarding the Duffin-Schaeffer Conjecture; in particular the Duffin-Schaeffer Theorem. With the above strategy in mind,  we now investigate the convergence/divergence behaviour  of the sum \eqref{sum12}.

\begin{lemma}\label{lem+}
Let $\psi:\N\to\Rp$ be decreasing and $\alpha,\gamma\in\R$ be given.
If the sum \eqref{yy} diverges and
\begin{equation}\label{cond+1}
R_N(\alpha;\gamma)\gg N\log N\qquad\text{for all $N\in\N$}
\end{equation}
then the sum \eqref{sum12} diverges.
Conversely, if the sum \eqref{yy} converges and \eqref{cond+} is satisfied then the sum \eqref{sum12} converges.
\end{lemma}

\begin{proof}
By the partial summation formula, for any $N\in\N$ we have that
\begin{equation}\label{ggg}
\sum_{n \le N }  \Psi^{\gamma}_{\alpha}(n)   =    \sum_{n \le N }  \big(\psi(n) -\psi(n+1) \big)R_n(\alpha,\gamma)+\psi(N+1) R_N(\alpha,\gamma)\,.
\end{equation}
Now, if \eqref{cond+1} holds and \eqref{yy} diverges, then using the monotonicity of $\psi$ and the fact that $\sum_{m \le n} \log m \asymp n\log n $, we obtain that
\begin{eqnarray}
\sum_{n \le N }  \Psi^{\gamma}_{\alpha}(n)&\gg & \sum_{n \le N }  \big(\psi(n) -\psi(n+1) \big)  n \log n \ + \ \psi(N+1) N \log N  \nonumber \\[2ex]
& \asymp &
\sum_{n \le N }  \big(\psi(n) -\psi(n+1) \big)\sum_{m \le n }  \log m  \ + \ \psi(N+1) \sum_{m \le N}  \log m   \nonumber \\[2ex]
&= & \sum_{n \le N } \,  \psi (n)  \;  \log n \to\infty\qquad\text{as }N\to\infty.\nonumber
\end{eqnarray}

The proof of the convergence case of the lemma follows the same line of argument as above but with $\gg$ reversed and makes use of the fact that $\sum_{n \le N } \psi (n) \log n$ is bounded.
\end{proof}

Combining Lemma~\ref{lem+} with Theorem~\ref{T8} gives the following statement.

\begin{corollary}
If the sum \eqref{yy} diverges then for all $\alpha\not\in\mathfrak{L} \cup \Q$ and for all $\gamma \in \R$,
the sum \eqref{sum12} diverges.
\end{corollary}

\begin{lemma}\label{lem++}
Let $\psi:\N\to\Rp$ be given and $n\mapsto n\psi(n)$ be decreasing. Furthermore,  let $\alpha,\gamma\in\R$.
If the sum \eqref{yy} diverges and
\begin{equation}\label{cond+2}
S_N(\alpha;\gamma)\gg (\log N)^2\qquad\text{for all $N\in\N$},
\end{equation}
then the sum \eqref{sum12} diverges.
Conversely, if the sum \eqref{yy} converges and \eqref{cond} is satisfied, then the sum \eqref{sum12} converges.
\end{lemma}

\begin{proof}
By the partial summation formula, for any $N\in\N$ we have that
\begin{align}
\sum_{n \le N }  \Psi^{\gamma}_{\alpha}(n)   & =    \sum_{n \le N }  \big(n\psi(n) -(n+1)\psi(n+1) \big)S_n(\alpha,\gamma)\nonumber\\[0ex]
&\hspace*{20ex}+(N+1)\psi(N+1) S_N(\alpha,\gamma)\,.\label{ggg+}
\end{align}
Now, if \eqref{cond+2} holds and the sum \eqref{yy} diverges, then using the monotonicity of $n\mapsto n\psi(n)$ and the fact that $\sum_{m \le n} \frac{\log m}{m} \asymp (\log n)^2 $, we obtain that
\begin{align*}
\sum_{n \le N }  \Psi^{\gamma}_{\alpha}(n)
&\gg\sum_{n=1}^N(n\psi(n)-(n+1)\psi(n+1))(\log n)^2+\\[0ex]
&\hspace*{15ex}+(N+1)\psi(N+1)(\log N)^2\\[2ex]
&\gg \sum_{n=1}^N(n\psi(n)-(n+1)\psi(n+1))\sum_{m=1}^n\frac{\log m}{m}+\\[2ex]
&\hspace*{16ex}+(N+1)\psi(N+1)\sum_{m=1}^N\frac{\log m}{m}   \, .
\end{align*}
By the partial summation formula, the r.h.s. is $\sum_{n \le N }\psi(n)\log n $ and so we have that
\begin{align*}
\sum_{n \le N }  \Psi^{\gamma}_{\alpha}(n) \ \gg \  \sum_{n \le N } \psi(n)\log n \ \to \ \infty\quad\text{as } \ N \ \to  \ \infty\,.\qquad
\end{align*}

The proof of the convergence case of the lemma follows the same line of argument as above but with $\gg$ reversed and makes use of the fact that $\sum_{n \le N } \psi (n) \log n$ is bounded.
\end{proof}

\subsection{Proof of Theorem~\ref{T12}}

Combining Lemmas~\ref{lem+} and \ref{lem++} together with the observations  made at the beginning of \S\ref{wer} completes the proof of Theorem~\ref{T12}.

\subsection{Proof of Theorem~\ref{T13} \label{sect13} }

For $\xi>0$,  define $\mathcal{N}:=\mathcal{N}(\xi)\subseteq\N$ by
\begin{equation} \label{N}
\mathcal{N}= \{n_i \}_{i \in \N}:=\{n\in \N: \varphi(n)/n\ge\xi\}  .
\end{equation}
By choosing $\xi$ small enough we can guarantee that the set $\mathcal{N}$ has positive lower asymptotic density in $\N$; i.e. that
\[\liminf_{N \to \infty}  \# \{n_i \in \mathcal{N} : n_i \leq N \} / N  > 0.\]  To see this, first of all note that
\begin{align}
\frac{\varphi(n)}{n} &= \prod_{p|n}\left(1-\frac{1}{p}\right) = \exp\left(-\sum_{p|n}\frac{1}{p} - \sum_{p|n} \sum_{m=2}^\infty \frac{1}{mp^m}\right) \nonumber\\
&\gg \exp\left(-\sum_{p|n}\frac{1}{p}\right).\label{phiestimate1}
\end{align}
By \cite[Lemma 4]{Vaaler78}, for any $\xi'>0$ we can find an
integer $v$ (depending on $\xi'$) so that for any $x>0$, the number of integers $1\le n \le x$ which
satisfy
\[\sum_{\substack{p|n \\ p\ge v}}\frac{1}{p} \ge \xi'\]
is less than $x/2$. For all other integers we have by (\ref{phiestimate1}), together with Mertens's Theorem \cite[Theorem 429]{HardyWright}, that
\[\frac{\varphi(n)}{n} \ge c\cdot\exp(-\xi'-\log\log v),\]
for some constant $c>0$. Therefore choosing
\[\xi = c\cdot\exp(-\xi'-\log\log v)\]
yields a set $\mathcal{N}$ with lower asymptotic density greater than $1/2$, which verifies our assertion. In particular, a consequence of the fact that the  density is  greater than  $1/2$ is that
\begin{equation} \label{yesthis}
n_i \, \leq  \, 2 i,
\end{equation}
for all $n_i\in\mathcal{N}$ with sufficiently large $i$. Now let $\psi^{\gamma}_{\alpha}$  denote the approximating function $\Psi^{\gamma}_{\alpha} $ restricted to $ \mathcal{N}$; i.e. for $ n \in \N$
$$
\psi^{\gamma}_{\alpha} (n)  :=\left\{\begin{array}{ccl}
\frac{\psi(n) }{ \|n\alpha - \gamma \|}   \hspace*{-5ex} & \text{if} & \hspace*{-5ex} n\in\mathcal{N}\\[2ex]
0 & \text{otherwise}\, . &
\end{array}
\right.
$$
By definition, we have that
$$
 \cS^{\delta} (\psi^{\gamma}_{\alpha}) \, \subset \,    \cS^{\delta} (\Psi^{\gamma}_{\alpha}),
$$
and that
\begin{equation} \label{That}
\sum_{n \le N } \frac{\varphi(n)}{n}  \,   \psi^{\gamma}_{\alpha}(n)  \,  \ge \,   \xi \sum_{ n \le N}    \psi^{\gamma}_{\alpha}(n) \, .
\end{equation}
Next, it follows from a well know discrepancy result in the theory of uniform distribution \cite[Proposition 4]{Schmidt1964},   that if $n_i$ is a strictly increasing sequence of positive integers and $\gamma \in \R$,  then for almost all $\alpha \in \R$
\begin{equation}\label{cc10}
\sum_{i\le T} \frac{1}{\|n_i\alpha - \gamma \|} \gg  T \, \log T  \, .
\end{equation}
Actually, this statement is explicitly established   in the proof of  \cite[Theorem 2]{Schmidt1964} -- see the last displayed inequality on \cite[page~516]{Schmidt1964}. This together with (\ref{yesthis})  implies  that  there is a set $\A\subseteq \R $ of full measure  such that for any $\alpha\in\A$
\begin{eqnarray}
\sum_{n \le n_T }\psi^{\gamma}_{\alpha}(n)  & = & \sum_{i  \le T }\psi^{\gamma}_{\alpha}(n_i)    \nonumber\\[2ex]
&   =  & \sum_{i\le T}  \big(\psi(n_i) -\psi(n_{i+1}) \big) \sum_{j \le i }  \frac{1}{ \|n_j \alpha - \gamma\| } \nonumber\\[1ex]
& & \hspace*{15ex} \ + \ \psi(n_{T+1}) \sum_{i \le T}  \frac{1}{\|n_i \alpha - \gamma\|}  \nonumber \\[1ex]
& \gg  & \sum_{i \le T }  \psi(n_i)  \log i    \ \ge \   \sum_{i \le T }  \psi( 2 i )  \log i    \nonumber \\[2ex]
&\gg &  \sum_{n \le T }  \psi( n)  \log n    \, . \nonumber
\end{eqnarray}

\noindent Thus the divergence of the sum \eqref{yy} implies that
\begin{equation*}\label{ohyesyes}
\sum_{n=1}^{\infty} \psi^{\gamma}_{\alpha}(n) =\infty   \qquad \forall \  \alpha \in \A   \quad  \mbox{and} \quad    \forall  \   \gamma \in \R  .
\end{equation*}
The upshot of this together with (\ref{That})  is that for any  $ \alpha \in \A $ and $\gamma \in \R$  the  hypotheses of the Duffin-Schaeffer Theorem, namely (\ref{divcond1}) and (\ref{dst}),  are satisfied for the function $\psi^{\gamma}_{\alpha}$.  Thus, when $\delta = 0$ the Duffin-Schaeffer Theorem  implies that set $\cS^{\delta} (\psi^{\gamma}_{\alpha}) $ is of full measure. Hence,   it follows that $\cS^{\delta} (\psi^{\gamma}_{\alpha}) $ with $\delta =0$  is of full measure, and this establishes Theorem \ref{T13}.

\medskip

\begin{remark}
Note that we have proven a little more than what is stated in Conjecture~\ref{rub}. Namely,  that for any $\gamma \in \R$  and  almost all $(\alpha, \beta)  \in\R^2$  the inequality
\begin{equation}\label{gy}
\|n\alpha - \gamma\| \,  \|n\beta \|' < \psi(n)
\end{equation}
holds for infinitely  $n \in \N $  if  the the sum \eqref{yy} diverges, where $\|n\beta\|'$ stands for the distance of $n\beta$ to the nearest integer coprime to $n$.   Observe that the  argument used above would prove Conjecture~\ref{rub} in full (i.e. for any real $\delta$) if we had the inhomogeneous version of the Duffin-Schaeffer Theorem (see Problem \ref{crazy}) at hand.
\end{remark}

\subsection{Proof of Theorem~\ref{T14}}

Let  $\Psi^{\gamma}_{\alpha}(n)$ be given by
$$
\Psi^{\gamma}_{\alpha}(n):=\frac{\psi(n)}{\|n\alpha-\gamma\|}\qquad {\rm with }\qquad
\psi(n):=\frac{1}{n(\log n)^2 \, \log\log\log n }.
$$
Recall that
$$
\frac{\varphi(n)}{n}   \gg \frac{1}{\log \log n }\,,
$$
see for example \cite[Theorem 328]{HardyWright}. It follows, by the partial summation formula together with  Theorem~\ref{T8}, that  for any irrational $\alpha \in \R\setminus\mathcal{L}, $
\begin{eqnarray*}
\sum_{n\le N}\frac{\varphi(n)}{n}\Psi^{\gamma}_{\alpha}(n) & \gg  & \sum_{n\le N}\frac{\psi(n)}{\log\log n}\\[1ex]
&= & \sum_{n\le N}\left(\frac{\psi(n)}{\log\log n}-\frac{\psi(n+1)}{\log\log(n+1)}\right)\sum_{m=1}^n\frac{1}{\|m\alpha-\gamma\|}\\[1ex]
&&\qquad \hspace*{16ex}   \qquad  +  \quad \frac{\psi(N+1)}{\log\log(N+1)}\sum_{m=1}^N\frac{1}{\|m\alpha-\gamma\|}\\[1ex]
&\gg &\sum_{n\le N}\frac{\psi(n)  \, \log n}{\log\log n}  \quad =  \quad \sum_{n\le N}\frac{1}{n\log n\,\log\log n \, \log \log \log n }.
\end{eqnarray*}
Thus
$$
\sum_{n=1}^{\infty} \frac{\varphi(n)}{n}\Psi^{\gamma}_{\alpha}(n)=\infty \, .
$$
It now follows,  on assuming the truth of the Duffin-Schaeffer Conjecture for functions $\Psi^{\gamma}_{\alpha}(n)$,  that for almost all $\beta \in \R $
the inequality
$$
 |n\beta -s| < \Psi^{\gamma}_{\alpha}(n)
$$
holds for infinitely coprime pairs $(n,s) \in \N \times\Z$. This completes the proof.

\section{Proof of Theorem \ref{T11}} \label{multiproof}

The approach we develop in this section inherits some ideas from the theory of regular/ubiquitous systems -- see \cite{BBD-Baker, BDV-06}. To begin with we recall some basic statements. Note that the conclusion of Theorem~\ref{T11} is trivial when $\alpha\in\Q$. Hence, throughout the proof we will assume that $\alpha$ is irrational.

\subsection{A zero-one law and quasi-independence on average}\label{qiasec}

We begin by recalling the following zero-one law originally discovered by Cassels \cite{Cassels-50:MR0036787}, see also \cite{zol}.

\begin{lemma}\label{lem17}
Let $\Psi:\N\to[0,+\infty)$ be any function such that $\Psi(n)\to0$ as $n\to\infty$. Then, the set $\mathcal{W}(\Psi)$ of $x\in[0,1]$ such that $\|nx\|<\Psi(n)$ for infinitely many $n\in\N$ is either of Lebesgue measure  zero or Lebesgue measure one.
\end{lemma}

The obvious consequence of this result is that establishing Theorem~\ref{T11} only requires us to show that the set of interest, that is the set of $\beta\in[0,1]$ such that \eqref{ineq} holds infinitely often, is of positive Lebesgue measure. To accomplish this task we will employ the following generalisation of the Borel-Cantelli Lemma from probability theory, see either of  \cite[Lemma 5]{Sprindzuk}, \cite[\S8]{BDV-06} or \cite[\S2.1]{notes}.

\noindent

\begin{lemma}\label{qia}
Let $E_t \subset [0,1]$ be a sequence of Lebesgue measurable sets such that
\begin{equation}\label{div}
\sum_{t=1}^\infty  |E_t|=\infty\,.
\end{equation}
Suppose that there exists a constant $C>0$ such that
\begin{equation}\label{qia0}
\sum_{t, t' = 1}^{T} |E_t \cap E_{t'} |\le C  \left( \sum_{t=1}^{T} |E_t| \right)^2
\end{equation}
for infinitely many $T\in\N$. Then
\begin{equation}\label{qia2}
| \limsup_{t \to \infty} E_t | \; \geq \;  \frac1C\,.
\end{equation}
\end{lemma}

\bigskip

The independence condition \eqref{qia0} is often refereed to as \emph{quasi-independence on average} and  together with the divergent sum condition guarantees that the associated  $\limsup$  set is of positive measure. It does not guarantee full measure; i.e. that $| \limsup_{t \to \infty} E_t |= 1$.  However, this is not an issue if we already know (by some other means) that the $\limsup$ set  satisfies a zero-one law.
%; that is
%\begin{equation*}
%|\limsup_{t \to \infty} E_t| =   \   0 \quad\text{or} \quad 1 \, .
%\end{equation*}

\medskip

\begin{remark}
In view of Lemma~\ref{lem17},  the value of $C$ (as long as it is  positive and finite) in Lemma~\ref{qia} is of no interest. The point is that if we can show that  $ | \mathcal{W}(\Psi) | > 0 $, then Lemma~\ref{lem17} implies  full measure; i.e. $ | \mathcal{W}(\Psi) | = 1 $.
\end{remark}

\subsection{Setting up a limsup set}\label{sulss}

Let $\alpha  \in \R \setminus \Q$ and $R\in\N$ satisfy $R>1$. Given $t,k\in\N$, let
$\Omega_{t,k}$ be the set of real numbers $\beta\in[0,1]$ such that there exists a triple $(n,r,s)\in\N\times \Z^2$ of coprime integers such that
\begin{equation}\label{x4}
    \left\{\begin{array}{rcl}
      R^{-k-1}~\le~|n\alpha-r| & < & R^{-k}\,,\\[0.5ex]
      |n\beta-s| & < & R^{-t+k}\,,\\[0.5ex]
      R^{t-1}~<~ n & \le & R^t\,.
    \end{array}\right.
\end{equation}

\begin{lemma}\label{lem14}
Let $\alpha  \in \R \setminus \Q$, let $(q_\ell)_{\ell\ge0}$ be the sequence of denominators of the principal convergents of $\alpha$, and let $R\ge256$.
Then for any positive integers $t$ and $k$ such that $2R^{-k}<\|q_2\alpha\|$ and
\begin{equation}\label{l42++}
R^{k+1}\le q_{\ell}<R^{t-1}
\end{equation}
for some $\ell$, one has the following estimate on the Lebesgue measure of\/ $\Omega_{t,k}:$
  $$
  |\Omega_{t,k}|\ge \tfrac12.
  $$
\end{lemma}

\begin{proof}
Let $\beta\in[0,1]$. By Minkowski's Theorem for convex bodies (see \cite[p.71]{CasselsGofN}),
there are integers $n,s,r$, not all zero, satisfying the system of inequalities

\begin{equation}\label{zz1}
    \left\{\begin{array}{l}
    |n\alpha-r|< R^{-k}\,,\\[1ex]
    \left|n\beta-s\right|< R^{-t+k}\,,\\[1ex]
    |n|\le R^{t}\,.
    \end{array}\right.
\end{equation}

\noindent In view of \eqref{l42++} we have that $t>0$ and $t-k>0$. Hence $R^{-k}<\tfrac12$ and $R^{-t+k}<\tfrac12$, and $n$ cannot be zero as otherwise $n=r=s=0$. Also,  without loss of generality we can assume that $n>0$ and $n,r,s$ are coprime. Note that if $\beta\in[0,1]\setminus\Omega_{t,k}$ then we also necessarily have that
either
$$
\left\{
\begin{array}{l}
\|n\alpha\|=|n\alpha-r|  <  R^{-k-1}\,,\\[1ex]
1\le n\le R^t\,,
\end{array}
\right.
$$
or
$$
\left\{\begin{array}{l}
\|n\alpha\|=|n\alpha-r|  <  R^{-k}\,,\quad\\[1ex]
1\le n\le R^{t-1}\,.
\end{array}
\right.
$$

\noindent In view of \eqref{l42++}, Lemma~\ref{lem12++} is applicable to  either of the above systems and therefore the number of integers $n$ satisfying at least one of these systems is bounded above by $64 R^{t-k-1}$. Furthermore, observe that for every $n$ the measure of $\beta\in[0,1]$ satisfying $|n\beta-s|<R^{-t+k}$ for some $s\in\Z$ is $2R^{-t+k}$ (see, for example, \cite[Lemma~8]{Sprindzuk}). Hence, the Lebesgue measure of $[0,1]\setminus\Omega_{t,k}$ is bounded above  by
$$
 64 R^{t-k-1}\times 2R^{-t+k}=128R^{-1}\le \tfrac12
$$
as, by hypothesis,  $R\ge256$. Hence, $|\Omega_{t,k}|\ge\tfrac12$ as required.
\end{proof}

Let $\cT^*$ be any subset of pairs of positive integers $(t,k)$ satisfying \eqref{l42++} for some $\ell$. The precise choice of $\cT^*$ will be made later. Further, given a positive integer pair $(t,k)\in\cT^*$, define

\begin{equation}\label{yy15}
N(t,k):=\left\{(n,r,s)\in\N\times\Z^2:\left\{\begin{array}{l}
                                       R^{t-1}~<~ n~\le~ R^t,\\
                                       R^{-k-1}~<~|n\alpha-r|~<~ R^{-k},\\
                                       0\le s\le n,\ \gcd(n,r,s)=1
                                       \end{array}\right.
\right\}.
\end{equation}

\noindent Clearly, for two different pairs $(t,k)$ and $(t',k')$ from $\cT^*$
\begin{equation}\label{yy16}
(n,r,s)\in N(t,k)\quad\& \quad (n',r',s')\in N(t',k')\quad \Longrightarrow \quad n\neq n'\,.
\end{equation}
By the definition of $\Omega_{t,k}$, we have that
\begin{align*}
\Omega_{t,k}&\subset \bigcup_{(n,r,s)\in N(t,k)}\left\{\beta\in\R:\left|n\beta-s\right|<R^{-t+k}\right\}\\[2ex]
&\subset
\bigcup_{(n,r,s)\in N(t,k)}\left\{\beta\in\R:\left|\beta-s/n\right|<R^{-2t+k+1}\right\}\,.
\end{align*}
Let $Z(t,k)$ be a maximal subcollection of $N(t,k)$ such that
\begin{equation}\label{zz6}
  \left|\frac{s_1}{n_1}-\frac{s_2}{n_2}\right|> R^{-2t+k}
\end{equation}
for any distinct triples $(n_1,r_1,s_1)$ and $(n_2,r_2,s_2)$ from $Z(t,k)$. By the maximality of $Z(t,k)$, it follows  that
$$
\Omega_{t,k} \  \subset  \bigcup_{(n,r,s)\in Z(t,k)}\left\{\beta\in\R:\left|\beta-\frac{s}{n}\right|<(R+1)R^{-2t+k}\right\}\,.
$$
Since, by definition, for any $(t,k)\in\cT^*$ condition \eqref{l42++} is satisfied for some $\ell$, Lemma~\ref{lem14} is applicable, and we have that $|\Omega_{t,k}|\ge \tfrac12$. Therefore,
$$
\tfrac12  \ \le \ |\Omega_{t,k}| \ \le  \  \# Z(t,k)\times (2R+2)R^{-2t+k},
$$
and so
$\#Z(t,k)\ge \frac{1}{4R+4}R^{2t-k}$. By \eqref{zz6}, we also have that $\#Z(t,k)\le R^{2t-k}$.
Thus, we conclude that
\begin{equation}\label{zz8}
  C_1R^{2t-k}  \ \le \   \#Z(t,k)   \ \le  \  R^{2t-k}\qquad\text{with }  \quad C_1=\frac{1}{4R+4}\,.
\end{equation}
Now, given $\xi\in\R$, let
\begin{equation}\label{yy8}
E_{t,k}(\xi)=\{\beta\in \R:|\beta-\xi|<\psi(R^t)R^{-t+k}\}\,.
\end{equation}
Furthermore, define
\begin{equation}\label{Etk}
E_{t,k}:=\bigcup_{(n,r,s)\in Z(t,k)} E_{t,k}(s/n)\,
\end{equation}
and let $E$ be the set of $\beta\in\R$ such that
$
\beta\in E_{t,k}
$
for infinitely many pairs $(t,k)\in\cT^*$. The following fairly straightforward statement reveals the role of $E$ in establishing Theorem~\ref{T11}.

\begin{lemma}\label{lem15}
  Let $E$ be as above. Then $E\subset [0,1]$ and for every $\beta\in E\setminus\Q$ we have that
\begin{equation}\label{mnb}
  \|n\alpha\|\,\|n\beta\|<\psi(n)
\end{equation}
for infinitely many $n\in\N$.
\end{lemma}

\begin{proof}
Without loss of generality, we can assume that $\psi(n)\to0$ as $n\to\infty$.
Let $\beta\in E\setminus\Q$. Then, there exist infinitely many pairs $(t,k)\in\cT^*$
such that
$$
\left\{\begin{array}{l}
                                       n~\le~ R^t,\\[1ex]
                                       |n\alpha-r|~<~ R^{-k},\\[1ex]
                                       |\beta-s/n|<\psi(R^t)R^{-t+k}
                                       \end{array}\right.
$$
for some $(n,r,s)\in\N\times\Z^2$. Since $t>k$ for $(t,k)\in\cT^*$, $t$ must take arbitrarily large values. Since $\psi$ is decreasing, we have that
\begin{align}
\|n\alpha\|\,\|n\beta\| & \le n\,|n\alpha-r|\cdot |\beta-s/n|\nonumber \\[1ex]
& \le R^t\,R^{-k}\,\psi(R^t)R^{-t+k}\nonumber \\[1ex]
&=\psi(R^t)\le\psi(n)\,.\label{cc11}
\end{align}
Since $\psi(R^t)\to0$ as $t\to\infty$, if there are only finitely many $n$ arising this way, we would be able to find an $n\in\N$ such that $\|n\alpha\|\,\|n\beta\|=0$. Since $\alpha$ is irrational, we would get that $\|n\beta\|=0$, which means that $\beta\in\Q$ and contradicts the assumption that $\beta\in E\setminus\Q$. Hence, there must be infinitely many $n\in\N$ satisfying \eqref{cc11}, and hence \eqref{mnb}. The inclusion $E\subset [0,1]$ follows from the fact that the approximants $s/n$ to $\beta$ lie in $[0,1]$. The proof is thus complete.
\end{proof}

\begin{remark} \label{mum}
By Lemmas~\ref{lem17}, \ref{qia} and \ref{lem15}, the proof of Theorem~\ref{T11} would be complete if we found a subset $\cT^*$ of suitably ordered pairs $(t,k)$ such that the associated sequence $E_{t,k}$ satisfies conditions \eqref{div} and \eqref{qia0}. In the next subsection we present an explicit choice of $\cT^*$. Subsequently, we deal with establishing conditions \eqref{div} and \eqref{qia0} of Lemma~\ref{qia}.
\end{remark}

\subsection{Choosing the indexing set $\cT^*$}

The goal of this section is to make a choice of the indexing set $\cT^*$ of pairs of positive integers $(t,k)$ introduced in \S\ref{sulss}. Recall,  that we are given a monotonically decreasing function $\psi:\N\to\Rp$ such that the sum \eqref{yy} diverges. In what follows we extend $\psi$ to all real numbers $x\ge 1$ by setting $\psi(x)=\psi(\lfloor x\rfloor)$. Clearly, the extended function is decreasing.
Without loss of generality,  we can assume that
\begin{equation}\label{yy4}
n\psi(n)\le \tfrac12\qquad\text{for all }n\in\N\,.
\end{equation}
If this were not the case, we could replace $\psi$ with $\psi_1(q)=\min\{\psi(q),(2q)^{-1}\}$, which is monotonically decreasing and satisfies the divergence condition -- see \cite[Lemma~4]{Ber99} for a similar argument.

\noindent Next, using the Cauchy condensation test, we obtain that
\begin{equation}\label{yy2}
\sum_{t=1}^{\infty} \,  t\,R^t\,\psi (R^t)=\infty \, .
\end{equation}
Explicitly,
\begin{align*}
\infty  \ = \ \sum_{n=1}^{\infty} \psi (n)\,\log n \
&=\sum_{t=1}^{\infty} \sum_{R^{t-1}\le n < R^{t}} \psi (n)\,\log n\\
&\le \sum_{t=1}^{\infty} \sum_{R^{t-1}\le n < R^{t}} \psi (R^{t-1})\,\log R^{t}\\
&=\sum_{t=1}^{\infty} (R^{t}-R^{t-1})\psi (R^{t-1})\,\log R^{t}\\
&\ll \sum_{t=1}^{\infty} t\,R^{t}\,\psi (R^{t})\,.
\end{align*}

Let $\eta>0$ be a sufficiently small real parameter and $t_0$ be a sufficiently large integer. Define
$$
\cT(t_0,\eta):=\{t\in\N:t\ge t_0,\ \psi(R^t)\ge R^{-(1+\eta)t}\}\,.
$$
Since the sum $\sum_{t=1}^\infty tR^{-\eta t}$ converges for $\eta>0$, the divergnce condition  \eqref{yy2} implies that
\begin{equation}\label{yy3}
\sum_{t\in\cT(t_0,\eta)} \,  tR^t\psi (R^t)=\infty\,.
\end{equation}
Let
\begin{equation}\label{cc102}
\nu_1=\tfrac13+\eta\qquad\text{and}\qquad \nu_2=\tfrac13+2\eta\,,
\end{equation}
and let $\cT$ be the subset of $t\in\cT(t_0,\eta)$ such that
$$
R^{\nu_2t}\le q_\ell <R^{t-1}
$$
for some $\ell$. We now show that we can assume that
\begin{equation}\label{yy3+++}
\sum_{t\in\cT} \,  tR^t\psi (R^t)=\infty\,.
\end{equation}
Indeed, by \eqref{yy3}, this is true if the sum over the complement of $\cT$ converges. Note that $t$ lies in the complement of $\cT$ if
\begin{equation}\label{cc12}
q_\ell<R^{\nu_2t}\quad\text{and}\quad R^{t-1}\le  q_{\ell+1}\,
\end{equation}
for some $\ell$. In particular, this implies that
$$
\|q_\ell\alpha\|~\stackrel{\eqref{vb7}}{<}~\frac{1}{q_{\ell+1}}~\stackrel{\eqref{cc12}}{\le}~ R^{-t+1}=R(R^{\nu_2t})^{-1/\nu_2}\le R q_{\ell}^{-1/\nu_2}\,.
$$
Note that the latter holds only finitely often for sufficiently small $\eta$ if the exponent of approximation of  $\alpha$ is less than $3$. Alternatively, if we are assuming \eqref{lb}, then we have that
$$
\psi(q_\ell)\ge Rq_\ell^{-1/\nu_2}
$$
for all sufficiently large $\ell$. In this case
$$
\|q_\ell\alpha\|<R q_{\ell}^{-1/\nu_2}\le\psi(q_\ell)
$$
for infinitely many $\ell$. Then, for every $\beta\in[0,1]$ we have that
\begin{equation}\label{vb}
\|q_{\ell}\alpha\|\cdot\|q_{\ell}\beta\|\le \|q_{\ell}\alpha\|<\psi(q_{\ell})
\end{equation}
for infinitely many $\ell$. Thus, if \eqref{yy3+++} does not hold the conclusion of Theorem~\ref{T11} holds anyway. Thus, from now on \eqref{yy3+++} will be assumed.

\noindent Finally we let
$$
\cT^*:=\{(t,k):t\in\cT,\ \lceil \nu_1 t\rceil\le k < \lfloor \nu_2 t\rfloor\}\,.
$$

%\subsubsection{The measure diverges}

We now verify that the sum of the measures of the sets $E_{t,k}$ taken over $(t,k)\in\cT^*$ is divergent.  Moreover, we provide an estimate for  the rate of divergence. In what follows

\begin{equation}\label{S}
S_\cT(T):=\sum_{t\in\cT,\ t\le T}tR^t\psi(R^t),
\end{equation}
and $\sum\limits^*$ indicates the summation over $(t,k)\in\cT^*$;  for example, $\sum^*\limits_{t\le T}$ means `sum over $(t,k)\in\cT^*$ with $t\le T$'.

\medskip

\begin{lemma}
For any $T>t_0$ we have that
\begin{equation}\label{ddd1}
\eta\, C_1\, S_\cT(T) \  \le \  \sum^*_{t\le T}|E_{t,k}|\  \ \le  \
2\eta\, S_\cT(T)\,  ,
\end{equation}
where $C_1$ is as in \eqref{zz8}.  In particular,
\begin{equation}\label{ddd2}
\sum^*_{t\le T}|E_{t,k}|\to\infty\quad\text{as  \   $T\to\infty$.}
\end{equation}
\end{lemma}

\begin{proof}
Let $(t,k)\in\cT^* $ with $ t\le T$. Recall that for any distinct triples $(n_1,r_1,s_1)$ and $(n_2,r_2,s_2)$ from $Z(t,k)$ we have that \eqref{zz6} is satisfied. Furthermore, by \eqref{yy8}, \eqref{yy4} and the fact that $R^{-t+k}<1/2$,  the radii of $E_{t,k}(s/n)$ and $E_{t,k}(s'/n')$ are $\le \tfrac12 R^{-2t+k}$. Hence $E_{t,k}(s/n)$ and $E_{t,k}(s'/n')$ are disjoint for distinct triples $(n_1,r_1,s_1)$ and $(n_2,r_2,s_2)$ from $Z(t,k)$. Therefore, on  using \eqref{zz8}, it follows that
\begin{align}
\nonumber |E_{t,k}| & =\sum_{(n,r,s)\in Z(t,k)}|E_{t,k}(s/n)|\\[1ex]
 &=2\psi(R^t)R^{-t+k}\cdot\#Z(t,k)\\[1ex]
 &\stackrel{\eqref{zz8}}{\ge} 2\psi(R^t)R^{-t+k}\cdot C_1R^{2t-k}
=  \ 2C_1R^{t}\psi(R^t)\,,\label{we1}
\end{align}
and
\begin{align}
\nonumber|E_{t,k}| & =2\psi(R^t)R^{-t+k}\cdot\#Z(t,k)\\[1ex]
 &\stackrel{\eqref{zz8}}{\le} 2\psi(R^t)R^{-t+k}\cdot R^{2t-k} = \ 2R^{t}\psi(R^t)\,.\label{we2}
\end{align}
Now, for each fixed $t\in\cT$ the number of different $k$ such that $(t,k)\in\cT^*$ is $\lfloor \nu_2 t\rfloor-\lceil \nu_1 t\rceil\ge \nu_2 t-\nu_1 t-2=\eta t-2\ge \eta t/2$, as long as $t_0\ge 2/\eta$. Also, $\lfloor \nu_2 t\rfloor-\lceil \nu_1 t\rceil\le \eta t$. Then, \eqref{ddd1} readily follow from \eqref{we1} and \eqref{we2}, while the divergence condition \eqref{ddd2} follows from \eqref{ddd1} and \eqref{yy3}.
\end{proof}

\subsection{Overlaps estimates  for $E_{t,k}$}

In the previous section, we established the divergent sum condition \eqref{div} of Lemma~\ref{qia} for the sets $E_{t,k}$ with $(t,k) \in \cT^*$.  In order to complete the proof of Theorem~\ref{T11} it remains to  establish the quasi-independence on average condition \eqref{qia0} of Lemma~\ref{qia} for these sets.  Observe, that  in view of \eqref{ddd1}, this  boils down to showing that for $T$ sufficiently large
\begin{equation} \label{dad}
\sum^*_{t\le T}\sum^*_{t'\le T'}|E_{t,k}\cap E_{t',k'}| ~\ll~S_\cT(T)^2   \, .
\end{equation}

\subsubsection{Preliminary analysis}

Let $(t,k)$ and $(t',k')$ be in $\cT^*$. In particular, we have that $t,t'\ge t_0$,
\begin{equation}\label{yy5}
  \psi(R^t)\ge R^{-(1+\eta)t}\qquad\text{and}\qquad \psi(R^{t'})\ge R^{-(1+\eta)t'}\,,
\end{equation}
and also that
\begin{equation}\label{xcx}
\lceil \nu_1 t\rceil \le k< \lfloor\nu_2 t\rfloor\qquad\text{and}\qquad \lceil \nu_1 t'\rceil \le k'< \lfloor\nu_2 t'\rfloor\,.
\end{equation}

Our goal is to estimate the measure of $E_{t,k}\cap E_{t',k'}$ for  $(t,k)\neq(t',k')$.  To begin with, note that, by \eqref{Etk}
\begin{equation}\label{yy10}
E_{t,k}\cap E_{t',k'}~=~ \bigcup_{\substack{(n,r,s)\in Z(t,k)\\[0.5ex] (n',r',s')\in Z(t',k')}} E_{t,k}(s/n)\cap E_{t',k'}(s'/n')\,.
\end{equation}
Clearly,
$$
|E_{t,k}(s/n)\cap E_{t',k'}(s'/n')|\le \min\big\{\,|E_{t,k}(s/n)|, \, |E_{t',k'}(s'/n')|\,\big\}.
$$
Together with \eqref{yy8}, this  gives that
\begin{equation}\label{yy9}
|E_{t,k}(s/n)\cap E_{t',k'}(s'/n')| \le 2\min\{\psi(R^t)R^{-t+k}, \psi(R^{t'})R^{-t'+k'}\}\,.
\end{equation}
Using \eqref{zz6} and \eqref{yy8} we obtain that for any given $(n',r',s')\in Z(t',k')$ the number of triples $(n,r,s)\in Z(t,k)$ such that
\begin{equation}\label{yy12}
E_{t,k}(s/n)\cap E_{t',k'}(s'/n')\neq\emptyset
\end{equation}
is at most
$$
2+\frac{|E_{t',k'}(s'/n')|}{R^{-2t+k}}\ = \ 2+\frac{2\psi(R^{t'})R^{-t'+k'}}{R^{-2t+k}}\,.
$$
By \eqref{zz8}, we have that $\#Z(t',k')\le R^{2t'-k'}$. Hence the total number of pairs of triples
$(n',r',s')\in Z(t',k')$ and $(n,r,s)\in Z(t,k)$ such that
$E_{t,k}(s/n)\cap E_{t',k'}(s'/n')\not=\emptyset$ is at most
$$
2\left(1+\frac{\psi(R^{t'})R^{-t'+k'}}{R^{-2t+k}}\right)R^{2t'-k'}\,.
$$
Together with \eqref{yy9}.  this gives that
\begin{eqnarray}\label{est1}
|E_{t,k}\cap E_{t',k'}|  & \le  &  8\psi(R^t)R^{t}\psi(R^{t'})R^{t'}   \\[1ex]
&  \ll  &  |E_{t,k} |  \, | E_{t',k'}|  \nonumber
\end{eqnarray}
provided that
\begin{equation}\label{bound1+}
\frac{\psi(R^{t'})R^{-t'+k'}}{R^{-2t+k}}\ge 1\,.
\end{equation}
Since the roles of $(t,k)$ and $(t',k')$ can be reversed in the above argument estimate \eqref{est1} also holds when
\begin{equation}\label{bound2+}
\frac{\psi(R^{t})R^{-t+k}}{R^{-2t'+k'}}\ge 1\,.
\end{equation}
The upshot is that if either  \eqref{bound1+}  or \eqref{bound2+}  holds then we are in good shape. In short,   \eqref{est1}  together with \eqref{we1} and \eqref{we2}   implies that the sets $E_{t,k}$  and  $ E_{t',k'}  $ are pairwise quasi-independent; namely that
$$|E_{t,k}\cap E_{t',k'}|
 \ll   |E_{t,k} |  \, | E_{t',k'}|  \, . $$

\subsubsection{Further analysis}

We now use a divergent technique to estimate from above the number of pairs of triples $(n,r,s)\in Z(t,k)$ and $(n',r',s')\in Z(t',k')$ such that \eqref{yy12} holds. First of all note that \eqref{yy12} implies that
$$
\left|\beta-\frac{s}{n}\right|<\psi(R^t)R^{-t+k}\quad\text{and}\quad
\left|\beta-\frac{s'}{n'}\right|<\psi(R^{t'})R^{-t'+k'}
$$
for some $\beta\in[0,1]$. Hence, by the triangle inequality, we get that
\begin{equation}\label{yy13}
\left|\frac{s}{n}-\frac{s'}{n'}\right|\le 2\max\{\psi(R^t)R^{-t+k}, \psi(R^{t'})R^{-t'+k'}\}\,.
\end{equation}
Then, multiplying \eqref{yy13} by $n'n$ and using the fact that $n\le R^t$ and  that $n'\le R^{t'}$,  we obtain
that
\begin{equation}\label{cond3}
|n's-ns'|\le 2\max\{\psi(R^t)R^{t'+k}, \psi(R^{t'})R^{t+k'}\}=:\Delta\ .
\end{equation}

\noindent Thus the original counting problem related to \eqref{yy12} is replaced by the problem of estimating   the number of  solutions to \eqref{cond3}.  This  is typical for the type of problem under consider, see \cite{Harman1998} or \cite{Sprindzuk}. However, the available techniques to analyse  solutions  to \eqref{cond3}  assume  that the pairs $(n,s)$ and $(n',s')$ are coprime and thus for distinct  pairs we have that
\begin{equation}\label{notzero}
n's-ns'\neq0\, .
\end{equation}
 Unfortunately, we are not able to impose such an  assumption and thus we  need to develop a different argument.

In this section, let us continue with  the task of counting solutions to \eqref{cond3} under the condition  that \eqref{notzero} holds.    With this in mind, first of all observe that condition \eqref{notzero} together with \eqref{cond3} implies that $\Delta\ge1$.  Fix $n,n' \in \N$. Clearly, $n$ and $n'$ uniquely define $r$ and $r'$, since they are the closest integers to $n\alpha$ and $n'\alpha$ respectively. Thus, to fulfill our goal it will be sufficient to count the number of different pairs $(s,s')$ subject to \eqref{cond3} and \eqref{notzero} simultaneously. We consider two cases: (i) the rank of the $\Z$-module generated by the collection of vectors $(s,s')$ is 1, and  (ii) the rank is 2. Clearly, these cases cover all possible options.

\medskip

\noindent\underline{\textit{Rank 1 case}}. In this case we have that all vectors $(s,s')$ in question are collinear. Then there is a fixed non-zero integer vector $(s_0,s'_0)$ such that any other integer vector $(s,s')$ in question has the form $(s,s')=\ell(s_0,s'_0)$ for some $\ell\in\Z$. Since $1\le|n's-ns'|\le\Delta$, we obtain that
$1\le |\ell|\cdot|n's_0-ns'_0|\le\Delta$. Thus, $|\ell|\le\Delta/|n's_0-ns'_0|\le\Delta$. Therefore, the number of pairs $(s,s')$ in the rank one case is no more than $2\Delta$.

\medskip

\noindent \underline{\textit{Rank 2 case}}. In this case there are 2 linearly independent vectors $(s,s')$. All the vectors $(s,s')$ in question lie in the convex subset of points $(x,y)\in\R^2$ defined by the following system of inequalities
$$
|n'x-ny|\le\Delta,\qquad 0\le x\le n.
$$
The volume of this set is easily seen to be $2\Delta$. Hence, by Blichfeld's Theorem\footnote{Blichfeld's Theorem states that for any convex bounded body $B\subset\R^n$ and any lattice $\Lambda$ in $\R^n$ such that ${\rm rank}(B\cap\Lambda)=n$ the cardinality of $B\cap\Lambda$ is $\le \,n!\,\frac{\textrm{vol}_n(B)}{\det\Lambda}+n$.} \cite{Blichfeldt-1921}, this body contains at most $4\Delta+2\le6\Delta$ integer vectors.

The upshot of our discussion is that in either case the number of integer vectors $(s,s')$ in question is $\le6\Delta$. Hence, using \eqref{yy9} and the definition of $\Delta$, we obtain the following estimate
\begin{align*}
\sum_{(s,s')}&|E_{t,k}(s/n)\cap E_{t',k'}(s'/n')|\le\\
&\le 24\max\{\psi(R^t)R^{t'+k}, \psi(R^{t'})R^{t+k'}\}\cdot\min\{\psi(R^t)R^{-t+k}, \psi(R^{t'})R^{-t'+k'}\}\\[2ex]
&=24R^{-t}R^{-t'}\max\{\psi(R^t)R^{t'+k}, \psi(R^{t'})R^{t+k'}\}\cdot\min\{\psi(R^t)R^{t'+k}, \psi(R^{t'})R^{t+k'}\}\\[2ex]
&=24R^{-t}R^{-t'}\cdot\psi(R^t)R^{t'+k}\cdot\psi(R^{t'})R^{t+k'}\\[2ex]
&=24\psi(R^t)R^{k}\,\psi(R^{t'})R^{k'}\,.
\end{align*}
Recall that, by definition, $Z(t,k)\subset N(t,k)$, and so $n$ satisfies the condition $\|n\alpha\|<R^{-k}$. Since for $(t,k)\in\cT^*$ inequalities \eqref{l42++} are satisfied for some $\ell$, Lemma~\ref{lem12++} is applicable. This implies  that the number of different integers $n$ in question is $\le 32 R^{t-k}$. Similarly, the number of distinct values which can be realized by the integer $n'$ is $\le 32 R^{t'-k'}$. Therefore,

\begin{equation}\label{estimate2}
\sum_{n's-ns'\neq0}|E_{t,k}(s/n)\cap E_{t',k'}(s'/n')| \ \le \ \underbrace{24576}_{24\cdot32^2}\,\psi(R^t)R^{t}\,\psi(R^{t'})R^{t'},
\end{equation}

\noindent where the sum is taken over $(n,r,s)\in Z(t,k)$ and $(n',r',s')\in Z(t',k')$ subject to condition \eqref{notzero}.
The upshot is that if  \eqref{notzero} holds then again  we are in good shape;  the sets $E_{t,k}$  and  $ E_{t',k'}  $ are pairwise quasi-independent.

\subsubsection{The remaining case}

In this section we consider the case when none of the conditions \eqref{bound1+}, \eqref{bound2+} or \eqref{notzero} holds. Thus, for the rest of the proof we will assume that
\begin{equation}\label{cond1}
\frac{\psi(R^{t'})R^{-t'+k'}}{R^{-2t+k}}\le 1\,,\qquad\frac{\psi(R^{t})R^{-t+k}}{R^{-2t'+k'}}\le 1
\end{equation}
and investigate the following subset of the overlap between the sets $E_{t,k}$ and $E_{t',k'}$:
\begin{equation}\label{rtrtr}
\bigcup_{n's-ns'=0}E_{t,k}(s/n)\cap E_{t',k'}(s'/n')\,,
\end{equation}
where the union is taken over $(n,r,s)\in Z(t,k)$ and $(n',r',s')\in Z(t',k')$ subject to the condition
\begin{equation}\label{zero}
n's-ns'=0\,.
\end{equation}
Before we continue with the analysis of \eqref{rtrtr}, we first show that $t$ and $t'$ satisfy the inequalities
\begin{equation}\label{vb1005}
  t\le \left(1+\tfrac{6\eta}{5-6\eta}\right)t'\qquad\text{and}\qquad t'\le \left(1+\tfrac{6\eta}{5-6\eta}\right)t\,.
\end{equation}
To see this, first of all note that \eqref{cond1} together with \eqref{yy5} imply that
$$
\frac{R^{-(1+\eta)t'}R^{-t'+k'}}{R^{-2t+k}}\le 1\qquad\text{and}\qquad
\frac{R^{-(1+\eta)t}R^{-t+k}}{R^{-2t'+k'}}\le 1\,.
$$
In view of \eqref{xcx},  the first inequality above  implies that
\begin{align*}
(\tfrac53-2\eta)t\le 2t-k ~\le~ (2+\eta)t'-k' ~\le~ \tfrac53t'
\end{align*}
whence the first  inequality associated with  \eqref{vb1005} follows.
 The proof of the second inequality is identical,  with the roles of $(t,k)$ and $(t',k')$ interchanged.

\medskip

Now we return to estimating the measure of \eqref{rtrtr}. Let $\beta$ be any element of \eqref{rtrtr}. Then there exist $(n,r,s)\in Z(t,k)$ and $(n',r',s')\in Z(t',k')$ satisfying \eqref{zero}, such that

\begin{equation}\label{cc100}
\left\{\begin{array}{l}
                                       R^{t-1}\le n~\le~ R^t,\\[1ex]
                                       |n\alpha-r|~<~ R^{-k},\\[1ex]
                                       |n\beta-s|<R^{-t+k}
                                       \end{array}\right.
\end{equation}
and
\begin{equation}\label{cc101}
\left\{\begin{array}{l}
                                       R^{t'-1}\le n'~\le~ R^{t'},\\[1ex]
                                       |n'\alpha-r'|~<~ R^{-k'},\\[1ex]
                                       |n'\beta-s'|<R^{-t'+k'}.
                                       \end{array}\right.
\end{equation}

\noindent Note that since the vectors $(n,r,s)$ and $(n',r',s')$ are primitive and distinct and that $n>0$ and $n'>0$, the vectors $(n,r,s)$ and $(n',r',s')$ are not collinear. Hence, the cross product of $(n,r,s)$ and $(n',r',s')$; i.e.
$$
(A,B,C):=(n,r,s)\times(n',r',s')\,,
$$
is non-zero integer vector. By \eqref{zero}, we have that
$$
B=-ns'+n's=0.
$$
Therefore,
\begin{equation}\label{A+C}
|A|+|C|>0\,.
\end{equation}
Without loss of generality, we can assume that $0<\alpha<1$. Then $0\le r\le n$ and $0\le r'\le n'$. Further, observe that
\begin{equation}\label{vb01}
C\alpha=(nr'-n'r)\alpha =
(n\alpha-r) r'-(n'\alpha-r') r\,,
\end{equation}
from which it follows that
\begin{equation}\label{|C|}
|C| \ \le  \ \frac1\alpha \Big(R^{-k+t'}+R^{-k'+t}\Big)\,.
\end{equation}
Similarly, since $n's-ns'=0$, we obtain that
$$
A=rs'-r's=(r-n\alpha)s'-(r'-n'\alpha)s\,.
$$
As before, since $0<\beta<1$, we obtain that
\begin{equation}\label{|A|}
|A|  \ \le  \  R^{-k+t'}+R^{-k'+t}\,.
\end{equation}
Assuming that $\eta<1/9$, by \eqref{cc102}, \eqref{cc100} and the inequalities $\lceil \nu_1 t\rceil\le k < \lfloor \nu_2 t\rfloor$, we obtain that
\begin{equation}\label{vb1002}
\begin{array}{l}
|n\alpha-r|< R^{-k}\le R^{-\nu_1t}\,,\\[2ex]
|n\beta-s|<R^{-t+k}\le R^{-(1-\nu_2)t}\le R^{-\nu_1t}\,.
\end{array}
\end{equation}
Observe that
\begin{align*}
(1,\alpha,\beta)\times(n,r,s)&=
\left(
\left|\begin{array}{cc}\alpha & \beta \\ r & s \end{array}\right|,
-\left|\begin{array}{cc}1 & \beta \\ n & s \end{array}\right|,
\left|\begin{array}{cc}1 & \alpha \\ n & r \end{array}\right|
\right)\\[2ex]
&=\big(-\alpha(n\beta-s)+\beta(n\alpha-r),n\beta-s,-(n\alpha-r)\big)\,.
\end{align*}
Hence, on using \eqref{vb1002} and the fact that $0\le\alpha,\beta\le 1$, we obtain that
\begin{equation}\label{vb1003}
|(1,\alpha,\beta)\times(n,r,s)|\le \sqrt{6}\,R^{-\nu_1t}\,.
\end{equation}
A similar argument shows that
\begin{equation}\label{vb1004}
|(1,\alpha,\beta)\times(n',r',s')|\le \sqrt{6}\,R^{-\nu_1t'}\,.
\end{equation}
In particular, since $n\ge R^{t-1}$, \eqref{vb1003} implies that the (acute) angle $\theta$ between $(1,\alpha,\beta)$ and $(n,r,s)$ satisfies
\begin{align*}
|\sin\theta| &~=~ \frac{|(1,\alpha,\beta)\times(n,r,s)|}{|(1,\alpha,\beta)|\cdot|(n,r,s)|}\\[1ex]
& ~\le ~ \frac{\sqrt{6}\,R^{-\nu_1t}}{n}\\[1ex]
&~\le~ \sqrt{6}R\,R^{-(1+\nu_1)t}\,.
\end{align*}
Similarly, since $n'\ge R^{t'-1}$, \eqref{vb1004} implies that the (acute) angle $\theta'$ between $(1,\alpha,\beta)$ and $(n',r',s')$ satisfies
$$
|\sin\theta'|\le \sqrt{6}R\,R^{-(1+\nu_1)t'}\,.
$$
Let $\varrho$ be the angle between $(n,r,s)$ and $(n',r',s')$. Then, by the triangle inequality for the projective distance (see, for example, \cite[\S3]{Ber12}),
\begin{align*}
|\sin\varrho| &~\le~|\sin\theta|+|\sin\theta'|\\[1ex]
 &~\le~ 2\max\{|\sin\theta|,|\sin\theta'|\}\,.
\end{align*}
On the other hand, the angle $\tilde\theta$ between $(1,\alpha,\beta)$ and the vector subspace of $\R^3$ spanned by $(n,r,s)$ and $(n',r',s')$ is at most $\min\{\theta,\theta'\}$ and thus satisfies the inequality
$$
|\sin\tilde\theta|\le\min\{|\sin\theta|,|\sin\theta'|\}\,.
$$
Hence the volume of the parallelepiped generated by $(1,\alpha,\beta)$, $(n,r,s)$ and $(n',r',s')$ is
\begin{align}
\nonumber\underbrace{|(1,\alpha,\beta)|}_{\le\sqrt3}\cdot\underbrace{|(n,r,s)|}_{\le\sqrt3\,n}&\cdot \underbrace{|(n',r',s')|}_{\le\sqrt3\,n'}\cdot|\sin\varrho|\cdot|\sin\tilde\theta|\\[2ex]
\nonumber&\le 6\sqrt{3}\,n\,n'|\sin\theta|\cdot|\sin\theta'|\\[2ex]
\nonumber&\le
36\sqrt3R^2\,R^tR^{t'}R^{-(1+\nu_1)t}\,R^{-(1+\nu_1)t'}\\[2ex]
&\le
36\sqrt3R^2\,R^{-\nu_1t}\,R^{-\nu_1t'}\,.\label{12.44}
\end{align}
We also have that this volume is equal to
$$
\big|(1,\alpha,\beta)\cdot\big((n,r,s)\times(n',r',s')\big)\big|=|(1,\alpha,\beta)\cdot(A,B,C)|=|C\beta+A|\,.
$$
Combining this with \eqref{12.44}, implies that
\begin{equation}\label{cc103}
|C\beta+A|\le 36\sqrt3R^2\,R^{-\nu_1t}\,R^{-\nu_1t'}\,.
\end{equation}
In particular, this together with \eqref{A+C} and the facts that $A\in\Z$ and that $t$ can be taken arbitrarily large, implies that $C\neq0$.
Combining \eqref{|C|}, \eqref{|A|} and \eqref{cc103}, we have proved the following statement.

\begin{lemma}\label{lemma21}
Let $\Upsilon=R^{-k+t'}+R^{-k'+t}$ and $\Theta=36\sqrt3R^2\,R^{-\nu_1t}\,R^{-\nu_1t'}$. Then \eqref{rtrtr} is a subset of the following
$$
\bigcup_{C=1}^{\lfloor \Upsilon/\alpha\rfloor}\{\beta\in[0,1]:\|C\beta\|<\Theta\}\,.
$$
\end{lemma}

Recall that if $\Theta<1/2$, then
$$
|\{\beta\in[0,1]:\|C\beta\|<\Theta\}|=2\Theta\,.
$$
Without loss of generality, the condition $\Theta<1/2$  can be assumed  since $t$ and $t'$ can be taken sufficiently large.
Note that
\begin{equation}\label{cc200}
t+t' \ \le \ 2\max\{t,t'\}~\stackrel{\eqref{vb1005}}{\le} ~ \frac{10}{5-6\eta}\min\{t,t'\}\,.
\end{equation}
Then, using Lemma~\ref{lemma21} and the inequalities $\lceil \nu_1 t\rceil\le k < \lfloor \nu_2 t\rfloor$, we obtain that the measure of the set \eqref{rtrtr} is
\begin{align*}
&\le \ \frac{72\sqrt3R^2}{\alpha}\,R^{-\nu_1t}\,R^{-\nu_1t'}\cdot\big(R^{-k+t'}+R^{-k'+t}\big)\\[1ex]
&\le \ \frac{72\sqrt3R^2}{\alpha}\,\cdot\Big(R^{-2\nu_1t+(1-\nu_1)t'}+R^{-2\nu_1t'+(1-\nu_1)t}\Big)\\[1ex]
&\stackrel{\eqref{vb1005}}{\le} ~\frac{72\sqrt3R^2}{\alpha}\,\cdot \Big(R^{-2\nu_1t+(1-\nu_1)\left(1+\tfrac{6\eta}{5-6\eta}\right)t}+R^{-2\nu_1t'+(1-\nu_1)\left(1+\tfrac{6\eta}{5-6\eta}\right)t'}\Big)\\[1ex]
&=  \ ~\frac{72\sqrt3R^2}{\alpha}\,\cdot \Big(R^{-\tfrac{11\eta-12\eta^2}{5-6\eta}t}+R^{-\tfrac{11\eta-12\eta^2}{5-6\eta}t'}\Big)\\[1ex]
&\le \  ~\frac{144\sqrt3R^2}{\alpha}\,\cdot R^{-\tfrac{11\eta-12\eta^2}{5-6\eta}\min\{t,t'\}}\\[1ex]
&\stackrel{\eqref{cc200}}{\le}~\frac{144\sqrt3R^2}{\alpha}\,\cdot R^{-\tfrac{11\eta-12\eta^2}{10}(t+t')} \ \le \  \frac{144\sqrt3R^2}{\alpha}\,\cdot R^{-\eta(t+t')}
\end{align*}
provided that $\eta<1/12$. Thus,
\begin{equation}\label{estimate3}
\Big|\bigcup_{n's-ns'=0}E_{t,k}(s/n)\cap E_{t',k'}(s'/n')\Big| ~\le~ \frac{144\sqrt3R^2}{\alpha}\,\cdot R^{-\eta(t+t')}\,,
\end{equation}
where, recall, the union is taken over $(n,r,s)\in Z(t,k)$ and $(n',r',s')\in Z(t',k')$ subject to condition  \eqref{zero} and assuming that \eqref{cond1} holds.

\subsection{The finale}

Estimates \eqref{estimate2} and \eqref{estimate3} combined together show that whenever conditions \eqref{cond1} are satisfied, we have that for $(t,k)\neq(t',k')$
\begin{equation}\label{cc300}
|E_{t,k}\cap E_{t',k'}| ~\ll~ \psi(R^t)R^{t}\,\psi(R^{t'})R^{t'}
~+~R^{-\eta(t+t')}\,.
\end{equation}
This estimate also holds when conditions \eqref{cond1} are not satisfied, this time as a consequence of \eqref{est1}. Thus, \eqref{cc300} together with \eqref{ddd1} implies that
\begin{align*}
\sum^*_{t\le T}&\sum^*_{t'\le T'}|E_{t,k}\cap E_{t',k'}| ~ \\
&\ll~ \sum^*_{t\le T}\sum^*_{t'\le T'}\psi(R^t)R^{t}\,\psi(R^{t'})R^{t'}+
\sum^*_{t\le T}\sum^*_{t'\le T'}R^{-\eta(t+t')}+S_\cT(T)
\\[2ex]
&=~ \left(\,\sum_{t\in\cT,\,t\le T}\ \sum_{\lceil \nu_1 t\rceil\le k< \lfloor\nu_2 t\rfloor}\psi(R^t)R^t\right)^2\\[2ex]
&\qquad \qquad \qquad +\left(\,\sum_{t\in\cT,\,t\le T}\ \sum_{\lceil \nu_1 t\rceil\le k< \lfloor\nu_2 t\rfloor}R^{-\eta t}\right)^2+S_\cT(T)\\[2ex]
&\le~\left(\,\sum_{t\in\cT,\,t\le T}\ t\psi(R^t)R^t\right)^2
+\left(\,\sum_{t\in\cT,\,t\le T}\ tR^{-\eta t}\right)^2+S_\cT(T)\\[2ex]
&\le~S_\cT(T)^2
+\left(\,\sum_{t=1}^\infty tR^{-\eta t}\right)^2+S_\cT(T)~\ll~S_\cT(T)^2
\end{align*}
for sufficiently large $T$, since $S_\cT(T)\to\infty$ as $T\to\infty$. This establishes \eqref{dad} and thereby completes the proof of Theorem~\ref{T11}.

\vspace{7ex}

\noindent{\bf Acknowledgements.}  From establishing the first results to getting to this stage,
this paper has taken  over five years to complete.  It has been a long
but nevertheless interesting   journey which has given rise to new and
(in our collective judgments) fundamental   problems, especially within  the classical
theory of inhomogeneous Diophantine approximation such as Conjecture
\ref{rub} and Problem \ref{crazy}. One consequence of the five year time
span is that we have missed Wolfgang Schmidt's eightieth birthday by two
years (almost certainly more once the paper comes out in print).
Hopefully it has been worth it: over the extra two years the
paper, we, and of course Wolfgang, have nicely matured!  The mathematical and personal influence of
Wolfgang  on the three of us has been immense, and we will forever be
in his debt.   His remarkable creativity and vision are still influencing
current research; in particular at York. Thank you!

The authors are grateful to the anonymous reviewer whose comments have been most valuable, especially for drawing our attention to various classical results missing from an earlier version of the paper.

SV would like to thank his wonderful parents, Lalji and Manchaben, for
graciously accepting my life, my love for  Bridget  and also for being fantastic
grandparents to the dynamic duo Ayesha and Iona.  All of these  are a
million miles from their norm. In short, they have been and remain to be
my greatest teachers.

\vspace{0ex}

{\footnotesize
\noindent VB, SV\,:\\
Department of Mathematics, University of York,\\
Heslington, York, YO10 5DD, England\\
e-mails: {\tt victor.beresnevich@york.ac.uk}\\
\hspace*{16ex}{\tt sanju.velani@york.ac.uk}

\noindent AH\,:\\
Department of Mathematics, University of Houston,\\ 
Texas, USA\\
e-mail: {\tt haynes@math.uh.edu}

}

\end{document}